\documentclass[10pt]{amsart}
\usepackage{amsxtra,amscd,amsthm,amsmath,amssymb}
\usepackage{longtable} 
\usepackage{enumerate}
\usepackage{verbatim}
\usepackage{graphicx}
\usepackage{xspace}
\usepackage{ifthen}
\usepackage{tabularx}
\usepackage[all,cmtip,line]{xy}

\usepackage{color}
\usepackage{mathrsfs}

\usepackage{amsthm}

\theoremstyle{plain}
\numberwithin{equation}{subsection}
\newtheorem{thm}[equation]{Theorem}
\newtheorem{ithm}[subsection]{Theorem}

\newtheorem{icor}[subsection]{Corollary}
\newtheorem{prop}[equation]{Proposition}
\newtheorem{propconstr}[equation]{Proposition/Construction}
\newtheorem{lemma}[equation]{Lemma}
\newtheorem{Definition}[equation]{Definition}

\newtheorem{Remark}[equation]{Remark}

\newtheorem{cor}[equation]{Corollary}

\theoremstyle{remark}
\newtheorem{para}[equation]{\bf}
\theoremstyle{plain}
\renewcommand{\subsubsection}{\addtocounter{equation}{1}{\vskip 6pt \bf\arabic{section}.\arabic{subsection}.\arabic{equation}}}
\theoremstyle{definition}

\newcommand{\quash}[1]{}  
\newcommand{\nc}{\newcommand}
\nc{\on}{\operatorname}

\newcommand{\lps}{[\![}
\newcommand{\rps}{]\!]}

\renewcommand{\phi}{\varphi}

\newcommand{\fraka}{{\mathfrak a}}

\newcommand{\fraki}{{\mathfrak i}}

\newcommand{\frakm}{{\mathfrak m}}

\newcommand{\eK}{{\sf K}}
\newcommand{\eE}{{\sf E}}

\newcommand{\eG}{{G}}

\newcommand{\frakM}{{\mathfrak M}}

\newcommand{\RR}{{\mathbb R}}
\newcommand{\Sh}{{\rm Sh}}

\newcommand{\SSh}{{\mathscr S}}

\newcommand{\frakS}{{\mathfrak S}}

\newcommand{\frakT}{{\mathfrak T}}

\newcommand{\frakX}{{\mathfrak X}}

\newcommand{\bbA}{{\mathbb A}}

\newcommand{\DD}{{\mathbb D}}

\newcommand{\calA}{{\mathcal A}}

\newcommand{\calD}{{\mathcal D}}

\newcommand{\calF}{{\mathcal F}}
\newcommand{\calG}{{\mathcal G}}

\newcommand{\calM}{{\mathcal M}}

\newcommand{\calP}{{\mathcal P}}
\newcommand{\calQ}{{\mathcal Q}}

\newcommand{\calT}{{\mathcal T}}

\newcommand{\calX}{{\mathcal X}}

\nc{\al}{{\alpha}} \nc{\be}{{\beta}}

\nc{\ve}{{\varepsilon}} \nc{\Ga}{{\Gamma}}

\nc{\La}{{\Lambda}}

\def\0{\circ}

\newcommand{\cal}{\mathcal}

\renewcommand{\AA}{{\mathbb A}}

\newcommand{\GG}{{\mathbb G}}

\renewcommand{\inf}{{\rm inf}}

\newcommand{\C}{{\mathbb C}}
\newcommand{\R}{{\mathbb R}}
\newcommand{\Q}{{\mathbb Q}}

\newcommand{\iso}{\cong}

\newcommand{\et}{{\text{\rm \'et}}}

\newcommand{\X}{{\mathcal X}}
\newcommand{\Y}{{\mathcal Y}}

\newcommand{\Gg}{{\mathcal G}}

\newcommand{\Gm}{{{\mathbb G}_{\rm m}}}
\newcommand{\Z}{{\mathbb Z}}
\newcommand{\F}{{\mathcal F}}
\newcommand{\ti}{\tilde}
\newcommand{\Spec}{{\rm Spec \, } }
\newcommand{\Spf}{{\rm Spf } }

 \renewcommand{\O}{{\mathcal O}}
 
\renewcommand{\SS}{{\mathscr S}}

\newcommand{\Dd}{{\mathscr D}}

\newcommand{\GL}{{\rm GL}}

\newcommand{\M}{{\mathcal M}}
\renewcommand{\L}{{\mathcal L}}

\newcommand{\Res}{{\rm Res}}

\newcommand{\cris}{{\rm cris}}

\newcommand{\Mloc}{{\rm M}^{\rm loc}}

\newcommand{\rH}{{\mathrm H}}

\def\thfill{\null\nobreak\hfill}

\def\endproof{\thfill\vbox{\hrule
  \hbox{\vrule\hbox to 5pt{\vbox to 5pt{\vfil}\hfil}\vrule}\hrule}}

\renewcommand{\P}{{\cal P}}

\newcommand{\Gal}{{\rm Gal}}

\newcommand{\GSp}{{\rm GSp}}

\newcommand{\OEv}{{\mathcal O}_{{\sf E}, (v)}}
\newcommand{\Lr}{{\mathcal L}}
\newcommand{\Ls}{{\mathcal L}}

\newcommand{\sG}{{\mathscr G}}

\newcommand{\sD}{{\mathscr D}}

  \newcommand{\sbk}{s_{a, \scriptscriptstyle{\rm BK}}}
  \newcommand{\sad}{s_{a, \scriptscriptstyle{\rm D}}}

\newcommand{\xU}{{\ensuremath{{\tilde M}_1}}}
\newcommand{\xUo}{{\ensuremath{{\tilde M}_{1,0}}}}

 \usepackage{relsize} 
\usepackage[bbgreekl]{mathbbol} 
\usepackage{amsfonts} 
\DeclareSymbolFontAlphabet{\mathbb}{AMSb} 
\DeclareSymbolFontAlphabet{\mathbbl}{bbold} 
\newcommand{\Prism}{{\mathlarger{\mathbbl{\Delta}}}}

\begin{document}

\title[ ]{On integral models of Shimura varieties}

\author[G. Pappas]{G. Pappas}
\thanks{Partially supported by  NSF grants  DMS-1701619, DMS-2100743, and the Bell Companies Fellowship Fund through the Institute for Advanced Study}

\address{Dept. of
Mathematics\\
Michigan State
Univ.\\
E. Lansing\\
MI 48824\\
USA} \email{pappasg@msu.edu}

 \begin{abstract}
 We show how to characterize integral models of Shimura varieties over places  
of the reflex field where the level subgroup is parahoric by formulating a definition of a ``canonical" integral model. We then prove that in Hodge type cases and under a tameness hypothesis, the integral models constructed by the author and Kisin
in previous work are canonical and,
 in particular, independent of choices. A main tool is a theory of displays with parahoric structure that we develop in this paper.  
  \end{abstract}
\date{\today}

\maketitle


\tableofcontents


 \def\ve{\medskip}

 \section{Introduction} 
 
\subsection{} In this paper, we show how to uniquely characterize integral models of Shimura varieties over  some primes where non-smooth reduction is expected. More specifically, we consider integral models over primes $p$ at which 
the level subgroup is parahoric. Then, under some further assumptions, we provide a notion of a ``canonical" integral model. 

At such primes, the Shimura varieties have integral models with complicated singularities (\cite{R}, \cite{R-Z}). This happens even for the most commonly used Shimura varieties with level structure, such as Siegel varieties, and it foils attempts to characterize the models by simple conditions. The main observation of this paper is that we can characterize these integral models by requiring that they support  suitable ``$\Gg$-displays", i.e. filtered Frobenius modules with $\Gg$-structure, where $\Gg$ is the smooth integral $p$-adic group scheme which corresponds to the level subgroup. We then prove that these modules exist in most Hodge type cases treated by the author and Kisin in \cite{KP}. As a corollary, we show that these integral models of Shimura varieties with parahoric level structure, are independent of the choices made in their construction. 

Let us first recall the story over ``good" places, i.e. over primes at which the level subgroup is hyperspecial. One expects that there is an integral model with smooth reduction at such primes. This expectation was first spelled out by Langlands in the 80's. Later, it was pointed out by Milne \cite{Milne} that one can uniquely characterize smooth integral models over the localization of the reflex field at such places by requiring that they satisfy a Neron-type extension property. Milne calls smooth integral models with this property ``canonical''. The natural integral models of Siegel Shimura varieties, at good primes, are  smooth and satisfy the extension property. Therefore, they are canonical. In this case, the extension property follows by the Neron-Ogg-Shafarevich criterion and a purity result of Vasiu and Zink \cite{VasiuZink} about extending abelian schemes over codimension $\geq 2$ subschemes of smooth schemes. This argument extends to the very general class of Shimura varieties of abelian type at good primes, provided we can show there is  a smooth integral model which is, roughly speaking, constructed using moduli of abelian varieties. This existence of such a canonical smooth integral model for Shimura varieties of abelian type at places over good primes was shown by Kisin (\cite{KisinJAMS}, see also earlier work of Vasiu \cite{VasiuAsian}). 

The problem becomes considerably harder over other primes. Here, we are considering primes $p$ at which  the level subgroup is parahoric. For the most part, we also require that the reductive group splits over a tamely ramified extension,
although our formulation is more general. Under these assumptions, models for Shimura varieties of abelian type, integral at places over such $p$, were constructed by Kisin and the author \cite{KP}.
This follows work of Rapoport and Zink \cite{R-Z}, of Rapoport and the author, and of many others, see \cite{Pappasicm}. The construction in 
\cite{KP} uses certain simpler schemes, the ``local models" that depend only on the local Shimura data. Then, integral models for Shimura varieties of Hodge type are given by taking  the normalization of the Zariski closure of a well-chosen embedding of the Shimura variety in a Siegel moduli scheme over the integers. More generally, models of Shimura varieties of abelian type are obtained from those of Hodge type by a quotient construction that uses Deligne's theory of connected Shimura varieties. All these integral models of Shimura varieties have the same \'etale local structure as the corresponding local models. However, the problem of characterizing them globally or showing that they are independent of choices was not addressed in loc. cit.\footnote{with the exception of the very restricted result \cite[Prop. 4.6.28]{KP}.} Here, we give a broader notion of ``canonical" integral model and solve these problems when the varieties are of Hodge type.
Such a characterization was not known before, not even for general PEL type Shimura varieties.

\subsection{}  Let us now explain these results more carefully. 

Let $(G, X)$ be a Shimura datum \cite{DeligneCorvallis} with corresponding conjugacy class  of minuscule cocharacters  $\{\mu\}$ and reflex field $\eE$. To fix ideas, we will always assume that the center $Z(G)$ of $G$ 
has the same $\Q$-split rank  as   $\R$-split rank. (This condition holds for Shimura data of Hodge type.)
For an open compact subgroup $\eK\subset G(\bbA^f)$ of the finite adelic points of $G$, the Shimura variety
\[
{\rm Sh}_{\eK}(G, X)=G(\Q)\backslash (X\times G(\bbA_f)/\eK)
\]
has a canonical model over $\eE$. 

Fix an \emph{odd} prime $p$. Suppose $\eK=\eK_p\eK^p$, with $\eK_p\subset G(\Q_p)$ and
$\eK^p\subset G(\bbA^p_f)$, both compact open, with $\eK^p$ sufficiently small. 
Denote by $\Ls_\eK$ the pro-\'etale $\Gg(\Z_p)$-cover
\[
{\Ls}_\eK:=\varprojlim_{\eK'_p\subset \eK_p}\Sh_{\eK'_p\eK^p}(G, X) \to  \Sh_{\eK}(G, X) ,
\] 
with $\eK'_p$ running over all compact open subgroups of $\eK_p$.

Assume that:

 (a) The level $\eK_p$ is a \emph{parahoric} subgroup in the sense of Bruhat-Tits \cite{T}, i.e. $\eK_p$ is the neutral component of the stabilizer of a point in the (enlarged) building of $G(\Q_p)$. Then $\eK_p=\Gg(\Z_p)$, where $\Gg$ is the corresponding parahoric smooth connected affine group scheme over $\Z_p$ with $\Gg\otimes_{\Z_p}\Q_p=G_{\Q_p}$ (\cite{T}). 
 
 We will occasionally assume the slightly stronger:
 
 (a$+$) The level $\eK_p$ is a \emph{connected parahoric} subgroup. By definition, this means that $\eK_p$ is parahoric and is the stabilizer of a point in the enlarged building, i.e. we can choose the point such that the stabilizer is actually connected.  
 (Such a subgroup is sometimes also called a \emph{parahoric stabilizer}.)

Now choose a place $v$ of $\eE$ over $p$. Let $\O_{\eE, (v)}$ be the localization of the ring of integers $\O_{\eE}$ at $v$.
Denote by $E$ the completion of $\eE$ at $v$, by $\O_E$ the integers of $E$ and fix an algebraic closure $k$ of the residue field $k_E$ of $E$. We can also consider $\{\mu\}$ as a conjugacy class of cocharacters which is defined over $E$.
Under some mild assumptions (see \S \ref{LMtheory}), we have the local model
\[
\Mloc=\Mloc(\Gg, \{\mu\}),
\]
as characterized by \cite[Conj. 21.4.1]{Schber}. This is a flat and projective $\O_E$-scheme with $\Gg$-action. Its generic fiber
is $G_E$-equivariantly  isomorphic to the variety $X_\mu$ of parabolic subgroups of $G_E$ of type $\mu$,
and its special fiber is reduced. 

We will assume:

(b)  There is a closed group scheme
immersion $\iota: \Gg\hookrightarrow \GL_n$ over $\Z_p$ such that $\iota(\mu)$ is conjugate to one of the standard minuscule cocharacters of $\GL_n$,  $\iota(\Gg)$ contains the scalars, 
and the map $\iota$ gives an equivariant closed immersion
\[
\iota_*: \Mloc\hookrightarrow {\rm Gr}(d, n)_{\O_E}
\]
in a Grassmannian, where $d$ is determined by $\iota(\mu)$.

 Under the assumptions (a) and (b), we define ``$(\Gg, \{\mu\})$-displays" which are group-theoretic generalizations of Zink's displays (\cite{Zink}, \cite{ZinkCFT}). This is the main invention in the paper, see below. We think it has some independent interest.

We now ask for  $\O_{\eE, (v)}$-models $\SS_{\eK}=\SS_{\eK_p\eK^p}$ (separated schemes of finite type and flat over $\O_{\eE, (v)}$) of the Shimura variety $\Sh_{\eK}(G, X)$ which are normal. In addition, we require:
 
 \begin{itemize}
 
 \item[(1)] For $\eK'^p\subset \eK^p$, there are finite \'etale morphisms 
 \[
 \pi_{\eK'_p, \eK_p}: \SS_{\eK_p\eK'^p}\to \SS_{\eK_p\eK'^p}
 \]
  which extend the natural morphism $\Sh_{\eK_p\eK'^p}(G, X)\to \Sh_{\eK_p\eK^p}(G, X)$. 

 \item[(2)] The scheme $\SS_{\eK_p}=\varprojlim_{\eK^p}\SS_{\eK_p\eK^p}$ satisfies the ``extension property" for dvrs of mixed characteristic $(0, p)$, i.e. for any such dvr $R$
 \[
 \SS_{\eK_p}(R[1/p])=\SS_{\eK_p}(R).
 \]
 
 \item[(3)]  The $p$-adic formal schemes $\widehat \SS_{\eK}=\varprojlim_n \SS_\eK\otimes_{\O_{\eE, (v)}}\O_{\eE, (v)}/(p)^n$ support    locally universal {$(\Gg, \{\mu\})$-displays} 
$
 \Dd_\eK
$
 which are associated with $\Ls_{\eK}$. We  ask that these are compatible for varying $\eK^p$, i.e. that there are  compatible isomorphisms
 \[
 \pi_{\eK'_p, \eK_p}^*\Dd_{\eK }\simeq \Dd_{\eK'}
 \]
over the system of morphisms $\pi_{\eK'_p, \eK_p}$ of (1).
\end{itemize}
 \smallskip

We will explain below the rest of the terms in (3) including the meaning of having a $(\Gg, \{\mu\})$-display being ``associated" with the pro-\'etale $\Gg(\Z_p)$-cover $\Ls_{\eK}$. Our first main result is (always under the standing hypothesis on the center $Z(G)$, also $p\neq 2$):

 \begin{ithm}\label{introThm1}
  Assume $(G, X, v, \eK_p)$ satisfy (a) and (b) above. Then there is at most one (up to unique isomorphism), pro-system of normal $\O_{\eE, (v)}$-models $\SS_{\eK}=\SS_{\eK_p\eK^p}$ of the Shimura variety $\Sh_{\eK}(G, X)$ which satisfy (1), (2) and (3) above.
 \end{ithm}
 
 In fact, we prove a slightly more general result.  (See Theorem \ref{charThmA} and  Corollary \ref{charThm}.)

We call integral models $\SS_{\eK}$ which satisfy the above, \emph{canonical}. 

Assume now that the global Shimura datum $(G, X)$ is of Hodge type and that $G$ splits over a tamely ramified extension of $\Q_p$.  Then, under the  assumption (a+), assumption (b) is also satisfied. In this situation, ``nice" integral models $\SS_{\eK}(G, X)$ of the Shimura variety $\Sh_{\eK}(G, X)$ have been constructed in \cite[\S 4]{KP}, see \cite[Theorem 4.2.7]{KP}. These integral models  depend, a priori, on the choice of a suitable Hodge embedding. Our second main result is:
 
  \begin{ithm}\label{introThm2} (Theorem \ref{thmLast})
 Assume $(G, X)$ is of Hodge type, $G$ splits over a tamely ramified extension of $\Q_p$, and $ \eK_p$ is  connected parahoric.  The integral models $\SS_{\eK}(G, X)$ of \cite[Theorem 4.2.7]{KP} satisfy (1), (2), and  (3).
 \end{ithm}
 
 Since the integral models $\SS_{\eK}(G, X)$ of \cite[Theorem 4.2.7]{KP} are normal,  by combining the two results, we obtain:
 
 \begin{icor} (Theorem \ref{indThm})
 Assume $(G, X)$ is of Hodge type, $G$ splits over a tamely ramified extension of $\Q_p$, and $ \eK_p$ is  connected parahoric. The integral models $\SS_{\eK}(G, X)$ of \cite[Theorem 4.2.7]{KP} are, up to unique isomorphism, independent of the choices in their construction.  
 \end{icor}

\subsection{}   We now explain the terms that appear in condition (3). For more details, the reader is referred to the main body of the paper.

Suppose $R$ is a
  $p$-adic  flat  $\O_E$-algebra. 
Denote by $W(R)$ the ring of ($p$-typical) Witt vectors with entries from $R$
and by $\phi:  W(R)\to  W(R)$ the Frobenius endomorphism.  

A $(\Gg, \Mloc)$-display
$\calD=(\calP, q, \Psi)$ over $R$ consists of a $\Gg$-torsor $\calP$ over
$\Spec(W(R))$, a $\Gg$-equivariant morphism 
\[
q:\calP\times_{\Spec(W(R))}\Spec(R)\to \Mloc,
\]
and a $\Gg$-isomorphism $\Psi: \calQ\xrightarrow{\sim} \calP $. Here, $\calQ$ is a $\Gg$-torsor over $\Spec(W(R))$ which is the \emph{modification of $\phi^*\calP$ along $p=0$, given by $q$} (see    Proposition/Construction \ref{MoPro}). 

By its construction,
$\calQ$ comes together with an isomorphism of $G$-torsors 
\[
 \phi^*\calP_{|\Spec(W(R)[1/p])}\xrightarrow{\sim} \calQ_{|\Spec(W(R)[1/p])}
\]
over $\Spec(W(R)[1/p])$.  Composing this isomorphism with $\Psi_{|\Spec(W(R)[1/p])}$ gives
\[
 \Phi: \phi^*\calP_{|\Spec(W(R)[1/p])}\xrightarrow{\sim} \calP_{|\Spec(W(R)[1/p])}.
\]
This is the ``Frobenius" of the $(\Gg, \Mloc)$-display.

Since $\Mloc$ is determined by $(\Gg, \{\mu\})$ we often just say ``$(\Gg, \{\mu\})$-display".  The definition of a $(\Gg, \{\mu\})$-display resembles 
 that of a shtuka and of its mixed characteristic variants (\cite{Schber}). 
 (In the shtuka lingo, one would say that ``there is one leg 
along $p=0$, bounded by $\mu$".) It is also a generalization of the concept of $(\Gg,\{\mu\})$-display due to the author and B\"ultel \cite{BP}. (This required the restrictive assumption that 
 $\Gg$ is reductive over $\Z_p$; the extension to the parahoric case is far from obvious). When $\Gg=\GL_n$ and 
 $\Mloc={\rm Gr}(d, n)$, it amounts to a display of height $n$ and dimension $n-d$ in the sense of Zink
 (see \S\ref{compaZink}). By using Zink's Witt vector descent, we obtain
 a straightforward extension of this definition from ${\rm Spf}(R)$ to non-affine $p$-adic formal schemes like $\widehat \SS_{\eK}$.
 
 If $R$ is, in addition, a Noetherian complete local ring with perfect residue field and $p\neq 2$, there is a similar notion of a {\sl Dieudonn\'e} 
 $(\Gg, \{\mu\})$-display over $R$ for which all the above objects $ \calP$, $q$, $\Psi$,
 are defined over Zink's variant $\hat W(R)$ of the Witt ring. After applying a local Hodge embedding $\iota:\Gg\hookrightarrow \GL_n$, a Dieudonn\'e $(\Gg, \{\mu\})$-display $\calD$ induces a classical Dieudonn\'e display over $R$. By Zink's theory \cite{Zink}, \cite{ZinkCFT}, this gives a $p$-divisible group   over $R$.

 In (3), we ask that $\Dd_\eK=(\calP_\eK, q_\eK, \Psi_\eK)$ is ``associated" with ${\Ls}_{\eK}$. This definition  (Definition \ref{defass2}) is modeled on  the notion of ``associated" used by Faltings, e.g. \cite{Fa1}. It  requires that the $\Z_p$-\'etale local system over the generic fiber
 given by the Tate module of  the $p$-divisible group obtained from $\Dd_\eK$ and $\iota$ as above, 
 agrees with the local system given by ${\Ls}_{\eK}$
 and $\iota$. It also requires that, after specializing at each point with values in a mixed characteristic dvr, the \'etale-crystalline comparison isomorphism sends the \'etale (i.e. Galois invariant) tensors defining ${\Ls}_{\eK}$ to the crystalline tensors defining $\calP_\eK$.

 Finally, we explain the term ``locally universal" in (3): 
 
 For $\bar x\in \SS_\eK(k)$, denote by $\hat R_{\bar x}$ the completion of the strict
 Henselization of the local ring of $\SS_\eK$ at $\bar x$. Set $\calD_{\eK, \bar x}$ 
 for the $(\Gg, \{\mu\})$-display over $\hat R_{\bar x}$ obtained from $\Dd_\eK$ 
 by base change.  
 
 We say that the $(\Gg, \{\mu\})$-display $\Dd_\eK$ is locally universal\footnote{A more correct, but also more cumbersome, term would probably be ``locally formally universal".} if for all $\bar x\in \SS_\eK(k)$, there is a ``rigid" section $s_{\bar x}$ of the $\Gg$-torsor $\calP_{\eK,\bar x}$
over $\Spec( W(\hat R_{\bar x}))$, so that the  composition 
 \[
 q\cdot (s_{\bar x}\times 1): \Spec(\hat R_{\bar x})\to \calP_{\eK,\bar x}\times_{\Spec( W(\hat R_{\bar x}))}\Spec(\hat R_{\bar x})\to \Mloc
 \]
 gives an isomorphism between $\hat R_{\bar x}$ and the corresponding completed local ring of the local model $\Mloc$. Therefore, condition (3) fixes the ``singularity type" of $\SS_{\eK}$ at $\bar x$. In the above, rigid is for a version of the Gauss-Manin connection in our set-up.
 
  Note that this somewhat unusual definition is justified by the fact that  $(\Gg, \{\mu\})$-displays
 are defined only over $\Z_p$-flat bases, so these objects do not have deformation theory
 in the usual sense.

\subsection{}  Let us discuss the proofs of the two main theorems. 

We use an intermediate technical notion, the ``associated system" $(\Ls_\eK, \{\calD_{\eK,\bar x}\}_{\bar x\in \SS_\eK(k)})$ (Definition \ref{defass}), in which $ \calD_{\eK, \bar x}$ are Dieudonn\'e $(\Gg, \{\mu\} )$-displays over the completions of the strict Henselizations $\hat R_{\bar x}$, as before.
A  $(\Gg, \{\mu\})$-display $\Dd_\eK$ which is associated with $\Ls_\eK$, gives such an associated system by base change, as above.  In fact, most of our constructions just use $(\Ls_\eK, \{\calD_{\eK,\bar x}\}_{\bar x\in \SS_\eK(k)})$. Using Tate's theorem on extending homomorphisms of $p$-divisible groups, we show that 
$(\Ls_\eK, \{\calD_{\eK,\bar x}\}_{\bar x\in \SS_\eK(k)})$ is uniquely determined by $\Ls_\eK$ and that
the existence of a locally universal associated system is enough to characterize the 
integral model $\SS_\eK$ uniquely. This leads to the proof of Theorem \ref{introThm1}. 
 We also show that when $\Ls_{\eK}$ comes from the Tate module of a $p$-divisible group 
 with Galois invariant tensors, it can be completed to a (unique up to isomorphism) associated system.
This employs comparison isomorphisms of integral $p$-adic Hodge theory that use work of Kisin, Scholze and others \cite{Schber}, \cite{BMS}, and of Faltings \cite{Fa}, \cite{Fa1}.

To show Theorem \ref{introThm2}, we first show that the integral models given in \cite{KP} support a locally universal associated system. The issue is to show that the unique associated system obtained as above is locally universal.
(Recall that  this imposes that the singularities 
of the integral model $\SS_\eK$ agree with those of the local model $\Mloc$.) This is done, by reexaming the proof of the main result of \cite{KP}. Then, as in work of Hamacher and Kim \cite{HamaKim}, we can also show that these models support a (global) locally universal $(\Gg, \Mloc)$-display $\Dd_\eK$ associated with $\Ls_\eK$. It then follows, that in this case, the integral models of \cite{KP} are canonical as per the definition above.  

 Our point of view fits with the well-established idea, going back to Deligne, that most Shimura varieties should be moduli spaces of $G$-motives with level structures. As such, they should have (integral) canonical models. We can not make this precise yet. However, we consider the locally universal $\Gg$-display as the crystalline avatar of the universal $G$-motive and show that its existence is enough to characterize the integral model. In fact, there should be versions of this characterization for other $p$-adic cohomology theories (see \cite{BMS}, \cite{SchICM}, \cite{BSPrism}). In joint work of the author with M. Rapoport \cite{PRShtuka}, which   developed after the first version of this paper was written, this idea is realized and a different characterization is given. This uses, instead of $\Gg$-displays, $p$-adic $\Gg$-sthukas over the $v$-sheaf $\SS_\eK^\Diamond$, as defined by Scholze (see \cite{Schber}). Our approach in the current paper is more ``classical", since it does not use any of the tools of modern $p$-adic analysis (perfectoid adic spaces, $v$-sheaves and diamonds, Banach-Colmez spaces, etc.) and when it applies it gives somewhat more precise information.

\subsection{}  Finally, we give an outline of the contents of the paper. In \S 2 we prove some preliminary facts about rings of Witt vectors and other $p$-adic period rings, that are used in the constructions. We continue on to define some terminology and give some more preliminaries on torsors in \S 3. In \S 4 we give  the definition of (Dieudonn\'e) displays with $\Gg$-structure. We also construct various relevant structures such an infinitesimal connection and give the notion of a rigid section
of a $\Gg$-display. In \S 5, we show how to construct, using the theory of Breuil-Kisin modules,  such a display from a $\Gg(\Z_p)$-valued crystalline representation of a $p$-adic field. We also give other similar constructions, for example a corresponding Breuil-Kisin-Fargues $\Gg$-module. In \S 6, we give the definition of an associated system and show how to compare two normal schemes with the same generic fiber which both support locally universal associated systems for the same local system. We also show how to give an associated system starting from a $p$-divisible group whose Tate module carries suitable Galois invariant tensors. In \S 7, we apply this to Shimura varieties and show that systems 
 $(\SS_\eK, \Lr_\eK, \{\calD_{\eK,\bar x}\}_{\bar x\in \SS_\eK(k)})$ 
as above are unique. In \S 8, we prove that the integral models of Shimura varieties of Hodge type constructed in \cite{KP} carry locally universal associated systems and are therefore uniquely determined. 

 \subsection{Acknowledgments} We thank M. Rapoport, P. Scholze, and the referee, for useful suggestions and corrections, and V. Drinfeld for interesting discussions.

 \subsection{Notations}
 
 Throughout the paper $p$ is a prime and, as usual, we denote by $\Z_p$, $\Q_p$, the $p$-adic integers, resp. $p$-adic
 numbers. We fix an algebraic closure $\bar\Q_p$ of $\Q_p$. If $F$ is a finite extension of $\Q_p$, we will
 denote by $\O_F$ its ring of integers, by $k_F$ its residue field and by $\breve F$ the completion of the maximal unramified extension of $F$ in $\bar \Q_p$. We often write $X\otimes_AB$ to denote the base change of a scheme $X$ over $\Spec(A)$
 to $\Spec(B)$.
 

  \section{Algebraic preliminaries}

 We begin with some preliminaries about  rings of Witt vectors and other
  $p$-adic period rings. The proofs can be omitted on the first reading.
  
  We consider the following conditions for a  $\Z_p$-algebra $R$. In what follows, we will have to assume some of these.
  
  \begin{itemize} 
\item[(1)] $R$ is complete and separated for the adic topology given by a finitely  generated ideal $\calA$  that contains $p$, and

 \item[(2)] $R$ is formally of finite type over $W=W(k)$, where $k$ is a perfect field of characteristic $p$.

  \item[ (N)] $R$ is a \emph{normal} domain, is flat over $\Z_p$, and satisfies (1) and (2).

  \item[ (CN)] $R$  satisfies (N) and is a complete local ring. 
    
\end{itemize}

By ``$p$-adic", we always mean $p$-adically complete and separated.

    \subsection{Witt vectors and variants}
 
 \begin{para}
 For a $\Z_{(p)}$-algebra $R$, we denote by 
 \[
 W(R)=\{(r_1, r_2,\ldots , r_n, \ldots )\ |\ r_i\in R\}
 \]
  the ring of ($p$-typical) Witt vectors of $R$. Let $\phi: W(R)\to W(R)$ and $V: W(R)\to W(R)$
  denote the Frobenius and Verschiebung, respectively. Let $I_R=V(W(R))\subset W(R)$ be the ideal of elements
  with $r_1=0$. The projection $(r_1, r_2,\ldots )\mapsto r_1$ gives an isomorphism $W(R)/I_R\simeq R$. For $r\in R$, we set as usual
  \[
  [r]=(r, 0, 0, \ldots )\in W(R)
  \]
  for the Teichm\"uller representative.
  Also denote by
   \[
{\rm gh}:  W(R)\rightarrow \prod\nolimits_{i\geq 1}  R
  \]
  the Witt (``ghost'') coordinates ${\rm gh}=({\rm gh}_i)_i$.  Recall, 
  \[
  {\rm gh}_i((r_1, r_2, r_3, \ldots ))=r_1^{p^{i-1}}+pr_2^{p^{i-2}}+\cdots + p^{i-1}r_i.
  \]
  \end{para}

  \begin{para}
Let $R$ be a complete Noetherian local ring with maximal ideal $\frakm_R$ and perfect residue field $k$
 of characteristic $p$. Assume that $p\geq 3$. There is a splitting
  $W(R)=W(k)\oplus W(\frakm_R)$ and following Zink (\cite{ZinkCFT}), we can consider the subring 
  \[
  \hat W(R)=W(k)\oplus \hat W(\frakm_R)\subset W(R),
  \]
   where $\hat W(\frakm_R)$ consists of those Witt vectors with $r_n\in \frakm_R$ for which the sequence $r_n$ converges to $0$ in the $\frakm_R$-topology of $R$. The subring $\hat W(R)$ is stable under $\phi$ and $V$.  In this case, both $W(R)$ and $\hat W(R)$ are $p$-adically complete and separated local rings.
 \end{para}

  \subsection{}
  
  In this section, we assume that the $\Z_p$-algebra $R$  satisfies (1) and (2):

Let $\frakM\subset W(R)$ be a maximal ideal with residue field $k'$.
We have $I_R\subset \frakM$, since $W(R)$ is $I_R$-adically complete and separated (\cite[Prop. 3]{Zink}). Let 
$\frakm_R=\frakM/I_R\subset W(R)/I_R=R$ be the corresponding maximal ideal of $R$.
Suppose that $\hat R $ is the completion of $R$ at $\frakm$.  Then $W(\hat R)$ is local henselian.
Denote by $W(R)^h_{\frakM}$ the henselization of the localization $W(R)_\frakM$.

\begin{lemma}\label{A1lemma} Assume, in addition to (1) and (2), that $R$ is an integral domain and $\Z_p$-flat.  
Then, the natural homomorphism $R\to \hat R$ induces injections 
\[
W(R)\subset W(R)^h_{\frakM}\subset W(\hat R).
\]
\end{lemma}
  
  \begin{proof}
Since $R$ is $p$-torsion free, 
  \[
{\rm gh}:  W(R)\hookrightarrow \prod\nolimits_{i\geq 1}  R
  \]
 is an injective ring homomorphism. If $f=(f_1, f_2,\ldots )\not\in \frakM$, then $f_1\not\in \frakm$. 
  Since $p\in \frakm$, we have ${\rm gh}_i(f)\not\in\frakm$, for all $i$. In particular, ${\rm gh}_i(f)\neq 0$
  and so since $R$ is a domain, $f$ is not a zero divisor in $W(R)$. It follows that
$W(R)$ is a subring of $W(R)_\frakM$. We now consider $W(R)\subset W(\hat R)$. Notice that $f$ is invertible in $W(\hat R)$ 
 since $\hat R$ is $p$-adically complete and $f_1$ is invertible in $\hat R$. Hence, we have an injection $W(R)_\frakM\hookrightarrow W(\hat R)$. The ring $W(\hat R)$ is local and henselian and $W(R)_\frakM\hookrightarrow W(\hat R)$ is a local ring homomorphism. It follows that the henselization 
 $W(R)^h_\frakM$ is contained in $W(\hat R)$.
 \end{proof}

  \quash{
  \begin{lemma}\label{dejonglemma}
  Assume, in addition to (1) and (2), that $R$ is a normal domain flat over $\Z_p$. Then, for all $n\geq 0$, including $n=\infty$, we have
  \[
 W_n(R)=\{f\in W_n(R)[1/p]\ |\ \forall F, \forall \xi:  R\to \O_F, \ \xi(f) \in W_n(\O_F) \},
 \]
 where $F$ runs over all finite extensions of $W(k)[1/p]$ and $\xi$ over all $W(k)$-algebra homomorphisms.
 \end{lemma}

  \begin{proof}
  Consider $f\in W_n(R)[1/p]$ so that $g=p^af\in W_n(R)$ for some $a\geq 0$. 
 Assume that $\xi(g)$ is divisible by $p^a$ in $W_n(\O_F)$, for all $\xi: R\to \O_F$.
 We would like to show that $g$ is divisible by $p^a$ in $W_n(R)$. This will be the case when 
 certain universal polynomials in the ghost coordinates ${\rm gh}_i(g)$ which have coefficients in $\Z[1/p]$, take values in $R$. By \cite[Prop. 7.3.6]{deJongCrys}  this is equivalent to asking that the same 
 polynomials in ${\rm gh}_i(\xi(g))$ take values in $\O_F$, for all $\xi$. This is true by our assumption.
 \end{proof}
 }

    \subsection{}\label{Algcond}
    
     Here we suppose that the $\Z_p$-algebra  $R$ satisfies (N).
    
\begin{para}
  We start by recalling a useful statement shown by de Jong \cite[Prop. 7.3.6]{deJongCrys}, in a slightly different language.
 
 \begin{prop}\label{dejong}
  We have
 \[
 R=\{f\in R[1/p]\ |\ \forall F, \forall \xi:  R\to \O_F, \ \xi(f) \in \O_F \},
 \]
 where $F$ runs over all finite extensions of $W(k)[1/p]$ and $\xi$ over all $W(k)$-algebra homomorphisms. \hfill $\square$
 \end{prop}

  \begin{cor}\label{dejonglemma}
We have
\[
 W(R)= (W(R)[1/p])\cap\prod\nolimits_{\xi} W(\O_F) 
\]
 where $F$, $\xi$ are as above.  
 \end{cor}

  \begin{proof}
 Under our assumption, $W(R)$ is $\Z_p$-flat. Consider $f\in W(R)[1/p]$ so that $g=p^af\in W(R)$ for some $a\geq 0$. 
 Assume that $\xi(g)$ is divisible by $p^a$ in $W(\O_F)$, for all $\xi: R\to \O_F$.
 We would like to show that $g$ is divisible by $p^a$ in $W(R)$. This will be the case when 
 certain universal polynomials in the ghost coordinates ${\rm gh}_i(g)$ which have coefficients in $\Z[1/p]$, take values in $R$. By Proposition \ref{dejong},  this is equivalent to asking that the same 
 polynomials in ${\rm gh}_i(\xi(g))$ take values in $\O_F$, for all $\xi$. This is true by our assumption.
 \end{proof}

\begin{prop}\label{corInter2}  Suppose $R$ satisfies (CN). We have
  \[
  W(R)\cap (\prod_{\xi: R\to \O_F}\hat W(\O_F))=\hat W(R).
  \]
  Here the product is over all finite extensions $F$ of $W(k)[1/p]$ and all $W(k)$-algebra homomorphisms $\xi: R\to \O_F$.
  The intersection takes place in  $\prod_{\xi: R\to \O_F} W(\O_F)$.
  \end{prop}

\begin{proof}
This follows from the definitions and:

\begin{prop}\label{limitProp} Suppose that $(f_n)_n$ is a sequence of elements of the maximal ideal $\frakm_R$ such that, for every finite extension $F$ of $W(k)[1/p]$ and every $W(k)$-algebra homomorphism $\xi: R\to \O_F$, the sequence
$(\xi(f_n))_n$ converges to $0$ in the $p$-adic topology of $F$. Then $(f_n)_n$ converges to $0$ in the $\frakm_R$-topology.
\end{prop}

\begin{proof}
Under our assumption on $R$, there is a finite injective ring homomorphism 
\[
\phi^*: R_0=W(k)\lps t_1,\ldots, t_r\rps\hookrightarrow R.
\]
(See \cite[p. 212]{Mats}, also \cite[proof of Lemma 7.3.5]{deJongCrys}.) We will use this to reduce the proof to the case $R=R_0$. Let $d$ be the degree of the  extension ${\rm Frac}(R)/{\rm Frac}(R_0)$ of fraction fields. For $f\in R$, we let $P(T; f)$ be the irreducible polynomial of $f$ over ${\rm Frac}(R_0)$;  this has degree $d'\leq d$. Since $R$, $R_0$ are both normal and $\phi^*$ is integral, we see that $P(T;  f)$ has coefficients in $R_0$.

Assume that $(f_n)_n$ is a sequence of elements in $\frakm_R$ that satisfies the assumption of the proposition. Fix a finite Galois extension $F'$ of ${\rm Frac}(R_0)$ that contains ${\rm Frac}(R)$ and  let $R'$ be the integral closure of $R_0$ in $F'$ so that $R_0\subset R\subset R'$. Then $R'$ is finite over $R_0$ (\cite[(31.B)]{Mats}).
For each $n$ we can write
\[
P(T; f_n)=\prod_{i=1}^{d'_n}(T-f_{n, i})
\]
with $f_{n, i}\in R'$ and $f_{n, 0}=f_n$. The elements $f_{n, i}$ are Galois conjugates of $f_n$. Every  $\xi: R\to \O_F$ as in 
the statement of the proposition extends to $\xi': R'\to \O_{F'}$, where $F'/F$ is finite, and the valuations of $\xi'(f_{n, i})$ agree with that of $\xi(f_n)$. It follows that for every $\xi': R'\to \O_{F'}$, the total sequence $(\xi'(f_{n, i}))_{n, i}$ goes to $0$.
Hence, the assumption of the proposition 
is satisfied for the sequence of the symmetric functions $(a_i(f_n))_n$ in $R_0$, for each $i$. These give the coefficients of the $P(T; f_n)$ and $f_n$ is a root 
of $T^{d-d'_n}P(T; f_{n})$ which we can write
\[
T^{d-d'_n}P(T; f_{n})=T^d+b_{1, n}T^{d-1}+\cdots +b_{d-1, n}T+b_{d, n}, \qquad b_{i, n} \in R_0.
\]
 Suppose now that we know
that the proposition is true for $R_0$. Then, we obtain that $b_{i, n}$ converges to $0$ in the $\frakm_{R_0}$-topology.
The identity $f^d_n+b_{1, n}f^{d-1}_n+\cdots +b_{d-1, n}f_n+b_{d, n}=0$ implies that $f_n^d$ converges to $0$ in the $\frak m_R$-topology. For $f\in R$, consider the sequence of ideals 
\[
\cdots \subset I_{a+1}=(\frakm_R^{a+1}; f)\subset I_a=(\frakm_R^a; f)\subset \cdots 
\]
of $R$. Krull's intersection theorem implies $\cap_{a=0}^{\infty}I_a=(0)$ and so, by Chevalley's lemma, the ideals $(I_a)_a$ also define the 
$\frakm_{R}$-topology of $R$. Since $f_n^d\in \frakm_{R}^a$ implies $f_n^{d-1}\in (\frakm_R^a; f_n)$, we quickly obtain, by decreasing induction on $d$, that $f_n$ converges to $0$ in the $\frak m_R$-topology. 

It remains to prove the proposition for the power series ring $R_0$:

 Set $Y=\Spec(R_0)$, and let $h: \ti Y\to Y$ be the blowup of $Y$ at the maximal ideal $\frakm=(p, t_1,\ldots, t_r)$. The exceptional divisor $E$ can be identified with ${\mathbb P}(\frakm/\frakm^2)\cong {\mathbb P}^r_k$. We argue by contradiction: Assume that $(f_n)_n$ is a sequence of elements in $\frakm$ that satisfies the assumption of the proposition but is such that $(f_n)_n$ does not converge to $0$ in the $\frakm$-topology. Then, by replacing $(f_n)_n$ by a subsequence, we can assume that, there is an integer $N\geq 1$, such that  $f_n\in \frakm^N-\frakm^{N+1}$, for all $n$. 
 Write
 \[
 f_n=\sum_{i\geq N} f_{n, i}, \qquad f_{n, i}=\sum_{a_0+\cdots +a_r=i}[c_{a_0a_1\cdots a_r}(n)]p^{a_0}t_1^{a_1}\cdots t^{a_r}_r,
 \]
 with $c_{a_0a_1\cdots a_r}(n)\in k$.
 Then the proper transform $Z(f_n)\subset \ti Y$ of $f_n=0$ intersects the exceptional divisor $E={\mathbb P}^r_k$ along the hypersurface $S_n\subset {\mathbb P}^r_k={\rm Proj}_k(k[u_0,\ldots, u_r])$ of degree $N$ defined by the homogeneous equation
 \begin{equation}\label{HypersurfaceEq}
\sum_{a_0+\cdots +a_r=N} c_{a_0a_1\cdots a_r}(n)u_0^{a_0}u_1^{a_1}\cdots u^{a_r}_r=0.
 \end{equation}
\begin{lemma}  After replacing $(f_n)_n$  by a subsequence,  we can find a $\bar k$-valued point $x$ of the exceptional divisor ${\mathbb P}^r_k$  which does not lie on any of the proper transforms $Z(f_n)$ of $f_n=0$, for all $n$.
\end{lemma}

\begin{proof}  We argue by contradiction: Assume that for any given point $x\in  {\mathbb P}^r(\bar k)$ and almost all $n$, $Z(f_n)$ contains $x$. Then, also for every finite set of points $A(m)=\{x_1,\ldots , x_m\}$, we have $A(m)\subset Z(f_n)$, for almost all $n$.
Since $Z(f_n)\cap {\mathbb P}^r$ is, for each $n$, a hypersurface $S_n$ of fixed degree $N$, when $m\gg N$ this is not possible.
\end{proof}
\smallskip

Now choose $x$ as given by the lemma and lift it 
to a point $\ti x\in \ti Y(\O_{F})$, where $F$ is some finite extension of $
W(k)_\Q$; this induces an $O_F$-point of $Y$ by $\ti Y(O_F)\to Y(O_F)$ and we also denote this by $\ti x$.
 By assumption, $\ti x^*(f_n)\to 0$ in $F$. By our choice, the image of $\ti x: \Spec(O_F)\to \ti Y$ intersects the exceptional divisor $E={\mathbb P}^r_k$ away from the hypersurface $S_n$.
Using this and equation (\ref{HypersurfaceEq}) which cuts out $S_n$, we obtain 
\[
v(\ti x^*(f_{n, N}))=N\cdot {\rm min}_{i=0}^r \{\ti x^*(t_i)\},
\]
where $v$ is the $p$-adic valuation and, for uniformity, we denote $p$ by $t_0$. 
Since $v(\ti x^*(f_{n, i}))\geq (N+1)\cdot {\rm min}_{i=0}^r \{\ti x^*(t_i)\}$ for $i\geq N+1$, 
we also have $v(\ti x^*(f_{n}))=N\cdot {\rm min}_{i=0}^r \{\ti x^*(t_i)\}$, which contradicts $v(\ti x^*(f_n))\to +\infty$.
\end{proof}\end{proof}
\end{para}
    
   \subsection{Some perfectoid rings}\label{appCompl}
We  assume that the $\Z_p$-algebra $R$ satisfies (CN) with $k=\bar{\mathbb F}_p$. 

  \begin{para} Fix an algebraic closure $\overline{ F(R)}$ of the fraction field $F(R)={\rm Frac}(R)$. Denote by $\ti R$
  the union of all finite normal $R$-algebras $R'$ such that:
  
\begin{itemize}
\item[1)] $R\subset R'\subset \overline{ F(R)}$, and

\item[2)] $R'[1/p]$ is finite \'etale over $R[1/p]$.
\end{itemize}

Note that all such $R'$ are local and complete. We will denote by $\bar R$ the integral closure
of $R$ in $\overline{ F(R)}$, so that $\bar R$ is the union of all $R'$ as in (1).

Let us set
\[
\Gamma_R={\rm Gal}(\ti R[1/p]/R[1/p]),
\]
which acts on $\ti R$.
 Also denote by
  \[
  \ti R^\wedge=\varprojlim\nolimits_n \ti R/p^n\ti R,\quad  \ti R^\wedge=\varprojlim\nolimits_n \bar R/p^n\bar R,
  \]
   the $p$-adic completions. 
   
    \end{para}
  
  \begin{para}  When $R=W=W(k)$, we denote ${\bar W}^\wedge={\ti W}^\wedge$ by $\O$.

 \begin{prop}\label{prop131}
 The natural maps $\bar R\to \bar R^\wedge$, $\ti R\to \ti R^\wedge$, are injections and induce
  isomorphisms $\bar R/p^n\bar R\simeq \bar R^\wedge/p^n\bar R^\wedge$, $\ti R/p^n\ti R\simeq \ti R^\wedge/p^n\ti R^\wedge$, 
 for all $n\geq 1$.
 \end{prop}
 
 \begin{proof}
This is given by the argument in \cite[Prop. 2.0.3]{Brinon} which deals with the case of $\ti R$ and the case of $\bar R$ is similar.
 \end{proof}
 
 \begin{prop}\label{prop132} Let $S$ be $\ti R^\wedge$ or $\bar R^\wedge$.
 
 a)  $S$ is $p$-adically complete and separated and  is flat over $\Z_p$.

 b)   $S$ is an integral perfectoid algebra (in the sense of \cite[3.2]{BMS}). 
 
 c) $S$ is local and strict henselian.
 \end{prop}
  
  \begin{proof}
  
  Part (a) is also given by \cite[Prop. 2.0.3]{Brinon}. 
   Let us show (b) for $S=\ti R^\wedge$. The argument for $\bar R^\wedge$ is similar and actually simpler. 
  We see that $\ti R$ and so $S$, contains an element $\pi$ with $\pi^p=p$. Then $S$ is   $\pi$-adically complete. Using
 \cite[Lemma 3.10]{BMS}, it is now enough to show that the Frobenius $\phi: S/\pi S\to S/p S$ is an isomorphism and that $\pi$ is not a zero divisor in $S$. Since $S$ is $\Z_p$-flat, $\pi$ is not a zero divisor. By Proposition \ref{prop131}, we
  have $\ti R/p\ti R\simeq S/pS$ and similarly $\ti R/\pi\ti R\simeq S/\pi S$. Hence, it is enough to show that
  $\phi: \ti R/\pi \ti R\to \ti R/p \ti R$ is an isomorphism. Suppose now $x\in \ti R$ satisfies $x^p=p y$ with $y\in \ti R$. Then $(x/\pi)^p=y$ and since $\ti R$ is a union of normal domains, we have $x/\pi=z\in \ti R$. This shows injectivity. To show surjectivity, consider $a\in R'\subset \ti R$ and consider 
  \[
  R''=R'[X]/(X^{p^2}-p X-a).
   \] 
  This is a finite $R$-algebra, and so also $p$-adically complete. It is \'etale over $R'[1/p]$ since, by $p$-adic completeness, the derivative  $p(pX^{p^2-1}-1)$ is a unit in $R''[1/p]$. Now there is $R''\to \overline{ F(R)}$ that extends $R'\subset \overline{ F(R)}$ and the image $b$ of $X$  in $\overline{ F(R)}$ is contained in a finite $R'$-algebra which is also \'etale over $R'[1/p]$. This gives $b\in \ti R$ with $b^{p^2}\equiv a\ {\rm mod\ } p\ti R$ which implies surjectivity. 
   
For part (c), since $S=\ti R^\wedge$ is $p$-adically complete, it is enough to show that these properties are true for $S/pS\simeq \ti R/p\ti R$. We can see that $\ti R$ is both local and strict henselian, and then so is the quotient $\ti R/p\ti R$. The argument for $\bar R^\wedge$ is similar.
\end{proof}

 \begin{thm}\label{FaltingsAlmost}
The action of $\Gamma_R$ on $\ti R$ extends to a
$p$-adically continuous action on $\ti R^\wedge$ and we have
\[
(\ti R^\wedge)^{\Gamma_R}=R.
\]
\end{thm}

\begin{proof}
 By Faltings \cite{Fa} or \cite[Prop. 3.1.8]{Brinon}, we have
 \[
 R\subset (\ti R^\wedge)^{\Gamma_R}\subset R[1/p].
 \]
 Using this, we see that it remains to show that $\ti R^\wedge\cap R[1/p]=R$, with the intersection in $\ti R^\wedge[1/p]$.
 Suppose $f\in \ti R^\wedge\cap R[1/p]$. By applying Proposition \ref{dejong}, we see that it is enough to show 
 that $\xi(f)\in \O_F$, for all $W$-algebra homomorphisms $\xi: R\to \O_F$ with $F$ a finite extension of 
 $W[1/p]$. Choose such a $\xi: R\to \O_F$. We can extend $\xi$ to $\bar \xi: \ti R\to \bar\O_F=\O_{\bar F}$ and then by $p$-adic completion to 
 \[
 {\bar \xi}^\wedge: \ti R^\wedge\to \O.
 \]
  This gives  ${\bar \xi}^\wedge: \ti R^\wedge[1/p]\to \O[1/p]$. But then
 $\bar\xi^\wedge(f)\in F\cap \O=\O_F$.
 \end{proof}
 
  \end{para}

 \subsection{Period rings}\label{appAinf}
 
 We  continue with the same assumptions and notations. In particular, $R$ satisfies (CN) with $k=\bar{\mathbb F}_p$. 
 
 \begin{para}
We now restrict to the case $S=\ti R^\wedge$.
 We will use the notations of \cite[\S 3]{BMS}. Consider the tilt 
 \[
 S^\flat=\varprojlim\nolimits_\phi S/pS=\varprojlim\nolimits_\phi S
 \]
 and similarly for $\O^\flat$. 
 
 \begin{lemma}\label{localS}
  The ring $S^\flat$ is local strict henselian with residue field $k$. 
 \end{lemma}
 
\begin{proof} As we have seen, the rings $S$ and $S/pS=\ti R/p\ti R$ are local and strict henselian with residue field $k$. Denote by $x\mapsto \bar x$
the map $S/pS\to k$.
The Frobenius $S/pS\to S/pS$ is surjective and, hence, $S^\flat\to S/pS$ is surjective. If 
$
x=(x_0, x_1, x_2, \ldots )\in S^\flat
$
with $x_i\in S/pS$, $x_{i+1}^p=x_i$, has $x_0$ a unit, then all $x_i$ and also $x$ are units. Hence, $S^\flat$ is local with residue field $k$ and $(x_0, x_1, x_2, \ldots )\mapsto (\bar x_0, \bar x_1, \bar x_2, \ldots )$ is the residue field map $S^\flat\to k$.
Now consider $f(T)\in S^\flat[T]$ with a simple root $\kappa=(\kappa_0, \kappa_1, \kappa_2, \ldots )$ in $k$, with $\kappa_i=\sqrt[p^i]{\kappa_0}$. Since $S/pS$ is local henselian, the simple root $\kappa_i\in k$ of $f_i(T)$ lifts uniquely to a root $a_i\in S/pS$.
By uniqueness, we have $a_{i+1}^p=a_i$ and so $a=(a_0, a_1,\ldots )$ is a root in $S^\flat$ that lifts $\kappa$. Hence, $S^\flat$ satisfies Hensel's lemma.
\end{proof}
\end{para}

\begin{para}\label{thetanotations}
 
 We set $A_{\rm inf}(S)=W(S^\flat)$ for Fontaine's ring. By \cite[Lemma 3.2]{BMS}, we have
\[
A_{\inf}(S)\cong\varprojlim\nolimits_{\phi} W_r(S).
\]
 This gives corresponding homomorphisms 
 \[
 \tilde \theta_r: A_\inf(S)\to W_r(S).
 \]
 We also have the  standard  homomorphism of $p$-adic Hodge theory
 \[
 \theta:  A_\inf(S)\to S,
 \]
 given by
 \[
 \theta(\sum\nolimits_{n\geq 0} [x_n]p^n)=\sum\nolimits_{n\geq 0} x_n^{(0)}p^n.
 \]
 Here, we write $x=(x^{(0)}, x^{(1)},\ldots )\in S^\flat=\varprojlim\nolimits_\phi S$.
 As in \cite[\S 3]{BMS}, the homomorphism $\theta$ lifts to 
 \[
 \theta_\infty:  A_\inf(S)\to W(S),
 \]
 given by
 \[
 \theta_\infty(\sum\nolimits_{n\geq 0} [x_n]p^n)=\sum\nolimits_{n\geq 0}[x^{(0)}_n]p^n.
 \]

 \end{para}
 
 \begin{para}\label{par143}
 
 In the following, $\alpha$ runs over all rings homomorphisms $\alpha: S\to \O$ which are obtained from some $W$-homomorphism
 $ \ti R\to \ti W=\bar W$ by $p$-adic completion. There is a corresponding $A_\inf(S)\to A_\inf(\O)$   given by applying
 the Witt vector functor to $\alpha^\flat: S^\flat\to \O^\flat$.
 Note that if $F$ is a finite extension of $W[1/p]$, then any homomorphism $ \xi: R \to \O_F$ extends to such an $\alpha: S\to \O$.

  \begin{lemma}\label{lemmainf0}
a)  The homomorphism $A_\inf(S)\to \prod_{\alpha} A_\inf(\O)$ is injective.

b) The ring $A_\inf(S)$ is $p$-adically complete, local strict henselian and $\Z_p$-flat.
 \end{lemma}
 
 \begin{proof} We first note that $\ti R\to \prod_\alpha \bar W$ is injective
 (reduce to the case $R$ is a formal power series ring by an argument as in the proof of Proposition \ref{limitProp}).
Also, $\ti R\cap (\prod_\alpha p\bar W)=p\bar R$, as this easily follows by Proposition \ref{dejong} applied to the algebras $R'$. Therefore, 
 \[
 S/pS=\ti R/p\ti R\subset \prod\nolimits_{\alpha} \bar W/p\bar W= \prod\nolimits_{\alpha}\O/p\O
 \]
 is injective. Hence, $S^\flat\to \prod_\alpha\O^\flat$ is injective and part (a) follows. 
 
 To show part (b), recall $A_\inf(S)=W(S^\flat)$.
 The ring $S^\flat$ is perfect, and so $pW(S^\flat)=I_{S^\flat}$, $W(S^\flat)/pW(S^\flat)=S^\flat$. It follows that $A_\inf(S)$ is $p$-adically complete and that $p$ is not a zero divisor. Lemma \ref{localS} now implies that $A_\inf(S)$ is local and strict henselian. \end{proof}

 \end{para}
 
 \begin{para}
 
 Now let us fix an embedding $\bar W\hookrightarrow \ti R$, which induces $\O\hookrightarrow  S$. 
 Let 
 \[
 \epsilon=(1, \zeta_p,   \zeta_{p^2},\ldots )\in \O^\flat=\varprojlim\nolimits_\phi \O
 \]
 be a system of primitive $p$-th power roots of unity. Set
  \[
 \mu=[\epsilon]-1\in A_{\inf}(\O)\subset A_\inf(S).
 \]

\begin{prop}\label{lemmainf1}
a) The element $\mu$ is not  a zero divisor in $A_\inf(S)$.

b) Suppose that $f$ in $A_\inf(S)$ is such that, for every $\alpha: S\to \O$ obtained from $\ti R\to \bar W$
as above, $\alpha(f)$ is in the ideal 
$(\mu)$ of $A_\inf(\O)$. Then, $f$ is in $(\mu)$.
\end{prop}

\begin{proof}
Part (a) follows from \cite[Prop. 3.17 (ii)]{BMS}. As in the proof of loc. cit., the ghost coordinate vectors of $\tilde\theta_r(\mu)$ are
\[
{\rm gh}(\tilde\theta_r(\mu))=(\zeta_{p^r}-1,\ldots, \zeta_p-1)\in S^r,
\]
and the result follows from this. 

Let us show (b). Recall
\[
A_\inf(S)=\varprojlim\nolimits_{\phi} W_r(S)\subset \prod\nolimits_{r\geq 1} W_r(S); \  a\mapsto (\tilde \theta_r(a))_r.
\]
Now suppose that $\alpha(f)=\mu\cdot b_\alpha$, for $b_\alpha\in A_\inf(\O)$. Apply $\tilde \theta_r$ to obtain
\[
\alpha(\tilde \theta_r(f))=\tilde \theta_r(\alpha(f))=\tilde\theta_r(\mu)\cdot \tilde\theta_r(b_\alpha).
\]
This implies that, for all $i=1,\ldots, r$, and all $\alpha$,
\[
\zeta_{p^i}-1\ |\ \alpha({\rm gh}_i(\tilde \theta_r(f)))
\]
in $\O$. The same argument as in the proof of Lemma \ref{lemmainf0} (a) above, gives
\[
S/(\zeta_{p^i}-1)S\hookrightarrow \prod\nolimits_\alpha \O/(\zeta_{p^i}-1)\O.
\]
This implies that $\zeta_{p^i}-1$ (uniquely) divides ${\rm gh}_i(\tilde \theta_r(f))$ in $S$. 

We claim that the quotients $g_{i, r}={\rm gh}_i(\tilde \theta_r(f))/(\zeta_{p^i}-1)$ in $S$ are the ghost coordinates
of an element $\gamma_r$ of $W_r(S)$, which is then the quotient $\tilde \theta_r(f)/\tilde \theta_r(\mu)$. To check this we have to show that certain universal polynomials in the $g_{i, r}$ with coefficients in $\Z[1/p]$ take values in $S$. This holds 
after evaluating by $\alpha: S\to \O$ and so the same argument using Proposition \ref{dejong} as before, shows that it is true. It follows that, for all $r$, $\ti\theta_r(\mu)$ uniquely divides $\ti\theta_r(f)$ in $W_r(S)$ and, in fact,
\[
\ti\theta_r(f)=\ti\theta_r(\mu)\cdot \gamma_r.
\]
Applying $\phi$ gives
\[
\ti\theta_{r-1}(f)=\ti\theta_{r-1}(\mu)\cdot \phi(\gamma_r).
\]
in $W_{r-1}(S)$. Therefore, $\phi(\gamma_r)=\gamma_{r-1}$. Hence, there is  $\gamma\in A_\inf(S)=\varprojlim\nolimits_{\phi} W_r(S)$ such that $\gamma_r=\ti \theta_r(\gamma)$. Then, 
$
f=\mu\cdot \gamma.
$
\end{proof}

\begin{cor}\label{corinf}
a) We have
$
A_\inf(S)=(A_\inf(S)[1/\mu])\cap \prod_\alpha A_\inf(\O).
$

b) Suppose that
$M_1$ and $M_2$ are two finite free $A_\inf(S)$-modules with $M_1[1/\mu]=M_2[1/\mu]$, and such that
$\alpha^*M_1=\alpha^*M_2$ as $A_\inf(\O)$-submodules of $\alpha^*M_1[1/\mu]=\alpha^*M_2[1/\mu]$, for all $\alpha$. Then $M_1=M_2$.
\end{cor}

\begin{proof}
Part (a) follows directly from the previous proposition. Part (b) follows by applying (a) to the entries of the matrices expressing 
a basis of $M_1$ as a combination of a basis of $M_2$, and vice versa.
\end{proof}
\end{para}
 
\begin{para}
As in \cite[\S 3]{BMS}, set $\xi=\mu/\phi^{-1}(\mu)$ which is a generator of the kernel of the homomorphism  $\theta: A_\inf(S)\to S$.
Let $A_{\rm cris}(S)$ be the $p$-adic completion of the divided power envelope of $A_\inf(S)$ along $(\xi)$. By \cite[App. to XVII]{Schber}, the natural homomorphism
\[
A_\inf(S)\rightarrow A_{\rm cris}(S)
\]
is injective.
We also record:

\begin{lemma}\label{lemmainf3}
$A_\inf(S)^{\phi=1}=\Z_p$.
\end{lemma}

\begin{proof}
It is enough to show that $(S^\flat)^{\phi=1}={\mathbb F}_p$, i.e. that $(\ti R/p\ti R)^{\phi=1}=(S/pS)^{\phi=1}={\mathbb F}_p$. 
Now argue as in \cite[6.2.19]{Brinon}:
Suppose $a\in \ti R$ is such that $a^p=a\, {\rm mod}\, p\ti R$. Then $a\in R'$, for some $R'/R$ finite normal with $R'[1/p]$ \'etale over $R[1/p]$ and $a^p=a\, {\rm mod}\, pR'$.
 By Hensel's lemma for the $p$-adically complete $R'$,
the equation $x^p-x=0$ has a root $b$ in $R'$ which is congruent to $a\, {\rm mod}\, pR'$. But $R'$ is an integral domain, so any such root is one of the standard roots in $\Z_p$, so $b$ is in $\Z_p$ and $b=a\, {\rm mod}\, pR'$ in ${\mathbb F}_p$.
\end{proof}

\end{para}

\ve
 
\section{Shimura pairs and $\Gg$-torsors}
 
 \subsection{Shimura pairs}\label{ssShimura}
 
 We first set up some notation for (integral) Shimura pairs and then define the notion of a local Hodge embedding.
 
  \begin{para}
  
  Let $G$ be a connected reductive algebraic group over $\Q_p$
  and $\{\mu\}$ the $G(\bar \Q_p)$-conjugacy class of a
  minuscule cocharacter $\mu: \Gm_{,\bar\Q_p}\to G_{\bar \Q_p}$. 
  
  To such a pair $(G, \{\mu\})$, we associate:
  
   \begin{itemize}
   
  \item The reflex field $E\subset \bar\Q_p$. As usual, $E$ is the field of definition of the conjugacy class $\{\mu\}$ (i.e. the finite extension of $\Q_p$ which corresponds to the subgroup of $\sigma\in \Gal(\bar\Q_p/\Q_p)$  such that $\sigma(\mu)$ is $G(\bar\Q_p)$-conjugate to $\mu$.)

  \item The $G$-homogeneous variety $X_\mu=X_\mu(G)$ of parabolic subgroups of $G$ of type $\mu$. This is a projective smooth $G$-variety over $E$ with $X_\mu(\bar\Q_p)=G(\bar\Q_p)/P_{\mu}(\bar\Q_p)$.
  
  \end{itemize}
  
  \end{para}
  
  \begin{para}\label{par113}
  
  An \emph{integral local Shimura pair} is $(\Gg, \calM)$ where: 
  
  \begin{itemize}
   
  \item[i)] $\Gg$ is a parahoric group scheme
  over $\Z_p$ with generic fiber $G$.
  
 \item[ii)]  $\calM$ is a normal flat and projective $\O_E$-scheme $\calM$
  with $\Gg$-action which is a model of $X_\mu$, in the sense that there is a 
  $G$-equivariant isomorphism
  \[
 \calM\otimes_{\O_E}E\simeq X_\mu.
  \]
     \end{itemize}
     \end{para}
     
     \begin{para}\label{LMtheory}
     The theory of local models suggests that there should be a canonical choice of a scheme $\calM$ as in (ii) which depends only on
     $(\Gg, \{\mu\})$, up to unique $\Gg$-equivariant isomorphism.  
     
     More precisely, Scholze conjectures \cite[Conjecture 21.4.1]{Schber}, the existence of such a scheme
     $\Mloc=\Mloc{(\Gg, \{\mu\})}$ (the \emph{local model}), which has, in addition, reduced special fiber (and hence is normal), and which is uniquely characterized by its corresponding $v$-sheaf (see loc. cit.  for details). Local models $\Mloc$ as in \cite[Conjecture 21.4.1]{Schber}, have been constructed in many cases. We list some  results:

 1) If $G$ splits over a tamely ramified extension of $\Q_p$, there is a construction of local models in 
\cite[\S 7]{PZ}  (which was adjusted as in \cite[\S 2.6]{HPR} when $p$ divides $|\pi_1(G_{\rm der}(\bar \Q_p)|$). Conjecturally, these satisfy the conditions of Scholze's  \cite[Conjecture 21.4.1]{Schber} (\cite[Conjecture 2.16]{HPR}). This has been shown in almost all cases that $(G, \{\mu\})$ is of local abelian type (see \S\ref{RemarkLHT}) 
\cite[Theorem 2.15]{HPR}, see also \cite{LourencoThesis}.

 2) If $(G, \{\mu\})$ is of local abelian type and $p$ is odd, the local models $\Mloc$ have been constructed by Louren\c co \cite[4.22, 4.24]{LourencoThesis}.

 In our application to Shimura varieties, we would like to choose $\calM=\Mloc$. 
However, it is convenient to develop the set-up for a more general $\calM$.
     
     \end{para}
     
 \begin{para}
 
 Let $(\Gg, \calM)$ be an integral local Shimura pair. We consider the following conditions:
 
(H1) There is a group scheme homomorphism
\[
\iota: \Gg\hookrightarrow \GL_n
\]
which is a closed immersion such that $\{\iota(\mu)\}$ is the conjugacy class
of one the standard minuscule cocharacters $\mu_d(a)={\rm diag}(a^{(d)}, 1^{(n-d)})$ of $\GL_n$ for some $1\leq d\leq n-1$, and $\iota(\Gg)$ contains the scalars $\Gm$.

Note that the corresponding $\GL_n$-homogeneous space $X_{\mu_d}(\GL_n)$ is the Grassmannian ${\rm Gr}(d, n)$ of $d$-spaces in $\Q_p^n$. Under the assumption (H1), $\iota$ gives an equivariant closed embedding $X_\mu\subset {\rm Gr}(d, n)_E$. Set $\Lambda=\Z_p^n$. The Grassmannian ${\rm Gr}(d, n) $  has a natural model over $\Z_p$ which we will denote by ${\rm Gr}(d, \Lambda)$. 
  
 (H2)  The normalization of the Zariski closure of $X_\mu$ in ${\rm Gr}(d, \Lambda)_{\O_E}$ is $\Gg$-equivariantly isomorphic to $\calM$. Hence, there is a $\Gg$-equivariant finite morphism
\[
\iota_*: \M\to {\rm Gr}(d, \Lambda)_{\O_E}
\]
which is $X_\mu \subset {\rm Gr}(d, n)_{E}$ on the generic fibers.

We call an $\iota$, that satisfies (H1) and (H2), a \emph{integral local Hodge embedding} for the pair $(\Gg, \calM)$.
When such an integral local Hodge embedding exists, we say that $(\Gg, \calM)$ is of \emph{integral local Hodge type.}
\smallskip

We often need the following stronger version of (H2):
\smallskip

(H2*)  The Zariski closure of $X_\mu$ in ${\rm Gr}(d, \Lambda)_{\O_E}$ is $\Gg$-equivariantly isomorphic to $\calM$.
Hence, $X_\mu \subset {\rm Gr}(d, n)_{E}$ extends to a $\Gg$-equivariant closed immersion
\[
\iota_*: \M\hookrightarrow {\rm Gr}(d, \Lambda)_{\O_E}.
\]

We call   an $\iota$, that satisfies (H1) and (H2*), a \emph{strongly integral local Hodge embedding} for the pair $(\Gg, \calM)$.
When such  an embedding exists, we say that $(\Gg, \calM)$ is of \emph{strongly integral local Hodge type.}
\end{para}

\begin{para}\label{RemarkLHT}
 This notion should also be compared to the often used weaker notion of \emph{local Hodge type} which refers to the (rational) local Shimura pair $(G, \{\mu\})$: 
 
 We say that $(G, \{\mu\})$ is of local Hodge type if there is a group scheme homomorphism
$
\iota: G\hookrightarrow \GL_n
$
which is a closed immersion such that $\{\iota(\mu)\}$ is the conjugacy class
of one of the standard minuscule cocharacters $\mu_d$.  
(There is also the following related notion: $(G, \{\mu\})$ is ``of abelian type" means that there is a central lift $(G_1,\{\mu_1\})$
of $(G_{\rm ad}, \{\mu_{\rm ad}\})$ which is of local Hodge type.)

The following statement that  relates the two notions when $\calM=\Mloc$ can  be extracted from the proof of \cite[Theorem 2.15]{HPR}:

\begin{prop}\label{propLHT}
Suppose that $(G, \{\mu\})$ is of local Hodge type with $\iota: G\hookrightarrow \GL_n$ as above such that $\Gm\subset\iota( G)$, and $\Gg$ is a parahoric stabilizer with $G=\Gg\otimes_{\Z_p}\Q_p$. Assume also that $p$ is odd, $p\nmid |\pi_1(G_{\rm der}(\bar\Q_p))|$, and that $G$ splits over a tamely ramified extension of $\Q_p$.
Then $(\Gg, \Mloc)$ is of strongly integral local Hodge type.\endproof
\end{prop}

\end{para}

\subsection{Torsors}\label{ss:torsors}

In this paragraph, $\Gg$ is a smooth connected affine flat group scheme
  over $\Z_p$ with generic fiber $G$. We will collect some general statements about $\Gg$-torsors.
We denote by ${\rm Rep}_{\Z_p}(\Gg)$ the exact tensor category of representations of $\Gg$ on finite free $\Z_p$-modules, i.e. of group scheme homomorphisms $\rho: \Gg\to \GL(\Lambda')$   with $\Lambda'$ a finite free $\Z_p$-module.

\begin{para}\label{para321}
Suppose $\iota:\Gg\hookrightarrow \GL(\Lambda)\simeq \GL_n$ is a closed group scheme immersion such that $\iota(\Gg)$ contains the scalars $\Gm$. Here $\Lambda$ is a free $\Z_p$-module of rank $n$.
Denote by $\Lambda^\otimes =\oplus_{m, n\geq 0}\, \Lambda^{\otimes m}\otimes_{\Z_p}(\Lambda^\vee)^{\otimes n}$ the total tensor algebra of $\Lambda$, where $\Lambda^\vee={\rm Hom}_{\Z_p}(\Lambda, \Z_p)$.
By using the improved\footnote{in the sense that one does not need symmetric and alternating tensors, as in \cite[Prop. 1.3.2]{KisinJAMS}.} version of \cite[Prop. 1.3.2]{KisinJAMS} given in \cite{DeligneLetter}, we can 
realize $\Gg$ as the scheme theoretic stabilizer of a finite list of tensors $(s_a)\subset   \Lambda^\otimes$: For any $\Z_p$-algebra 
$R$ we have
\[
\Gg(R)=\{g\in \GL(\Lambda\otimes_{\Z_p}R)\ |\ g\cdot (s_a\otimes 1)=(s_a\otimes 1), \forall\alpha\}.
\]
Since we assume that $\iota(\Gg)$ contains the scalars $\Gm$ and $a\in \Gm$ acts on $\Lambda^{\otimes m}\otimes_{\Z_p}(\Lambda^\vee)^{\otimes n}$ via $a^{m-n}$, we see that the $s_a$ are contained in the part of the tensor algebra with $n=m$.
In particular, we can assume that every tensor $s_a$ is given by a $\Z_p$-linear map $\Lambda^{\otimes n}\to \Lambda^{\otimes n}$, for some $n=n_a$.

  \end{para}

\begin{para}\label{gen12}
  Let $A$ be a $\Z_p$-algebra. Set $S=\Spec(A)$. Suppose that
  $T\to S$ is a $\Gg$-torsor. By definition, this means that $T$ supports a (left) $\Gg$-action  $\Gg\times T\to T$ such that
  $\Gg\times T \xrightarrow{\sim}   T\times_S T $ given by $(g, t)\mapsto (gt, t)$ is an isomorphism, and $T\to S$ is faithfully flat and quasi-compact (fpqc). By descent, $T$ is affine, so $T=\Spec(B)$ with 
  $A\to B$  faithfully flat. 
  
  If $\rho: \Gg\to \GL(\Lambda')$ is in ${\rm Rep}_{\Z_p}(\Gg)$,   we can consider the vector bundle over $S$ which is attached to $T$ and $\rho$: 
  \[
  T(\rho)=T\times^\Gg_{\Spec(\Z_p)} \bbA({\Lambda'}) =(T\times_{\Spec(\Z_p)}  \bbA({\Lambda'}))/\sim
  \]
 where $(g^{-1}t, \lambda)\sim (t, \rho(g)\lambda)$. Here, $\bbA({\Lambda'})$ is the affine space
 $\Spec({\rm Symm}_{\Z_p}(\Lambda'^\vee))$ over $\Spec(\Z_p)$. 
 In what follows, we often abuse notation, and also denote 
 by $T(\rho)$ the corresponding $A$-module of global sections of the bundle $T(\rho)$.
 
By \cite{Broshi}, see also \cite[19.5.1]{Schber}, this construction gives an equivalence   between the category of $\Gg$-torsors $T\to S$ and the category of exact tensor functors 
\[
T:  {\rm Rep}_{\Z_p}(\Gg)\to {\rm Bun}(S); \ \ \rho\mapsto T(\rho),
\]
 into the  category of vector bundles ${\rm Bun}(S)$ on $S$.
 
 Assume now that $T$ is a $\Gg$-torsor and $\iota: \Gg\hookrightarrow \GL(\Lambda)$ is as in \S\ref{para321}.
    
  \begin{prop}\label{torsorRep} The $A$-module $M=T(\iota)$ is locally free of rank $n$ and comes equipped  with tensors $(m_a)\subset M^\otimes$ such that there is a $\Gg$-equivariant isomorphism 
  \[
  T\simeq \underline{\rm Isom}_{(m_a), (s_a)}(T(\iota), \Lambda\otimes_{\Z_p}A).
  \]
  \end{prop}
    
    \begin{proof}   This is quite standard, see for example  \cite[Cor. 1.3]{Broshi} for a similar statement. We sketch the argument: By the above, $M=T(\iota)$ is a locally free $A$-module of rank $n={\rm rank}_{\Z_p}(\Lambda)$.  Since the construction of $T(\rho)$ commutes with tensor operations (i.e. $\rho\mapsto T(\rho)$ gives a tensor functor)   we have 
    \[
    M^\otimes\simeq T\times^\Gg_{\Spec(\Z_p)} \bbA( \Lambda^\otimes).
    \]
     We can think of $s_a\in \Lambda^\otimes$ as $\Gg$-invariant linear maps $s_a: \Z_p\to \Lambda^\otimes$ which give $1\times s_a: S=\Gg\backslash T\to M^\otimes$, i.e. tensors $m_a=1\times s_a\in M^\otimes$. Set 
     $T'= \underline{\rm Isom}_{(m_a), (s_a)}(M, \Lambda\otimes_{\Z_p}A)$ with its natural left $\Gg$-action.
    The base change $T'\times_{S}T$ is equivariantly identified with $ \Gg\times T \simeq T\times_{S}T$ and the proof follows.
    \end{proof}
    
\begin{Remark}\label{indIota}
{\rm 
Suppose that $\iota':\Gg\hookrightarrow {\rm GL}(\Lambda')$ is another closed group scheme immersion that realizes $\Gg$ as the subgroup scheme that fixes $(s_b')\subset (\Lambda')^\otimes$. It follows that there is a $\Gg$-equivariant isomorphism
\[
\underline{\rm Isom}_{(m'_b), (s'_b)}(M', \Lambda'\otimes_{\Z_p}A)\xrightarrow{\sim} \underline{\rm Isom}_{(m_a), (s_a)}(M, \Lambda\otimes_{\Z_p}A).
\]}
\end{Remark}

   For the following, we assume in addition that $A$ is local and henselian.
   
    \begin{prop}\label{RaynaudGruson} 
    Suppose that $M$ is a finite free $A$-module and let $(m_a)\subset M^\otimes$.
    Consider the   $A$-scheme
     \[
  T= \underline{\rm Isom}_{(m_a), (s_a)}(M, \Lambda\otimes_{\Z_p}A)
  \]
  which supports a natural $\Gg$-action. Suppose that there exists a set of local $\Z_p$-algebra homomorphisms $\xi: A\to R_\xi$, with $\cap _{\xi}{\rm ker(\xi)}=(0)$, and such that, for every $\xi: A\to R_\xi$, the base change $\xi^*T:=T\times_{S}\Spec(R_\xi)$ is a $\Gg$-torsor over $\Spec(R_\xi)$. 
  Then, $T\to S$ is also a $\Gg$-torsor.
    \end{prop}
    
    \begin{proof}
  The scheme $T$ is affine and $T\to S$ is of finite presentation. The essential difficulty is in showing that $T\to S$ is flat but under our assumptions, this follows from   \cite[Thm. (4.1.2)]{RaynaudGruson}. The fiber of $T\to S$ over the closed point of $S$ is not empty, hence $T\to S$ is also faithfully flat. Now the base change $T\times_ST$ admits a tautological section which gives a
  $\Gg$-isomorphism $T\times_S T\simeq \Gg\times T$. This completes the proof.
    \end{proof}

We will now allow some more general $\Z_p$-algebras $A$:

    \begin{cor}\label{CorRaynaudGruson}
    Set $A=W(R)$, where $R$ satisfies (N).   Suppose that $M$ is a finite projective $A$-module, 
    $(m_a)\subset M^\otimes$, and   $T=\underline{\rm Isom}_{(m_a), (s_a)}(M, \Lambda\otimes_{\Z_p}A)$. Assume that for all $W(k)$-algebra homomorphisms $\ti x: R\to \O_F$, where $F$ runs over all finite extensions of $W(k)_\Q$, the pull-back $T\otimes_A W(\O_F)$ is a $\Gg$-torsor over $W(\O_F)$. Then $T$ is a $\Gg$-torsor over $A$.
    \end{cor}
    
  \begin{proof} We first show the statement when $R$ is in addition complete and local, i.e. it satisfies (CN). Then $W(R)$ is local henselian and the result follows from Proposition \ref{RaynaudGruson} 
  applied to the set of homomorphisms $\xi: A=W(R)\to R_\xi=W(\O_F)$
  given as $\xi=W(\ti x)$. 
  
  We now deal with the general case. Under our assumptions, $A=W(R)$ is flat over $\Z_p$.
  Let $\frakM\subset W(R)$ be a maximal ideal with residue field $k'$. We have $I_R\subset \frakM$, since $W(R)$ is $I_R$-adically complete and separated (\cite[Prop. 3]{Zink}). Let 
$\frakm_R=\frakM/I_R\subset W(R)/I_R=R$ be the corresponding maximal ideal of $R$. Our assumptions on $R$ imply that the residue field $k'$ is a finite extension of $k$.
Suppose that $\hat R$ is the completion of $R\otimes_{W}W(\bar k')$ at $\frakm_R\otimes_{W}W(\bar k')$.  Then $W(\hat R)$ is local and strictly henselian.
Denote by $W(R)^{\rm sh}_{\frakM}$ the strict henselization of the localization $W(R)_\frakM$. 
By Lemma \ref{A1lemma} we have
  \[
   W(R)^{\rm sh}_{\frakM}\subset W(\hat R).
  \]
We also have 
\[
\hat R\subset \prod\nolimits_{\xi: R\to \O_F} \O_F
\]
where the product is over all $\xi: R\to \O_F$ that factor through $\hat R$. 
By Proposition \ref{RaynaudGruson} applied to $R_\xi=W(\O_F)$,  the base change $T\otimes_{W( R)}W(R)^{\rm sh}_\frakM$   is 
  a $\Gg$-torsor.  By descent, so is the base change $T\otimes_{W(R)}W(R)_\frakM$ over 
$W(R)_\frakM$.\quash{We are also given that $T[1/p]$ is a $G=\Gg[1/p]$-torsor. 
Consider the ring $W(R)/pW(R)$; the natural map $W(R)/pW(R)\to R/pR$ is surjective 
with nilpotent kernel $I_R/(pW(R)\cap I_R)=I_R/pI_R=I_R/I_R^2$ (see \cite[(3)]{Zink} for the last equality). The spectrum of $W(R)/pW(R)$ agrees with that of $R/pR$.  Since $T[1/p]$ is flat over $W(R)[1/p]$, and}
Since this is true for all maximal ideals $\frakM\subset W(R)$, it follows that $T$ is flat over $W(R)$. The result now   follows as in the proof of Proposition \ref{RaynaudGruson}.
\end{proof}
 
 \begin{Remark}
 {\rm When $R$ satisfies (CN), Corollary \ref{CorRaynaudGruson} also holds with $W(R)$, $W(\O_F)$, replaced by $\hat W(R)$, $\hat W(\O_F)$ respectively.   
 
 }
 \end{Remark}
  \end{para}

\begin{para}
Set $D=\Spec(W(k)\lps u\rps)$, $D^\times=\Spec(W(k)\lps u\rps)-\{(p,u)\}$, with $k$ perfect.  We will use the following purity result of Ansch\"utz:

\begin{thm}\label{pure} (\cite[Theorem 8.4]{An})
Assume $\Gg$ is parahoric. Then, every $\Gg$-torsor over $D^\times$ is trivial.
\end{thm}

\begin{Remark}
{\rm This purity property was previously shown (\cite[Prop. 1.4.3]{KP}) for $k=\bar{\mathbb F}_p$ and all parahoric group schemes $\Gg$ with $G=\Gg\otimes_{\Z_p}\Q_p$ that splits over a tamely ramified extension of $\Q_p$ and has no factors of type $E_8$ (\cite[Prop. 1.4.3]{KP}). This is the only case needed for the proofs of Theorem \ref{introThm2} and Corollary \ref{indThm}. The result   fails for most smooth affine group schemes over $\Z_p$ with reductive generic fiber.

}
\end{Remark}
\end{para}

\ve

 \section{Displays with $\Gg$-structure}\label{s1}
 
 In this section, we define $(\Gg, \M)$-displays and give some basic properties. We also define and study the notion of a locally universal $(\Gg, \M)$-display. Recall $(\Gg, \M)$ is an integral Shimura pair; in particular $\Gg$ is parahoric.

 \subsection{The construction of the modification}

\begin{para}
This subsection contains the main construction needed for the definition of a $(\Gg, \M)$-display.
We assume $R$ is a $p$-adic flat $\O_E$-algebra.  Set $A=W(R)$. 
(If, in addition, $R$ is complete local Noetherian and $p\geq 3$, there is an obvious variant with $A=\hat W(R)$.)

\begin{propconstr}\label{MoPro} Assume that $(\Gg,\M)$ is of integral local Hodge type. 
There is a functor
\[
(\calP, q)\mapsto (\calP, \calQ, \alpha)
\]
from the groupoid of pairs $(\calP, q)$ of a $\Gg$-torsor $\P$ over $\Spec(A)$ together with a $\Gg$-equivariant morphism $q:\P\otimes_{A}R\to \M$, to the groupoid of triples $(\calQ', \calQ, \alpha)$ of two $\Gg$-torsors $\calQ'$, $\calQ$ over $\Spec(A)$ and a $G$-equivariant isomorphism 
\begin{equation}\label{eqModify}
\alpha: \calQ\times_{\Spec(A)}\Spec(A[1/p])\xrightarrow{\sim} (\phi^*\calQ')\times_{\Spec(A)}\Spec(A[1/p])
\end{equation}
over $\Spec(A[1/p])$.
\end{propconstr}

We will also see that there are natural base change transformations for $R\to R'$. Also, the functor is constructed using a choice of an integral local Hodge embedding,
but, up to natural isomorphism, is independent of this choice, see Remark \ref{rem139}.

The isomorphism $\alpha$ allows us to think of $\calQ$ as a ``modification" of $\phi^*\P$ along $p=0$; this modification
 is ``bounded by $\calM$". 
The construction of $(\calP, q)\mapsto (\calQ, \alpha)$ occupies most of this subsection. The main point is
the construction  of a functorial map
\[
 \fraki_\Gg(R): \M(R)\to {\rm Gr}_{\Gg }(R)
\]
(see Proposition \ref{ModifProp}), where ${\rm Gr}_{\Gg }(R)$ is as below. 
\end{para}

\begin{para}

If $X$ is a scheme over $\Spec(A)$, we will write, for simplicity, $X[1/p]$ instead of  $X\times_{\Spec(A)}\Spec(A[1/p])$.

For $A=W(R)$, we will consider the 
 set  
 \[
 {\rm Gr}_{\Gg }(R)=\{(\calQ, \alpha)\}
 \]
 of isomorphism classes of pairs $(\calQ, \alpha)$ of
 \begin{itemize}
 
 \item a $\Gg$-torsor $\calQ$ over $\Spec(A)$,
 
 \item a $\Gg$-trivialization $\alpha: \calQ[1/p]\xrightarrow{\sim} \Gg[1/p]$ of $\calQ[1/p]$ over $\Spec(A[1/p])$. 
 \end{itemize}

 The group $G(A[1/p])$ acts 
 on  ${\rm Gr}_{\Gg }(R)$ by $g\cdot (\calQ, \alpha)=(\calQ, g\cdot \alpha)$. 
 Since $p$ is not a zero divisor in $A$, pairs $(\calQ, \alpha)$ as above form a discrete groupoid.
 
 If, in addition, $R$ is complete local with algebraically closed residue field, then $W(R)$ is local strictly henselian, and
 \[
 {\rm Gr}_{\Gg }(R)\cong G(W(R)[1/p])/\Gg(W(R)).
 \]
The set ${\rm Gr}_{\Gg }(R)$ resembles the set of $R$-points
of an affine Grassmannian of some sort.

\end{para}

\begin{para}

Let $R$ be a $p$-adic flat $\Z_p$-algebra and set $A=W(R)$. 
Since $\phi\circ V=p\cdot {\rm id}_{W(R)}$, we have $\phi(I_R)W(R)= p W(R)$. Hence, we obtain a ring homomorphism
\[
\bar\phi: R= W(R)/ I_R\xrightarrow{\phi} A/pA
\]
induced by the Frobenius $\phi: A\to A$. 
\end{para}

\begin{para}\label{para416}
We first discuss the case $\Gg=\GL(\Lambda)$ and $\calP$ is trivial.

Set $M=\Lambda\otimes_{\Z_p} A$. Let
$\calF\subset \Lambda\otimes_{\Z_p} R=M \otimes_{A} R$ be the $R$-locally direct summand
which corresponds to an $R$-valued point in the Grassmannian ${\rm Gr}(d, \Lambda)$. Set 
 \[
 M_1=({\rm id}_M\otimes {\rm gh}_1)^{-1}(\calF)
 \]
  so that 
 \[
 I_RM\subset M_1\subset M,
 \]
  and take
 $\xU$ to be the image of the map $\phi^*M_1\to \phi^*M$ induced by $M_1\hookrightarrow M$. 

 The quotient $(\Lambda\otimes_{\Z_p}R)/\calF$ is $R$-projective and $W(R)$ is $I_R$-adically complete. By lifting idempotents we can see that, locally on $R$, we can write 
\[
M=\Lambda\otimes_{\Z_p}A=L\oplus T,
\]
with $L$ and $T$ finite projective $A=W(R)$-modules such that $\calF$ is the image of $L\oplus I_RT$
under ${\rm id}_{\Lambda}\otimes {\rm gh}_1: \Lambda\otimes_{\Z_p}A\to \Lambda\otimes_{\Z_p}R$.
(For more details, see the proof of \cite[Lemma 2]{Zink}.)  
Then  
\[
\xU=\phi^*(L)\oplus p\phi^*(T)\subset \phi^*(M)\cong\Lambda\otimes_{\Z_p}A
\]
so that
\[
p\Lambda\otimes_{\Z_p}A\subset \xU\subset \Lambda\otimes_{\Z_p}A.
\]
The module $\xU$ has also the following description:
Base change $\calF\subset \Lambda\otimes_{\Z_p} R$ via $\bar\phi:  R\to A/pA$ to obtain an $A/pA$-submodule
\[
\bar\phi^* \calF\subset \Lambda\otimes_{\Z_p} A/pA
\]
which is locally an $A/pA$-direct summand. Then $\xU$ is the inverse image of $\bar\phi^* \F$ under the reduction 
$\Lambda\otimes_{\Z_p}A\to \Lambda\otimes_{\Z_p}A/pA$.

The $A$-module $\xU$ gives a $\GL(\Lambda)$-torsor $\calQ_{\GL}=\underline{\rm Isom}(\xU, \Lambda_A)$
over $A$,  together with a trivialization $\alpha$ of $\calQ_\GL[1/p]$ over $A[1/p]$. 
Sending
\[
\calF\mapsto (\calQ_{\GL}, \alpha)
\]
 gives a functorial (in $R$) map
\[
\fraki_\GL: {\rm Gr}(d, \Lambda)(R)\to  {\rm Gr}_{\GL(\Lambda) }(R).
\]
This satisfies 
\[
\fraki_{\GL}({\rm gh}_1(g)\cdot x)= \phi(g)\cdot \fraki_{\GL}(x),
\]
 for $g\in \GL(\Lambda)(A)$, $x\in {\rm Gr}(d, \Lambda)(R)$.
 
\begin{para} We will now explain how to 
 extend the construction above from $\GL(\Lambda)$ to a parahoric $\Gg\subset \GL(\Lambda)$.
 We assume that $R$ is a $p$-adic flat $\O_E$-algebra. Recall $A=W(R)$.

\begin{prop}\label{ModifProp}
Suppose that $(\Gg, \calM)$ is of integral local Hodge type.
There are functorial (in $R$) maps
 \[
 \fraki_\Gg(R): \M(R)\to {\rm Gr}_{\Gg }(R)
  \]
which satisfy 
 \begin{equation}\label{conjphi}
\fraki_{\Gg}({\rm gh}_1(g)\cdot x)=\phi(g)\cdot \fraki_{\Gg}(x),
 \end{equation}
 for $g\in \Gg(A)$, $x\in \calM(R)$.  
 \end{prop}

\begin{proof} We choose a integral local Hodge embedding $\iota: \Gg\hookrightarrow \GL(\Lambda)$ which induces
$\iota_*: \calM\to {\rm Gr}(d, \Lambda)_{\O_E}$.
Let $a\in \calM(R)$ be an $R$-valued point of $\calM$.  It will be enough to give $ \fraki_\Gg(R)$ for $R=\hat B$, the $p$-adic completion of $B$, where
$\Spec(B)\subset \calM$ varies over affine charts of $\calM$ and $a$ the tautological point. Recall that $\calM$ is normal, flat and proper
over $\O_E$. So, we can assume that $R$ satisfies (N). The image $\iota_*(a)$ is an $R$-valued 
point of the Grassmannian ${\rm Gr}(d, \Lambda)_{\O_E}$ and is given by a locally direct summand
\[
\calF=\calF_a\subset \Lambda\otimes_{\Z_p} R
\]
In what follows, we will omit $a$ from the notation.
The construction above for $\GL(\Lambda)$, applied to $\F$, gives
$
p\Lambda\otimes_{\Z_p}A\subset \xU\subset \Lambda\otimes_{\Z_p}A.
$
Notice that $\xU^\otimes\subset \xU^{\otimes}[1/p]=\Lambda^\otimes_A[1/p]$. We have 
$s_a\otimes 1\in \Lambda^\otimes_A\subset \Lambda^\otimes_A[1/p]$.

\begin{prop}\label{propQconstr}

a) The tensors $s_a\otimes 1 $  belong to $\xU^\otimes$.

b) The scheme of isomorphisms that respect the tensors
\[
\calQ:=\underline{\rm Isom}_{s_a\otimes 1, s_a\otimes 1 }(\xU, \Lambda_A)
\]
is a $\Gg$-torsor over $\Spec(A)$.

c) Since  $\xU[1/p]=\Lambda_A[1/p]$, we have a trivialization
\[
\alpha: \calQ[1/p]\xrightarrow{\ \simeq\ } G_{A[1/p]}.
\]

\end{prop}

\begin{proof}  Using purity for $\Gg$ (Theorem \ref{pure}), we see that the proof of \cite[Lemma 3.2.9]{KP} goes through in our situation and gives a) and b) after base-changing to ${\mathcal O}_K$, for all $\xi: R\to {\mathcal O}_K$. By Proposition \ref{RaynaudGruson} and Corollary \ref{CorRaynaudGruson}
this now implies parts a) and b), cf. the proof of \cite[Cor. 3.2.11]{KP}.
(This uses that $R$ satisfies (N), in particular that it is normal.) 
Part c) is easy.
 \end{proof}
 \smallskip
 
 The proof of Proposition \ref{ModifProp} now follows from
 Proposition \ref{propQconstr} above: Indeed, we set $\fraki_\Gg(a)=(\calQ, \alpha)$, with $\calQ$ and $\alpha$ as above.
 This gives the desired map.
 \end{proof}
 
 \end{para}
   
    \begin{Remark}\label{indpara}
{\rm The maps $\fraki_\Gg(R)$ are independent of the embedding  $\iota$. To see this suppose that $\iota': \Gg\hookrightarrow \GL(\Lambda')$ is another integral local Hodge embedding which gives $\iota'_*: \M\to {\rm Gr}(d', \Lambda')_{\O_E}$. We can consider the product 
 \[
 \iota''=\iota\times \iota': \Gg\xrightarrow{\Delta} \Gg\times \Gg\hookrightarrow \GL(\Lambda)\times  \GL(\Lambda')\subset  \GL(\Lambda\oplus \Lambda').
 \]
This induces
 \[
\iota''_*:  \M\xrightarrow{\Delta }\M\times_{\O_E}\M\to {\rm Gr} (d, \Lambda)\times_{\O_E} {\rm Gr}(d', \Lambda')\subset {\rm Gr}(d+d', \Lambda\oplus \Lambda')_{\O_E}.
 \]Consider $a\in \calM(R)$ and set $\F=\F_a$, $\F'=\F'_a$, $\F''=\F''_a$, for the submodules which correspond to the points $\iota_*(a)$, $\iota'_*(a)$, $\iota''_*(a)$, in the Grassmannians.
 By the construction, we have $\calF''=\calF\oplus \calF'$, $\xU''=\xU\oplus \xU'$, and the projections give $\xU''\to \xU$, $\xU'\to \xU$. These maps induce 
 $\Gg$-equivariant morphisms $\calQ\to \calQ''$ and $\calQ'\to \calQ''$ which identify these $\Gg$-torsors.}
 \end{Remark}

\end{para}

 \begin{para}\label{qMlocTor}
 
 {\sl Proof of \ref{MoPro}.}
 
 We can now give the construction of the modification. We assume that $(\Gg,\M)$ is of integral local Hodge type.
 We choose a integral local Hodge embedding $\iota: \Gg\hookrightarrow \GL(\Lambda)$ which induces
$\iota_*: \calM\to {\rm Gr}(d, \Lambda)_{\O_E}$. Suppose that  $\calP$ is a $\Gg$-torsor over $A=W(R)$ given together with a $\Gg$-equivariant morphism
\[
q: \calP\otimes_{A}R\to \M.
\]
The case that $\P$ is a trivial $\Gg$-torsor follows immediately from the proof of Proposition \ref{ModifProp}:
If $s$ is section of the $\Gg$-torsor $\calP$ then the composition $a(s)=q\cdot (s\otimes_AR)$ is an $R$-valued point of $\M$. 
The proof of Proposition \ref{ModifProp} gives a pair $\fraki_\Gg(a(s))$ of a 
new $\Gg$-torsor $\calQ_s$ with a trivialization $\alpha_s$ of $\calQ_s[1/p]$.

Let us discuss the general case: Note that $A=W(R)$ is complete and separated in the $I_R$-adic topology (\cite[Prop. 3]{Zink}), so $(W(R), I_R)$ is a henselian pair. Hence, since $\Gg$ is smooth, $\calP$ acquires a section over $W(R')$, where $R'$ is an \'etale cover of $R$ (cf. \cite[Prop. (B.0.2)]{BP}); we can make sure that $R'$ is also $p$-adic. The construction above and the equivariance property (\ref{conjphi}) combined with descent as in \cite[1.3]{Zink}, shows that $\calP$ together with $q$,
gives a $\Gg$-torsor $\calQ$ together with an isomorphism of $G$-torsors 
\begin{equation}\label{eqModify}
\alpha: \calQ[1/p]\xrightarrow{\sim} \phi^*\calP[1/p]
\end{equation}
over $A[1/p]$. 
 Explicitly, if $\P(\iota)=M$ is the corresponding finite projective $W(R)$-module with tensors $(m_a)$
 (see Proposition \ref{torsorRep}), then $\iota_*\cdot q$ gives 
 \[
 \calF\subset M\otimes_{W(R)}R.
 \]
 Set $M_1=({\rm id}_M\otimes {\rm gh}_1)^{-1}(\calF)$ so that $I_RM\subset M_1\subset M$ and take
 $\xU$ to be the image of the map $\phi^*M_1\to \phi^*M$ induced by $M_1\hookrightarrow M$. As in \S\ref{para416}, we
 obtain $p\phi^*(M)\subset \xU\subset \phi^*(M)$. Then,  as in Proposition \ref{propQconstr}, we have $\phi^*m_a\in \xU^{\otimes}$ and 
 \[
 \calQ=\underline{\rm Isom}_{(\phi^* m_a ), (s_a)}(\xU, \Lambda_A).
 \]\endproof

\begin{Remark}\label{rem139}
{\rm Note that in the above, the pair $(\calQ, \alpha: \calQ[1/p]\xrightarrow{\sim} \phi^*\calP[1/p])$ only depends on $\calP$, $\M$ and $q$  and is independent of $\iota$, up to unique  isomorphism; this follows from \S\ref{indpara}.  In fact, the argument gives that the functor of \ref{MoPro} is, up to natural isomorphism, independent of the choice of the integral local Hodge embedding $\iota$.}
\end{Remark}
  
 \end{para}

 \begin{Remark}\label{conjEmb}
 {\rm a) The above applies to $\M=\Mloc$, where $\Mloc=\Mloc{(\Gg, \{\mu\})}$
 are the local models of \cite{PZ}, when $\Gg(\Z_p)$ is connected parahoric, $p\nmid |\pi_1(G_{\rm der}(\bar\Q_p))|$,
 and there is a local Hodge embedding $\iota: G\hookrightarrow \GL_n$ with $\Gm\subset \iota(G)$.
 This follows from Proposition \ref{propLHT}.

 b) We conjecture that, for the local models $\calM=\Mloc$,  the maps $\fraki_\Gg$ exist in general (without assuming 
any Hodge type condition) and that Proposition \ref{MoPro} still holds:
 
  More precisely, suppose that $\Mloc=\Mloc{(\Gg, \{\mu\})}$
 is the local model over $\Spec(\O_E)$ conjectured to exist by Scholze \cite[Conjecture 21.4.1]{Schber}.
 Then, we expect that there are 
canonical functorial injective maps
 \[
\fraki_{\Gg,\mu}(R): \Mloc(R)\to {\rm Gr}_{\Gg }(R),
 \]
  for $R$ $p$-adic flat over $\O_E$, which also satisfy the equivariance
 property (\ref{conjphi}). 
 
 (One can speculate that the maps $\fraki_{\Gg,\mu}$ come from natural maps
 \[
  \Mloc\to {\rm Gr}^\Prism_\Gg,
 \]
 where ${\rm Gr}^\Prism_\Gg$ is a ``prismatic affine Grassmannian" for $\Gg$.)
 }
 \end{Remark}

\subsection{$(\Gg, \calM)$-displays.} 

\begin{para}
We now give the definition of a $(\Gg, \M)$-display over $R$, where 
$R$ is a $p$-adic flat $\O_E$-algebra.  We assume that $(\Gg, \M)$ is of integral local Hodge type.

 \begin{Definition}\label{DEFDISPLAY}
 A $(\Gg, \M)$-display over $R$ is a triple $\calD=(\calP, q, \Psi)$ of:
  
  \begin{itemize}
  \item  A $\Gg$-torsor $\calP$  over $W(R)$,
  
  \item  a $\Gg$-equivariant morphism $q:\calP\otimes_{W(R)}R\to \M$ over $\O_E$,
     
  \item a $\Gg$-isomorphism $\Psi: \calQ\xrightarrow{\sim} \calP$  where $\calQ$ is the 
  $\Gg$-torsor over $W(R)$ which is the modification of $\phi^*\calP$ given by $(\calP, q)$ in \ref{MoPro}.
  \end{itemize}

 \end{Definition}
 
Recall that, by   \ref{MoPro}, the pair $(\calP, q)$ gives   $\calQ$ together with an isomorphism (\ref{eqModify})
\[
\alpha: \calQ[1/p]\xrightarrow{\sim}\phi^*\calP[1/p].
\]
Composing $\alpha^{-1}$ with $\Psi[1/p]: \calQ[1/p]\xrightarrow{\sim} \calP[1/p]$
gives an isomorphism of $\Gg$-torsors over $W(R)[1/p]$
\[
\Phi: \phi^*\calP[1/p]\xrightarrow{\sim} \calP[1/p]
\]
which is also attached to the $(\Gg, \calM)$-display $\calD=(\calP, q, \Psi)$.
\end{para}

\begin{para}\label{remark323}{
 Suppose $(\Gg, \{\mu\})$ is a pair of a parahoric group scheme and a conjugacy class of a minuscule cocharacter $\mu$ of
 $\Gg_{\bar\Q_p}$.  Assume Scholze's conjecture \cite[Conj. 21.4.1]{Schber} on the existence 
 of the local model $\Mloc=\Mloc{(\Gg, \{\mu\})}$.  
  
 Suppose that either $(\Gg, \Mloc)$ is of integral local Hodge type, or more generally, that  
 the conjecture of Remark \ref{conjEmb} (b) is true for  $(\Gg, \{\mu\})$. Then the construction of the modification $\calQ$ 
 from $(\calP, q)$ goes through and the definition of a $(\Gg, \Mloc)$-display makes sense. In this case, instead of  ``$(\Gg, \Mloc)$-display", we will simply say ``$(\Gg, \{\mu\})$-display".
}
 \end{para}
 
 \begin{para}\label{global}
 Assume now that $\frakX$ is a $p$-adic formal scheme which is flat and formally of finite type 
 over ${\rm Spf}(\O_E)$. By Zink's Witt vector descent \cite[\S1.3, Lemma 30]{Zink}, there is a sheaf of rings  
 $W(\O_{\frakX})$ over $ \frakX $ such that for every open affine formal subscheme ${\rm Spf}(R)\subset \frakX$, we have $\Gamma(\Spf(R), W(\O_{\frakX}))=W(R)$. It now makes sense to give the natural extension of the above definition:  
 A
$(\Gg, \M)$-display over $\frakX$ is a triple $\calD=(\calP, q, \Psi)$ with the data 
 $\calP$, $q$,  $\Psi$, as above given over $W(\O_{\frakX})$.
 \end{para}

\subsection{Dieudonn\'e $(\Gg, \M)$-displays.}
  \begin{para} 
  
We now assume that $p$ is odd and that $R$ is in addition complete local Noetherian. We continue to suppose that $(\Gg, \M)$ is of integral local Hodge type.
 
 \begin{Definition}
 A Dieudonn\'e $(\Gg, \M)$-display over $R$ is a triple $\calD=(\calP, q, \Psi)$ of
 a $\Gg$-torsor $\calP$  over $\hat W(R)$, a $\Gg$-equivariant morphism \[
 q:\calP\otimes_{\hat W(R)}R\to \M
 \] over $\O_E$, and
  a $\Gg$-isomorphism $\Psi: \calQ\xrightarrow{\sim} \calP$  where $\calQ$ is the 
  $\Gg$-torsor over $\hat W(R)$ induced by $q$ in \ref{MoPro} (applied to $A=\hat W(R)$).
  \end{Definition}
  
    \end{para}
    
    \begin{para}\label{WvsWhat}
  Note that a Dieudonn\'e $(\Gg, \M)$-display over $R$, produces a $(\Gg, \M)$-display over $R$
  by base change along the inclusion $\hat W(R)\hookrightarrow W(R)$. Proposition \ref{Qconnection} holds for Dieudonn\'e $(\Gg, \M)$-displays over $R$ with $W(R)$ and $W(R/\fraka_R)$ replaced by $\hat W(R)$ and $\hat W(R/\fraka_R)$.
Most of the notions defined for $(\Gg, \M)$-displays, for example, the notion of rigid section and of locally universal, have obvious analogues for Dieudonn\'e $(\Gg, \M)$-displays. The obvious variant of Proposition \ref{locally universalProp2} for Dieudonn\'e displays holds. We will sometimes refer to these statements when we are really using 
their $\hat W$-variants instead, without explicitly alerting the reader.   
   \end{para}

   \subsection{Relation with Zink's displays.}\label{compaZink}
   
   \begin{para}
   
    In the next paragraph, we relate our notion of a (Dieudonn\'e) $(\Gg, \M)$-display over $R$ to the classical notion of Zink
   (\cite{Zink}, \cite{ZinkCFT}). This involves the use of the integral local Hodge embedding. To fix ideas, we only discuss Dieudonn\'e displays and for that we assume $p$ is odd and $R$ is also complete local Noetherian. 
   \end{para}
   
   \begin{para}\label{compaZink2}
    
 Suppose that $\calD=(\calP, q, \Psi)$ is a Dieudonn\'e $(\Gg, \M)$-display over $R$, and that $\iota: \Gg\hookrightarrow \GL(\Lambda)$ is an integral local Hodge embedding. 
 
 We set $M=\calP(\iota)$ which is a finite projective $\hat W(R)$-module of rank equal to ${\rm rank}_{\Z_p}(\Lambda)$. Then $q$ composed with $\iota$ produces $\calP\otimes_{\hat W(R)}R\to {\rm Gr}(d, \Lambda)_{\O_E}$. This morphism gives a locally direct summand
 $\calF\subset M\otimes_{\hat W(R)}R$. Let us denote by $M_1$  the inverse image of $\calF$ under the reduction homomorphism $M\to M\otimes_{\hat W(R)}R$. This is a $\hat W(R)$-module with $\hat I_R M\subset M_1\subset M$. 
 As in \cite[3.1.4]{KP}, we denote by $\tilde M_1$ the image of $\phi^*M_1$ under the $\hat W(R)$-homomorphism
 \[
 \phi^*M_1\to \phi^*M
 \]
 which is induced by the inclusion $M_1\hookrightarrow M$.  We have
 \[
p\phi^*M\subset \tilde M_1\subset \phi^*M.
 \]
 Then, our construction of $\calQ$ from $(\calP, q)$ 
 implies an identification $\tilde M_1=\calQ(\iota)$. (See \S\ref{qMlocTor},  $\tilde M_1$ corresponds to $U$ there.) The isomorphism $\Psi$ gives an isomorphism 
 \[
 \Psi(\iota): \tilde M_1\xrightarrow{\sim} M.
 \]
 We denote $\Phi_1: \phi^*M_1\to \tilde M_1\xrightarrow{\sim} M$ the composition.
 Note here that we also have $\Phi=\Phi(\iota): \phi^*M[1/p]\xrightarrow{\sim} M[1/p]$ given by $\phi^*M[1/p]=\tilde M_1[1/p]$ composed with $\Psi(\iota)[1/p]$. In fact, we see that
 \[
 M\subset \Phi(\phi^*M) \subset \frac{1}{p}M.
 \]
  We can now consider the triple $(M, M_1, \Phi_1)$.  By \cite[3.1.3, Lemma 3.1.5]{KP}, this triple defines a Dieudonn\'e display over $R$ in the sense of Zink \cite{ZinkCFT}.
 (Recall that we assume that $R$ is flat over $\Z_p$.) 
 
 It is useful to compare the notations here to those in display theory (e.g. \cite{ZinkCFT}): $\Phi_1$ here corresponds  to the linear map $F_1^\#$ induced by the semi-linear map denoted there by $F_1=V^{-1}$. The 
 linear map $F^\#$ that corresponds to the (semi-linear) Frobenius $F$
 of Zink's display is given here as\footnote{Note $\Phi=p^{-1}\cdot F^\sharp$ says that the Frobenius of the classical theory is here scaled by $p^{-1}$,  cf. \cite[p. 158]{Schber}.}
 \[
F^\#=p\cdot \Phi: \phi^*M\to M,
 \]
 so $\Phi=p^{-1}\cdot F^\#$, $\Phi_1=F^\#$.

  We will denote the Dieudonn\'e display 
  \[
  (M, M_1, F_1, F)
  \]
   by $\calD(\iota)$, since it is derived from $\calD$ and $\iota: \Gg\hookrightarrow \GL(\Lambda)$. By \cite{ZinkCFT}, there is a corresponding $p$-divisible group $\sG_R={\rm BT}(\calD(\iota))$ over $R$. By \cite{ZinkCFT}, \cite[Theorem B]{Lau},  
  \[
 (M, M_1, F_1, F)= \DD(\sG_R)(\hat W(R))
  \]
  where $\DD(\sG_R)$ denotes the (filtered) \emph{covariant} Dieudonn\'e crystal of $\sG_R$
  (the Frobenius is given by $F$.)
  Then, the tangent space of $\sG_R$ is canonically identified with the $R$-module $M/M_1$.
  Since $F$ is determined by $F_1$, we will write this as
  \[
  (M, M_1, F_1)
  \]
  in what follows.
  
    \end{para}

 \subsection{Rigidity and locally universal displays.}
 
  \begin{para}
  In this subsection, we assume until further notice that $R$ satisfies (CN), in particular it is normal and complete local Noetherian.
   
 We also continue to assume that $(\Gg, \calM)$ is of integral local Hodge type  and that
  $\calD=(\calP, q, \Psi)$ is a $(\Gg, \M)$-display over $R$.
  Under our assumptions on $R$, the $\Gg$-torsors $\calP$, $\calQ$ over $W(R)$ are trivial.
  We denote by $\calD_0=(\calP_0, q_0, \Psi_0)$ the display over $k$ obtained by 
  reduction of $\calD=(\calP, q, \Psi)$ 
 modulo $\frakm_R$.  
 \end{para}
    
 \begin{para}

Denote by $\frak m_R$ the maximal ideal of $R$. Set $\fraka_R=\frakm_R^2+(\pi_E)$,
where $\pi_E$ is a uniformizer of $\O_E$.
Observe that the Frobenius $\phi$ factors as 
\[
W(R/\fraka_R)\to W(k)\xrightarrow{\phi} W(k)\to W(R/\fraka_R).
\]

\begin{prop}\label{Qconnection}
There is a canonical isomorphism of $\Gg$-torsors
\begin{equation}\label{canonicaliso}
c: \calQ\otimes_{ W(R)}W(R/\fraka_R)\xrightarrow{\sim} \calQ_0\otimes_{W(k)}W(R/\fraka_R),
\end{equation}
where $\calQ$, resp. $\calQ_0$, is the $\Gg$-torsor for the display $\calD$, resp. $\calD_0$, as in
Definition \ref{DEFDISPLAY}. 
\end{prop}
\begin{proof}
Note that \cite[Lemma 3.1.9]{KP} gives the corresponding statement for classical displays and $\O_E=\Z_p$. Recall that 
the $\Gg$-torsor $\calQ$ is given as in the paragraph \ref{qMlocTor}, using the corresponding $\F\subset M\otimes_{W(R)}R$ and $\xU$.   We denote by $M_0$, $M_{1,0}$, $\xUo$, the modules associated to the display $\calD_0$ over $k$ obtained from $\calD$ by base change, as above.  Let us write $M=L\oplus T$, with $L$ and $T$ free $W(R)$-modules, such that $\F$ is given by $L$ modulo $I_{R}$.
Then $M_1=L\oplus I_RT$, so $\F$ gives the filtration  $I_RM\subset M_1\subset M$.
Then, as in the proof of \cite[Lemma 3.1.9]{KP}
\[
\xU\otimes_{W(R)}W(R/\fraka_R)=\phi^*(L)\oplus (p\otimes \phi^*(T))
\]
and $\phi^*(L)\simeq \phi^*(L_0)\otimes_{W(k)}W(R/\fraka_R)$, $\phi^*(T)\simeq \phi^*(T_0)\otimes_{W(k)}W(R/\fraka_R)$. Here, 
 we write $p\otimes -$ for $p\Z_p\otimes_{\Z_p} -$, and we have $L_0=L\otimes_{W(R)}W(k)$, $T_0=T\otimes_{W(R)}W(k)$.
This gives the isomorphism of \cite[Lemma 3.1.9]{KP}
\begin{equation}\label{connU}
c: \xU\otimes_{W(R)}W(R/\fraka_R)\xrightarrow{\simeq} {\xUo}\otimes_{W(k)}W(R/\fraka_R)
\end{equation}
which is independent of the choice of the normal decomposition $L\oplus T$. 
Using Proposition
\ref{torsorRep} we see that it is enough to show that $c$ preserves the tensors that correspond to $s_a$. As in \S\ref{para321},  we can assume these are of
 the form $s_a:\Lambda^{\otimes n}\to \Lambda^{\otimes n}$, with $n\geq 1$. These induce $m_a: M^{\otimes n} \to M^{\otimes n}$ which induce $u_a:=\phi^*(m_a) : \xU^{\otimes n}\to \xU^{\otimes n}$. In this situation, we have  to show that $c$ is compatible with $u_a$ in the sense that
the obvious diagram 
\begin{equation*}
\xymatrix{ \xU^{\otimes n}\otimes_{W(R)}W(R/\fraka_R)\ar[r]^{\ \ u_a\otimes 1\ \ }\ar[d]_{c}^{\simeq} & \xU^{\otimes n}\otimes_{W(R)}W(R/\fraka_R)\ar[d]^{\simeq}_{c}\\
   \xUo^{\otimes n}\otimes_{W(k)}W(R/\fraka_R)\ar[r]^{\ \ (u_a)_0\otimes 1\ \ } &\xUo^{\otimes n}\otimes_{W(k)}W(R/\fraka_R).}
\end{equation*}
is commutative. We start by giving a description of $\xU^{\otimes n}$. 

The filtration $I_RM\subset M_1\subset M$ induces a filtration on $M^{\otimes n}$ and we are interested in the $W(R)$-submodule
\begin{equation}\label{Ndecomp}
N:=\oplus_{j=0}^n (L^{\otimes n-j}\oplus I^j_R T^{\otimes j})\subset M^{\otimes n}
\end{equation}
which is the image of $M_1^{\otimes n}\to M^{\otimes n}$ (and so is independent of the normal decomposition $M=L\oplus T$.)
 (Note that when $n\geq 2$, the map $M_1^{\otimes n}\to M^{\otimes n}$ is rarely injective. Also note that   $I_R^m=p^{m-1}I_R$, if $m\geq 1$ (\cite[(7)]{Zink}), and $\phi(I_R^m)=(p^m)$.) The image of the map $\phi^*N\to \phi^*M^{\otimes n}$ induced by the inclusion $N\hookrightarrow M^{\otimes n}$, is $\xU^{\otimes n}$.

Next, we show that $m_a :  M^{\otimes n}\to M^{\otimes n}$   preserves $N$, i.e.  restricts to   $m_a: N\to N$. The rough idea is that this should hold because the point in the Grassmannian corresponding to $\F$ is in the closure of the $G$-orbit of the cocharacter $\mu$ and the tensors $m_a$ are fixed by the group $G$. More precisely, we show  $m_a(N)\subset N$ using ``restriction to $\O_F$-points"  
as follows: Suppose $\xi: R\to \O_F$ is a local $\O_E$-homomorphism. 
Denote by  $M_{\O_F}$ and $N_{\O_F}$    the corresponding
$W(\O_F)$-modules for the display over $\O_F$ obtained by base change of $\calD$ by $\xi$. As in the proof of \cite[Lemma 3.2.6]{KP},  we see,
using that $\xi^*\F$ is given by a $G$-cocharacter conjugate to $\mu$,  that $\xi^*m_a$ preserves
$N_{\O_F}\subset M_{\O_F}^{\otimes n}$. We can now deduce that 
$m_a: M^{\otimes n} \to M^{\otimes n}$ preserves $N$:  It is enough to check that
certain elements of $W(R)$ which are given as the coefficients of images of $m_a$ in a basis given by 
the decomposition (\ref{Ndecomp}) lie in $I^m_R=p^{m-1}I_R$, while, by the above, we know that their images in $W(\O_F)$ lie in 
$p^{m-1}I_{\O_F}$, for all such $\xi: R\to \O_F$. But this is true by a simple extension of the argument in the proof of Lemma \ref{dejonglemma}.

Now consider the commutative diagram
\begin{equation*}
\xymatrix{\phi^*N\otimes_{W(R)}W(R/\fraka_R)\ar[r]^\alpha\ar[d]_{c'}^{\simeq} &\xU^{\otimes n}\otimes_{W(R)}W(R/\fraka_R)\ar[r]\ar[d]_{c}^{\simeq} & \phi^*M^{\otimes n}\otimes_{W(R)}W(R/\fraka_R)\ar[d]^{\simeq}_{c''}\\
\phi^*N_0\otimes_{W(k)}W(R/\fraka_R)\ar[r]^{\alpha_0} &\xUo^{\otimes n}\otimes_{W(k)}W(R/\fraka_R)\ar[r] &\phi^*M^{\otimes n}_0\otimes_{W(k)}W(R/\fraka_R).}
\end{equation*}
Here $c'$, $c''$ are the canonical isomorphisms obtained by the factorization $W(R/\fraka_R)\to W(k)\to W(k)\to W(R/\fraka_R)$
of the Frobenius $\phi$ above. We have $\alpha(\phi^*(m_a))=u_a$. The fact that $c$ is compatible with the tensor $u_a$ now follows from the above, the functoriality of $c'$ and the fact that $\alpha$ is surjective.
\end{proof}

\end{para}

   \begin{para} \label{paraPsig} 

Continuing with the same assumptions, we have:
 
 \begin{Definition} 
 A section of $s$ of the $\Gg$-torsor $\calP$ is called rigid in the first order at $\frakm_R$ when, under the isomorphism
 of Proposition \ref{Qconnection},
 \[
 \Psi^{-1}(s)\, {\rm mod}\, W(\fraka_R)=\Psi^{-1}_0(s_0)\otimes 1,
 \] 
 where, again, the subscript $0$   signifies reduction   
 modulo $\frakm_R$.  
 \end{Definition}
 
In other words, we are asking that  the diagram
 \begin{equation*}
\xymatrix{ \calQ\otimes_{W(R)}W(R/\fraka_R)\ar[r]^{\ \Psi\ }_{\simeq}\ar[d]_{c}^{\simeq} & 
\P\otimes_{W(R)}W(R/\fraka_R)\ar[d]_{\Psi_0\cdot c\cdot \Psi^{-1}}^{\simeq}\ar[r]^{(s_{\fraka_R})^{-1}} &\Gg\otimes_{\Z_p}W(R/\fraka_R)\ar[d]_{\rm id}\\
   \calQ_0\otimes_{W(k)}W(R/\fraka_R)\ar[r]^{\ \Psi_0\ }_{\simeq} &\P_0\otimes_{W(k)}W(R/\fraka_R)\ar[r]^{s^{-1}_0} &\Gg\otimes_{\Z_p}W(R/\fraka_R)}
\end{equation*}
commutes. (In this,  we write $s_{\fraka_R}=s\otimes_{W(R)}W(R/\fraka_R)$ for simplicity.)

Given any section $s: \Gg\otimes_{\Z_p}W(R )\to \P$, the composition 
\[
g=s_0^{-1}\cdot (\Psi_0\cdot c\cdot \Psi^{-1})\cdot s_{\fraka_R}: \Gg\otimes_{\Z_p}W(R/\fraka_R)\to \Gg\otimes_{\Z_p}W(R/\fraka_R)
\]
is given by an element $g\in \Gg(W(R/\fraka_R))$ which reduces to the identity in $\Gg(W(k))$, i.e. with $g_0=1$.
Since  $\Gg$ is smooth, $\Gg(W(R))\to \Gg(W(R/\fraka_R))$ is surjective, and we can always find $h\in \Gg(W(R))$
with $h\,{\rm mod}\, W(\fraka_R)=g^{-1}$. Then $h\cdot s$ is rigid in the first order at $\frakm_R$. Hence,
 if for example $k$ is algebraically closed, there is always some section which is rigid in the first order at $\frakm_R$.

 \begin{Remark}\label{rigidCrystal}
{\rm  This notion of ``rigid in the first order" is comparable to a corresponding notion 
for Dieudonn\'e crystals that appears in \cite[Def. 3.31]{R-Z}, see Proposition \ref{Psiconstant} (a) below. In fact, the isomorphism
\[
\theta:= \Psi\cdot c^{-1}\cdot \Psi_0^{-1}: \P_0\otimes_{W(k)}W(R/\fraka_R)\xrightarrow{\sim} \P\otimes_{W(R)}W(R/\fraka_R)
\]
should correspond, in the case of Zink displays, to the trivialization given by the crystalline structure. (Note
that $(\frakm_R/\fraka_R)^2=0$.) Let us remark here that $s$ is rigid in the first order at $\frakm_R$
when we have $\theta(s_0)=s_{\fraka_R}$, a condition we might think of as saying ``$s$ is horizontal with respect to $\theta$ at the closed point of $R$".

}
\end{Remark}

  \begin{Definition}\label{Def145}
  A $(\Gg, \M)$-display $\calD$ over $R$ is \emph{locally universal}, if there is a section $s$ of $\calP$ which is rigid in the first order,
 such that the composition 
 \[
 q\cdot (s\otimes 1): \Spec(R)\to \calP\otimes_{W(R)}R\to \M
 \]
  gives an isomorphism between $R$ and the completion of the local 
  ring of $\M\otimes_{\Z_p}W(k)$ at the image of the closed point of $\Spec(R)$. 
  \end{Definition}

\begin{prop}\label{locally universalProp2}
Suppose that the $(\Gg, \calM)$-display $\calD=(\calP, q, \Psi)$ over $R$ is locally universal. Then $q: \calP\otimes_{W(R)}R\to \calM$ is formally smooth.
\end{prop}

\begin{proof} The action morphism $m: \Gg\times_{\Spec(\Z_p)}\M\to \M$ is smooth, since $\Gg$ is smooth.
Let $s: \Spec(R)\to \calP\otimes_{W(R)}R$ be a section which is rigid in the first order and is such that
$q\cdot (s\otimes 1)$ identifies $R$ with the completion  of the strict Henselization of $\calM$. Since $q$ is $\Gg$-equivariant, 
$\Gg\times _{\Spec(\Z_p)}\Spec(R)\simeq \calP\otimes_{W(R)}R$  given by $s$ identifies $q: \calP\otimes_{W(R)}R\to \calM$ with 
\[
 \Gg\times_{\Spec(\Z_p)}\Spec(R)\xrightarrow{{\rm id}_{\Gg}\times q\cdot (s\otimes 1) } \Gg\times_{\Spec(\Z_p)}\calM\xrightarrow{m} \M.
 \]
This is the composition of formally smooth morphisms, so also formally smooth.
\end{proof}
 \end{para}

    \begin{para}\label{paradeform} We return momentarily to Zink displays.
     We continue with the same assumptions on $\calD=(\P, q, \Psi)$ and $\iota$ as in \S \ref{compaZink2}. 
     In particular, $\calD$ is a Dieudonn\'e display.

  Recall that (e.g. \cite[2.2]{Zink}, \cite[Thm 3]{ZinkCFT}) if $(M, M_1, F_1)$ is a Dieudonn\'e display, then $M$ gives a crystal. In fact, we only need the following consequence: For $\fraka_R=\frakm_R^2+(\pi_E)$
  as above, there is a canonical isomorphism
  \begin{equation}\label{varthetamap}
 \vartheta:  M_0\otimes_{W(k)}\hat W(R/\fraka_R)\xrightarrow{\sim } M\otimes_{\hat W(R)}\hat W(R/\fraka_R).
  \end{equation}
  Using this (together with the main theorem of \cite{ZinkCFT}) one can understand the deformations of the $p$-divisible group $\sG_0$ given by $\calD_0(\iota)=\calD(\iota)\otimes_{\hat W(R)}W(k)$ (\cite[Thm 4]{ZinkCFT}): 
  
  Fix an identification $M_0= \Lambda\otimes_{\Z_p}W(k)$. The $p$-divisible group $\sG_{R/\fraka_R}$ which is given by $\calD(\iota)\otimes_{\hat W(R)}\hat W(R/\fraka_R)$ produces a $\Spec(R/\fraka_R)$-valued point of the Grassmannian ${\rm Gr}(d, \Lambda)_k$. This point is given by the submodule of $\Lambda\otimes_{\Z_p}R/\fraka_R$ which is the reduction of
  \[
 \Lambda\otimes_{\Z_p}\hat I_{R/\fraka_R}\subset \vartheta^{-1}(M_{1, R/\fraka_R})\subset \Lambda\otimes_{\Z_p}\hat W(R/\fraka_R).
  \]
  modulo $\hat I_{R/\fraka_R}$. 
  Conversely, every $\Spec(R/\fraka_R)$-valued point of the Grassmannian
 which lifts the $k$-point corresponding to $M_{1, 0}$ comes as above from a unique
 deformation $\sG_{R/\fraka_R}$ of $\sG_0$ over $R/\fraka_R$. This way we can identify
 the tangent space $\frakT_0{\rm Gr}(d, \Lambda)_k$ of ${\rm Gr}(d, \Lambda)_k$ at $M_{1, 0}$ with the tangent space $\frakT_0$
 of the formal deformation space  of the $p$-divisible group $\sG_0$.
 Here we need a more precise statement about the deformations that lift 
 over $R$ and the corresponding Dieudonn\'e $\Gg$-displays which we will give next.
 \end{para}
    
   \begin{para}
 We continue with the same assumptions on $\calD=(\P, q, \Psi)$ and $\iota$ as above.
 If $s$ is a section of $\calP$, then $s(\iota)$ is the corresponding frame, i.e. the isomorphism $s(\iota): \Lambda\otimes_{\Z_p}\hat W(R)\xrightarrow{\sim} M=\calP(\iota)$.

\begin{prop}\label{Psiconstant}
a) A section $s$ of $\P$ is rigid in the first order at $\frakm_R$ if and only if 
$\vartheta(s_0(\iota))=s_{\fraka_R}(\iota)$, where $\vartheta$ is the map (\ref{varthetamap}).

b) Suppose $s$ is a section of $\P$ which is rigid in the first order at $\frakm_R$. Let 
\[
\Spec(R/\fraka_R)\xrightarrow{\ } \frakT_0 
\]
be the classifying morphism of $\Spec(R/\fraka_R)$ into the tangent space $\frakT_0$
of the deformation space of $\sG_0$, which is given by the deformation $\sG_{R/\fraka_R}$. Then there is an isomorphism
$ \frakT_0{\rm Gr}(d, n)_k\xrightarrow{\sim} \frakT_0$ over $k$ such that
 \[
  \Spec(R/\fraka_R)\xrightarrow{s_{\frak a_R}\otimes 1} \calP\otimes_{\hat W(R)}R/\fraka_R\xrightarrow{q\otimes_R 1} \M\to {\rm Gr}(d, n)_{\O_E}
 \]
gives, after composing with $ \frakT_0{\rm Gr}(d, n)_k\xrightarrow{\sim} \frakT_0$, the classifying morphism above.
\end{prop}

\begin{proof} It follows directly from the definitions that if $s$ is rigid in the first order at $\frakm_R$, then the trivialization $s(\iota): \Lambda\otimes_{\Z_p}\hat W(R)\xrightarrow{\sim} M$ makes $\Psi_{R}(\iota): \tilde M_1\xrightarrow{\sim} M$ ``constant modulo $\fraka_R$" in the sense of \cite[Definition (3.1.11)]{KP}. (By this we mean that 
we take the identification $\Lambda\otimes_{\Z_p}\hat W(R)\cong M$ which is used in \cite[Definition (3.1.11)]{KP} 
 to be $s(\iota)$.) Both (a) and (b) now follow from the definition of $c$, the construction of the map $\vartheta$ in \cite[2.2]{Zink}, \cite{ZinkCFT}, and the argument in the proof  of \cite[Lemma 3.1.12]{KP}. (This lemma gives part (b) for a universal $(M, M_1, F_1)$.)
\end{proof}

\end{para}
 
 \ve

\section{Crystalline $\Gg$-representations}

In this section, we describe ``$\Gg$-versions'' of objects of integral $p$-adic Hodge theory 
 which can be attached to a $\Gg(\Z_p)$-valued crystalline representation.

  \subsection{}\label{BKrep}
  
  Fix $(\Gg, \M)$ as in \S\ref{ssShimura} of integral local Hodge type. Fix also an integral local Hodge embedding 
 \[
 \iota:\Gg\hookrightarrow \GL(\Lambda) 
  \]
with $\M\to \overline{\iota_*(X_\mu)}\subset {\rm Gr}(g,\Lambda)_{\O_E}$.

Let $F$ be a finite extension of $E$ or of $\breve E$ with residue field $k$. Let 
  \[
  \rho: {\rm Gal}(\bar F/F)\to \Gg(\Z_p) 
  \]
   be a Galois representation. We assume that $\iota\cdot \rho: {\rm Gal}(\bar F/F)\to   \GL(\Lambda[1/p])$ is crystalline. We give three flavors of ``$\Gg$-versions'' of Frobenius modules 
which can be attached to $\rho$ by integral $p$-adic Hodge theory.

\subsection{The Breuil-Kisin $\Gg$-module}

\begin{para}\label{BKtorsor}

Choose a uniformizer $\pi_F$ of $F$ and let $E(u)\in W(k)[u]$ be the Eisenstein polynomial 
 with $E(\pi_F)=0$. Choose also a compatible system of roots $\sqrt[p^n]{\pi_F}$ in $\bar F$.
 The Breuil-Kisin $\Gg$-module attached to $\rho$, is 
by definition, a pair $(\calP_{\rm BK}, \phi_{\calP_{\rm BK}})$ where
\begin{itemize}
\item[$\bullet$] $\calP_{\rm BK}$ is a $\Gg$-torsor over $\frakS=W(k)\lps u\rps$,

\item[$\bullet$] $\phi_{\calP_{\rm BK}}$ is an isomorphism of $\Gg$-torsors
\[
\phi_{\calP_{\rm BK}}:  \phi^*\calP_{\rm BK}[1/E(u)]\xrightarrow{\sim } \calP_{\rm BK}[1/E(u)].
\]
\end{itemize}
(Here, $\phi: \frakS\to \frakS$ is the ring homomorphism which extends the Frobenius 
on $W(k)$ and satisfies $\phi(u)=u^p$.)

It is constructed as follows. (It does depend on the choice of   $\sqrt[p^n]{\pi_F}$, $n\geq 0$.) 

 As in the proof of \cite[Lemma 3.3.5]{KP}, we  write $\O_{\calG}=\varinjlim_{i\in J} \Lambda_i$ with $\Lambda_i\subset \O_{\calG}$ of finite $\Z_p$-rank and $\calG$-stable. The Galois action on $\Lambda$ gives
 actions on $\Lambda_i$ and on $\O_{\Gg}$. We  apply the Breuil-Kisin functor 
 \[
 \frakM: {\rm Rep}_{K}^{{\rm cris}, \circ}\to {\rm Mod}^\phi_{/\frakS}
 \]
 (see   \cite[\S 1]{KisinJAMS}, \cite[Theorem 3.3.2]{KP} for notations and details of its properties.
 This depends on the choice of   $\sqrt[p^n]{\pi_F}$, $n\geq 0$, in $\bar F$). 
   Let 
  \[
  \frakM(\O_{\calG}):=\varinjlim\nolimits_{i\in J} \frakM(\Lambda_i).
  \]
  By \cite[Theorem 3.3.2]{KP}, the composition of $\frakM$ with restriction to $D^\times$  is an exact faithful tensor functor.
 Hence, we obtain that $\frakM(\O_{\calG})_{|D^\times}$ is a sheaf of algebras over $D^\times$ and that 
 \[
 \calP^\times_{\rm BK}:=\underline{\Spec}(\frakM(\O_{\calG})_{|D^\times})
 \]
  is a $\calG$-torsor over $D^\times$. Using purity, we can extend  $\calP^\times_{\rm BK}$ to a $\Gg$-torsor $\calP_{\rm BK}$ 
  over $D=\Spec(W(k)\lps u\rps)$ as follows: 
  
  Let us consider the scheme 
 \[
 \calP':= \underline {\rm Isom}_{\sbk, s_a\otimes 1}(\frakM(\Lambda), \Lambda\otimes_{\Z_p}W\lps u\rps)\subset  \underline {\rm Hom}(\frakM(\Lambda), \Lambda\otimes_{\Z_p}W\lps u\rps)
 \]
  of isomorphisms taking 
 $\sbk$ to  $s_a\otimes 1$. Here, as in loc. cit., 
 \[
 \sbk:=\tilde s_a\in \frakM(\Lambda)^\otimes
 \] 
 are the tensors obtained by applying the functor $\frakM(-)$ to the Galois invariant tensors $s_{a}\in \Lambda^\otimes$. 
 By \cite[Lemma 3.3.5]{KP}, the scheme $\calP'$ is naturally a $\calG$-torsor over $D$, which, in fact, is trivial.
 As in the proof of \cite[Lemma 3.3.5]{KP}, we  see that there is a natural isomorphism
 $\calP'_{|D^\times}\simeq \calP^\times_{\rm BK}$ as $\calG$-torsors over $D^\times$.
 Hence,   the 
$\calG$-torsor $\calP'$ over $D$ gives the desired extension;  we denote it by $\calP_{\rm BK} $.
We can see that the $\calG$-torsor $\calP_{\rm BK} $ over $D$ is uniquely determined (up to unique isomorphism) and is independent of the choice of $\Lambda$.  This follows from the fact that there is a bijection between sections of $\calP'$ over $W\lps u\rps$ and sections of $\calP'_{|D^\times }=\calP^\times_{\rm BK}$ over $D^\times$. 
The isomorphism $\phi_{\calP_{\rm BK}}$ comes directly from the construction $\calP_{\rm BK} =\calP'$
and is also independent  of choices.

Note here that we can view the Breuil-Kisin $\Gg$-module attached to $\rho$ as an exact  tensor functor
\[
{\rm Rep}_{\Z_p}(\Gg)\to {\rm Mod}^\phi_{/\frakS}.
\]
\end{para}

  \subsection{The Dieudonn\'e $\Gg$-display.}
  
  \begin{para}\label{BKdisplay}
   
 Assume here that $\iota\cdot \rho$ has Hodge-Tate weights in $\{0, 1\}$ and that in fact,
 the deRham filtration on $D_{\rm dR}(\Lambda[1/p])$ is given by a $G$-cocharacter conjugate to $\mu$. 
 Then, there is also a
  Dieudonn\'e $(\Gg, \M)$-display 
  \[
  \calD_{\rho}=(\calP, q, \Psi)
  \]
   over $\O_F$ which is attached to $\rho$. This is constructed as follows:
   
    Consider the Breuil-Kisin module $\frakM=\frakM(\Lambda)$  attached to $\Lambda$. It comes with the Frobenius $\phi_\frakM: \phi^*\frakM[1/E(u)]\xrightarrow{\sim} \frakM[1/E(u)]$. The condition on 
 the weights implies that
 \[
 \frakM\subset \phi_\frakM(\phi^*\frakM)\subset E(u)^{-1}\frakM.
 \]
  
  Let  $\frakS=W(k)\lps u\rps\to \hat W(\O_F)$ be the unique Frobenius equivariant map lifting the identity on $\O_F$ which is given by $u\mapsto [\pi_F]$.  We set
 \[
 \calP:=\calP_{\rm BK}\otimes_{\frakS, \phi}\hat W(\O_F)=\underline{\rm Isom}_{(\sad), (s_a\otimes 1)}(\phi^*\frakM(\Lambda), \Lambda\otimes_{\Z_p}\hat W(\O_F)).
 \] 
 Here, we set
 \[
\sad:= \phi^*(\sbk)\in \phi^*\frakM(\Lambda).
 \]
  To obtain the rest of the data of the Dieudonn\'e $\Gg$-display we proceed as follows:

 We can  write $\frakM=L\oplus T$, with $L$ and $T$ free $\frakS$-modules such that
 \[
 \phi_\frakM(\phi^*\frakM)=L\oplus {E(u)}^{-1}T.
 \]
 Denote by $\frakM_1\subset \phi^*\frakM$  the largest $\frakS$-submodule such that $\phi_\frakM(\frakM_1)  \subset \frakM$.
 We have
 \[
 E(u)\phi^*\frakM\subset \frakM_1\subset \phi^*\frakM.
 \]
  Then $\phi_\frakM(\frakM_1) = \frakM$ and we have an isomorphism
 \[
 \phi_\frakM: \frakM_1\xrightarrow{\sim} \frakM.
 \]
 The corresponding filtration
 \[
 \overline \frakM_1\subset (\phi^*\frakM)/E(u)(\phi^*\frakM)
 \]
 gives  an $\O_F$-valued point of a Grassmannian. Over $F$, this filtration is 
 the deRham filtration of $D_{\rm dR}(\Lambda[1/p])$ by \cite[Theorem 3.3.2 (1)]{KP}.
The condition that the deRham 
 filtration on $D_{\rm dR}(\Lambda[1/p])$ is given by a $G$-cocharacter conjugate to $\mu$
 now implies that this point is in the closure of the 
 $G$-orbit of $\mu$, hence gives an $\O_F$-point of $\M$.  This produces a $\Gg$-equivariant 
 morphism
 \[
 q: \phi^*\calP_{\rm BK}\otimes_\frakS \O_F\to \M.
 \]
 Since $\calP\otimes_{\hat W(\O_F)}\O_F=\phi^*\calP_{\rm BK}\otimes_{\frakS}\O_F$ we obtain
 \[
 q: \calP\otimes_{\hat W(\O_F)} \O_F\to \M.
 \]
 This gives $\calQ$ and $\Psi$ is then determined by $\phi_{\calP_{\rm BK}}$.
 To give these more explicitly, set $M=\frakM\otimes_{\frakS, \phi}\hat W(\O_F)=\phi^*\frakM\otimes_{\frakS}\hat W(\O_F)$
 which acquires the tensors $\sad=\phi^*(\sbk)\in M^{\otimes}$. 
We have
 \[
 \frakM_1\otimes_{\frakS}\hat W(\O_F)\subset M=\phi^*\frakM\otimes_{\frakS}\hat W(\O_F).
 \]
 Using that $\phi(E([\pi_F]))/p$ is a unit in $\hat W(\O_F)$, after applying $\phi$, we obtain a filtration 
 \[
 p(\phi^*M)\subset \tilde M_1:=\phi^*(\frakM_1\otimes_{\frakS}\hat W(\O_F))\subset \phi^*M.
 \]
As in the proof of \cite[Lemma 3.2.9]{KP}, the tensors $\phi^*(\sad)\in \phi^*M^{\otimes}$ lie in $\tilde M_1^\otimes$ and 
\[
\calQ=\underline{\rm Isom}_{(\phi^*\sad), (s_a\otimes 1)}(\tilde M_1, \Lambda\otimes_{\Z_p}\hat W(\O_F)).
\]
The ``divided Frobenius" $\tilde M_1\xrightarrow{\sim} M$ which is obtained by pulling back 
$\phi_\frakM: \frakM_1\xrightarrow{\sim} \frakM$ along $\phi: \frakS\to \hat W(\O_F)$ sends the tensors 
$\phi^*(\sad)$ to $\sad$ and gives the $\Gg$-isomorphism $\Psi:\calQ\xrightarrow{\sim}\calP$.

  \end{para}
  
  \subsection{The Breuil-Kisin-Fargues $\Gg$-module.}

    Here, we use the notations of \S\ref{appCompl}, \S\ref{appAinf}. In particular, $\O$ is the $p$-adic completion of the integral closure $\bar\O_F$ of $\O_F$ in $\bar F$ and $\O^\flat$ is its tilt. For simplicity, set $A_\inf=A_\inf(\O)$.

    \begin{para}\label{BKF}
     By definition, a (finite free) Breuil-Kisin-Fargues (BKF) module over $A_\inf$ is a finite free $A_\inf$-module
    $M$ together with an isomorphism
    \[
    \phi_M: (\phi^*M)[1/\phi(\xi)]\xrightarrow{\sim} M[1/\phi(\xi)]
    \]
    where $\xi$ is a generator of the kernel of $\theta$. (See \cite{Schber}, \cite{BMS}).
    
    Similarly, a Breuil-Kisin-Fargues $\Gg$-module over $A_\inf$ is, by definition, a pair $(\calD_\inf, \phi_{\calD_{\inf}})$, where $\calD_\inf$ is a $\Gg$-torsor over $A_\inf$ and
\[
\phi_{\calD_\inf}: (\phi^*\calD_{\inf})[1/\phi(\xi)]\xrightarrow{\sim} \calD_{\inf}[1/\phi(\xi)]
\]
 is a $\Gg$-equivariant isomorphism.
    \end{para}
       
 \begin{para}
 
    Now fix a uniformizer $\pi=\pi_F$ of $F$ and also a compatible system of roots $\pi^{1/p^n}$, for $n\geq 1$, giving an element $\pi^\flat=(\pi, \pi^{1/p},\ldots ) \in \O^\flat$. 
 These choices define a $\phi$-equivariant homomorphism
\[
f : \frakS=W\lps u\rps\to A_\inf
\]
given by $u\mapsto [\pi^\flat]^p$ and which is the Frobenius on $W=W(k)$. By \cite[Proposition 4.32]{BMS}, the association 
\[
\frakM\mapsto M=\frakM\otimes_{\frakS}A_\inf
\] 
defines an exact tensor functor from Breuil-Kisin modules over $\frakS$ to Breuil-Kisin-Fargues (BKF) modules over $A_\inf$.

We can compose the above functor with the tensor exact functor
\[
{\rm Rep}_{\Z_p}(\Gg)\to {\rm Mod}^\phi_{/\frakS}
\]
given by the Breuil-Kisin $\Gg$-module $\calP_{\rm BK}$  over $D=\Spec(\frakS)$ of \S\ref{BKtorsor}. We obtain a tensor exact functor
\[
{\rm Rep}_{\Z_p}(\Gg)\to {\rm Mod}^\phi_{/A_\inf}
\]
to the category ${\rm Mod}^\phi_{/A_\inf}$ of finite free BKF modules over $A_\inf$.
This functor gives a $\Gg$-torsor $\calD_{\inf}$ over $A_{\inf}$ which admits a $\Gg$-equivariant isomorphism
\[
\phi_{\calD_\inf}: (\phi^*\calD_{\inf})[1/\phi(\xi)]\xrightarrow{\sim} \calD_{\inf}[1/\phi(\xi)].
\]
(Here, $\phi(\xi)=f(E(u))$ for $E(u)\in W\lps u\rps$ an Eisenstein polynomial for $\pi$.
The element $\xi$ generates the kernel of $\theta: A_\inf\to \O$.)
Hence, $(\calD_{\inf}, \phi_{\calD_\inf})$ is a  Breuil-Kisin-Fargues $\Gg$-module
which is attached to $\rho$.

More explicitly, set
\[
M_\inf:=M_\inf(\Lambda)=\frakM(\Lambda)\otimes_{\frakS, f} A_\inf.
\]
The  tensors $s_{a}\in \Lambda^\otimes$ induce $\phi$-invariant tensors $s_{a, \inf}\in M^\otimes_\inf$. These are the base changes 
\[
s_{a, \inf}=f^*(\sbk)
\]
of $\sbk\in \frakM(\Lambda)^\otimes$. We have
\[
\calD_{\inf}\cong \underline{\rm Isom}_{(s_{a, \inf}), (s_a\otimes 1)}(M_\inf, \Lambda\otimes_{\Z_p}A_{\inf}(\O))
\] 
as $\Gg$-torsors.

 \end{para}

  \begin{para}
 
Assume  that $\Lambda$, acted on by $\rho$, is isomorphic to the Galois representation on the Tate module $T=T_p(\mathscr G)$ of a $p$-divisible group $\mathscr G$ over $\O_F$.
We have 
\begin{equation}\label{dualinf}
M_\inf\cong  M(\mathscr G).
\end{equation}
Here, $M(\mathscr G)$ is the BKF module associated (\cite[Theorem 17.5.2]{Schber}) to the base change over $\O$ of $\mathscr G$ of the $p$-divisible group $\mathscr G$. 
The corresponding $\phi$-linear Frobenius $\phi_{M_\inf}$ satisfies
\[
 M_\inf\subset \phi_{M_\inf}(M_\inf)\subset  \phi(\xi)^{-1} M_\inf.
\]
(Here, again, we denote a $\phi$-linear map and its linearization by the same symbol.)
\end{para}
       
 \begin{para}
 Choose $p$-th power roots of unity giving $\epsilon=(1, \zeta_p, \zeta_{p^2}, \ldots )\in \O^\flat$ and set $\mu=[\epsilon]-1\in A_\inf$. 
 
 Let  $\underline{\Q_p/\Z_p}$ be the constant $p$-divisible group. By \cite[theorem 17.5.2]{Schber}, there is a comparison map
\[
\Lambda\cong {\rm Hom}_\O(\underline {\Q_p/\Z_p}, \mathscr G_\O )
\xrightarrow{\sim} {\rm Hom}_{A_\inf, \phi}(A_\inf, M_\inf)=M_\inf^{\phi_{M_\inf}=1}.
\]
This induces the $\phi$-invariant
isomorphism
\[
\Lambda\otimes_{\Z_p} A_\inf[1/\mu]\cong M_\inf[1/\mu]\cong\frakM(\Lambda)\otimes_{\frakS,f} A_\inf[1/\mu].
\]
It follows from the constructions and \cite[4.26]{BMS} that under these isomorphisms the tensors 
$s_{a}\otimes 1$, $s_{a, \inf}$ 
and $\sbk\otimes 1$ correspond. 
\end{para}
       
 \begin{para}\label{BKFcompatible}
 
 The constructions of the previous paragraphs are compatible in the following sense. 
 Assume that $\rho$ is as in the beginning of \S\ref{BKdisplay}; then $\Lambda$ is the Tate module of a $p$-divisible group over $\O_F$.
  Fix   $\pi=\pi_F$ of $F$ and a compatible system of roots $\pi^{1/p^n}$, for $n\geq 1$, giving   $\pi^\flat=(\pi, \pi^{1/p},\ldots ) \in \O^\flat$ as above. Recall  the homomorphism 
 \[
 \theta_\infty: A_{\inf}(\O)=W(\O^\flat)\to W(\O), \quad \theta_\infty([(x^{(0)}, x^{(1)}, \ldots )])=[x^{(0)}].
 \]
 The diagram
 $$
 \begin{matrix}
  \frakS &  \xrightarrow{\ \ \ f\ \ \ } & A_{\inf}(\O)\\
\phi\downarrow\ \ \ \  &\ \ & \downarrow\theta_\infty\\
 \hat W(\O_F)&\xrightarrow{\ \ \ \ \ \ \ \ } & W(\O), \\
 \end{matrix}
 $$
 where the bottom horizontal map is given by the inclusion, commutes.
 We then have isomorphisms of $\Gg$-torsors
 \begin{equation}\label{compareEq}
 \calP\otimes_{\hat W(\O_F)}W(\O)\simeq \calP_{\rm BK}\otimes_{\frakS, \phi} W(\O)\simeq \calD_\inf\otimes_{A_\inf(\O)}W(\O)
 \end{equation}
which are compatible with the Frobenius structures: This can be seen by combining 
results of \cite[\S 4]{BMS}, \cite[\S 17]{Schber}, and the above constructions.
Similarly, we can see that both the $(\Gg, \calM)$-display $\calD$ and the Breuil-Kisin-Fargues $\Gg$-module $(\calD_{\inf}, \phi_{\calD_\inf})$ are, up to a canonical isomorphism, independent of the choice of $\pi_F$ and its roots $\sqrt[p^n]{\pi_F} $ in $\bar F\subset C$. 
 \end{para}
 
 \begin{Remark}
 {\rm The various compatibilities after (often confusing) Frobenius twists between these different objects, all attached to the same integral crystalline representation, can be explained via the theory of prisms and prismatic cohomology of Bhatt and Scholze \cite{BSPrism}. Indeed, the BK and BKF $\Gg$-modules should be ``facets" of a single object, a prismatic Frobenius crystal with $\Gg$-structure over ${\rm Spf}(\O_F)$.  On the other hand, the $\Gg$-display fits somewhat less directly into this and seems to be tied more closely to $p$-divisible groups by using the Hodge embedding.}
 \end{Remark}

  \ve
  
  \section{Associated systems}\label{s2}
  
  Here, we define the notion of an associated system and give several results. The main result says, roughly, that a pro-\'etale $\Gg(\Z_p)$-Galois cover which is given by the Tate module of a $p$-divisible group over a normal base with appropriate \'etale tensors, can be extended uniquely to an associated system (see Theorem \ref{Exists} for the precise statement).
 We also show how to use the existence of \emph{locally universal} associated systems to compare formal completions of normal schemes with the same generic fiber (Proposition \ref{MainProp}). Finally, we show that the definition of associated is independent of the choice of the local Hodge embedding (Proposition \ref{indassociated}).
 
  In all of \S \ref{s2} we assume, without further mention, that $(\Gg, \calM)$ is of {\sl strongly} integral local Hodge type. All the Hodge embeddings $\iota: \Gg\hookrightarrow \GL(\Lambda)$ we consider are strongly integral: they induce a closed immersion $\iota_*:\calM\hookrightarrow {\rm Gr}(d,\Lambda)_{\O_E}$.
    
  \subsection{Local systems and associated systems}
  
    Let us suppose that  $\calX$ is a flat $\O_E$-scheme of finite type, which is normal
 and has smooth generic fiber. Suppose that we are given  a Galois cover 
  of $X=\calX[1/p]$ with group $\Gg(\Z_p)=\varprojlim\nolimits_n \Gg(\Z/p^n\Z)$; this gives 
  a pro-\'etale $\Gg(\Z_p)$-cover $\Ls$ on $X$.

 \begin{para}\label{311} For any $\bar x\in \calX(k)$, let $\hat R_{\bar x}$ be the completion of the strict Henselization of the local ring $R_{\bar x}$  of $  \calX $ at
$\bar x$. 

Suppose we have a Dieudonn\'e  $(\Gg, \M)$-display $\calD_{\bar x}=(\calP_{\bar x}, q_{\bar x}, \Psi_{\bar x})$ over $\hat R_{\bar x}$. Choose a (strongly integral) local Hodge embedding $\iota$.

 In accordance with our notations in \S\ref{gen12}, \S \ref{compaZink}, we will denote by 
\[
\calD_{\bar x}(\iota)=(M_{\bar x}, M_{1, \bar x}, F_{1, \bar x})
\]
the Dieudonn\'e display over $\hat R_{\bar x}$ induced from $\calD_{\bar x}$ using $\iota$
and the construction of \cite[3.1.5]{KP}. As usual, Proposition \ref{torsorRep} gives tensors $m_{a, \bar x}\in M_{\bar x}^{\otimes}$ corresponding to $s_a\in \Lambda^{\otimes }$.

By \cite{ZinkCFT}, there is a corresponding $p$-divisible group $\mathscr G(\bar x)$    over $\hat R_{\bar x}$
\[
\mathscr G(\bar x)={\rm BT}( \calD_{\bar x}(\iota)),
\]
of height   $n={\rm rank}_{\Z_p}(\Lambda)$.
Recall (\S \ref{compaZink}), we have a canonical isomorphism
 \begin{equation}\label{Dieudonne}
 \calD_{\bar x}(\iota)\cong{\DD}(\mathscr G(\bar x))(\hat W(\hat R_{\bar x}))
\end{equation}
of  Dieudonn\'e displays, where on the right hand side,  ${\DD}$ denotes the evaluation of the \emph{covariant} Dieudonn\'e crystal.

 \smallskip

 Consider the following two conditions. The first is:

\medskip
   
  A1)  There is an isomorphism $\alpha=\alpha_{\bar x}$ of $\Z_p$-local systems over $\hat R_{\bar x}[1/p]$ between the local system given by the Tate module $T$ of the $p$-divisible group $\mathscr G({\bar x})={\rm BT}( \calD_{\bar x}(\iota))$ and the pull-back of $\Lr(\iota)$.  
  \medskip
  
 Before we state the second condition, we observe the following. Assuming (A1), for any $\ti x\in \calX(\O_F)$ that lifts $\bar x$, the Galois representation $\rho({x})$ obtained from $x^*\Ls$, is crystalline. By \cite[Theorem 3.3.2 (2)]{KP}, the isomorphism $\alpha$ in (A1) induces an isomorphism
 \begin{equation}\label{Dieudonne2}
\calD_{\rho(x)}(\iota)\cong\DD(\ti x^*\mathscr G(\bar x))(\hat W(\O_F)).
\end{equation}
Here, $\ti x^*\mathscr G(\bar x)$ is the $p$-divisible group over $\O_F$ obtained by base-changing $\mathscr G(\bar x)$ by $\ti x$ and $\calD_{\rho(x)}$ is the Dieudonn\'e $(\Gg, \calM)$-display attached to $\rho(x)$ by \S\ref{BKdisplay}.
Combining (\ref{Dieudonne}), (\ref{Dieudonne2}), and base change gives an isomorphism
 \begin{equation}\label{Dieudonne2a}
  \calD_{\rho(x)}(\iota)   \xrightarrow{\sim}     \calD_{\bar x}(\iota) \otimes_{\hat W(\hat R_{\bar x})}\hat W(\O_F)
\end{equation}
of Dieudonn\'e displays over $\O_F$.

\smallskip

We can now state the second condition (it only makes sense after we assume (A1)):
\smallskip

A2) For every $\ti x\in \calX(\O_F)$ lifting $\bar x$, there is an isomorphism
of Dieudonn\'e $(\Gg,\M)$-displays above
 $
 \calD_{\rho(x)}   \xrightarrow{\sim}  \calD_{\bar x} \otimes_{\hat W(\hat R_{\bar x})}\hat W(\O_F)
$
 over $\O_F$ which, after applying $\iota$, induces the isomorphism (\ref{Dieudonne2a})
\[
  \calD_{\rho(x)}(\iota)   \xrightarrow{\sim}     \calD_{\bar x}(\iota) \otimes_{\hat W(\hat R_{\bar x})}\hat W(\O_F).
\]
\smallskip

More concretely,  we see that condition (A2) is equivalent to the following:
\smallskip

A2')  For every $\ti x\in \calX(\O_F)$ lifting $\bar x$, and every $a$, the isomorphism (\ref{Dieudonne2a}) maps the tensor $\sad=s_{a, \scriptscriptstyle{\rm D}, \ti x}\in  M^{\otimes}$ attached to $\rho(x)$ in \S \ref{BKdisplay}
to the base change $m_{a, \bar x}\otimes 1\in M_{\bar x}^{\otimes }\otimes_{\hat W(\hat R_{\bar x})}\hat W(\O_F)$ of the tensor $m_{a,\bar x}\in M_{\bar x}^{\otimes }$.

 \begin{Definition}\label{defass}
 If (A1) and (A2)  hold for $\bar x\in \calX(k)$, we say that $\Ls$ and $\calD_{\bar x}$ are \emph{associated}.
 If (A1) and (A2) hold for all $\bar x\in \calX(k)$, with $\alpha_{\bar x}$ the isomorphism in (A1),
 we call $(\Ls, \{\alpha_{\bar x}, \calD_{\bar x}\}_{\bar x\in \calX(k)})$ an \emph{associated system}.
 \end{Definition}

\begin{Definition}\label{locally universal}
The associated system $(\Ls, \{\alpha_{\bar x}, \calD_{\bar x}\}_{\bar x\in \X(k)})$ is \emph{locally universal} over $\X$,
if for every $\bar x\in \calX(k)$, $\calD_{\bar x}$ is locally universal in the sense of Definition \ref{Def145}.
 \end{Definition}

 The definition of ``associated'' uses the local Hodge embedding $\iota$ which we, for now,  fix in our discussion. We will later show that it is independent of this choice, see  Proposition  \ref{indassociated}. Most of the time, we will omit the notation
 of the isomorphisms $\alpha_{\bar x}$ and write $(\Ls, \{\calD_{\bar x}\}_{\bar x\in \X(k)})$ for the associated system.

 \begin{prop}\label{depL}
 If $\Ls$ and $\calD_{\bar x}$, for $\bar x\in \calX(k)$, are associated,   
 then $\calD_{\bar x}$ is, up to isomorphism, 
 uniquely determined by $\Ls$. 
 \end{prop}
 
 \begin{proof} Suppose that $\Ls$ and $\calD'_{\bar x}$ are also associated. Then $\mathscr G({\bar x})[1/p]\simeq \mathscr G'({\bar x})[1/p]$ as $p$-divisible groups over $\hat R_{\bar x}[1/p]$, since they both have the same Tate module which is given by  the restriction of $\Ls(\iota)$ to $\hat R_{\bar x}[1/p]$. Tate's theorem applied to the normal Noetherian domain $\hat R_{\bar x}$, extends this to a unique isomorphism
 $\beta: \mathscr G({\bar x})\xrightarrow{\sim} \mathscr G'({\bar x})$.  Therefore, using \cite{ZinkCFT}, we obtain an isomorphism 
 of Dieudonn\'e displays $\delta: \calD_{\bar x}(\iota)\xrightarrow{\sim} \calD'_{\bar x}(\iota)$. 
 This amounts to an isomorphism
 \[
 (M, M_1, F_1)\xrightarrow{\delta} (M', M'_1, F'_1).
 \]
 Here both $M$, $M'$ are free $\hat W(\hat R_{\bar x})$-modules of rank $n$. The   
 $(\Gg, \M)$-displays $\calD_{\bar x}$ and $\calD'_{\bar x}$ have corresponding $\Gg$-torsors $\calP$, $\calP'$. By the construction of Proposition \ref{torsorRep}, these $\Gg$-torsors are given by $M$, $M'$ and tensors
  $m_{a}\in M^{\otimes}$, $m'_{a}\in M'^{\otimes}$, respectively. We would like to show that $\delta: M\to M'$ 
  satisfies $\delta^{\otimes}(m_a)=m'_a$, i.e. $\delta$ lies in 
 the $\hat W(\hat R_{\bar x})$-valued points of the closed subscheme
 \[
\underline{ {\rm Hom}}_{(m_a), (m'_a)}(M, M')=\Spec(B/I)\hookrightarrow \underline{ {\rm Hom} }(M, M')= \Spec(B).
 \]
 Here, $B\simeq  \hat W(\hat R_{\bar x})[(t_{ij})_{1\leq i, j\leq n}]$, non-canonically. Let us consider $f(t_{ij})\in I$. We would like to show that 
  $f(\delta_{ij})=0$ in $\hat W(\hat R_{\bar x})$, where $\delta_{ij}\in \hat W(\hat R_{\bar x})$ are the coordinates of the 
  $\hat W(\hat R_{\bar x})$-linear map $\delta$. 
  Condition (A2) implies that $\ti x^*\delta$ respects the tensors $\ti x^*m_a$, 
  $\ti x^*m'_a$, so  $(\ti x)^*(f(\delta_{ij}))=0$,
  for all $\ti x$ lifting $\bar x$. This implies that $f(\delta_{ij})=0$, so $\delta$ respects the tensors. It now follows that 
  $\delta$ respects the rest of the data that give the $(\Gg, \M)$-displays
  $\calD_{\bar x}$ and $\calD'_{\bar x}$.
 \end{proof}

 \subsection{Local systems and associated displays}

 Let $\Dd$ be a $(\Gg, \M)$-display over the $p$-adic 
 formal scheme $\frakX=\varprojlim_n \calX\otimes_{\O_E}\O_E/(p)^n$. 
 
 \begin{Definition}\label{defass2}
 We say that the $(\Gg, \M)$-display $\Dd$ over $\frakX$ is associated with $\Lr$ if, for all
 $\bar x\in \calX(k)$, there is 
 \begin{itemize}
 \item a Dieudonn\'e $(\Gg, \M)$-display $\Dd_{\bar x}$
 which is associated with $\Lr$,
 \item
  an isomorphism of  $(\Gg, \M)$-displays
 \[
 \Dd_{\bar x}\otimes_{\hat W(\hat R_{\bar x})}W(\hat R_{\bar x})\simeq \Dd\otimes_{W(\O_{\frakX})}W(\hat R_{\bar x}).
 \] 
 \end{itemize}
 \end{Definition}
 
 Note that, then, $(\Lr, \{\Dd_{\bar x}\}_{\bar x\in \X(k)})$ is an associated system.

  \begin{Definition}
   We say that the $(\Gg, \Mloc)$-display $\Dd$ over $\frakX$ which is associated with $\Lr$, is locally universal over $\frakX$,
  if the associated system $(\Lr, \{\Dd_{\bar x}\}_{\bar x\in \X(k)})$ is \emph{locally universal} over $\X$.
  \end{Definition}

 \subsection{Rigidity and uniqueness}  
 
   Assume now that $\X$ and $\X'$ are two flat $\O_E$-schemes of finite type,  normal with the \emph{same} smooth generic fiber $X=\X[1/p]=\X'[1/p]$. Suppose that $(\Lr, \{\calD_{\bar x}\}_{\bar x\in \X(k)})$ and $(\Lr', \{\calD'_{\bar x'}\}_{\bar x'\in \X'(k)})$ are locally universal associated systems on $\X$ and $\X'$ respectively, with $\Lr= \Lr'$ on $X$.

     Denote by $\Y$ the normalization of the Zariski closure of the diagonal embedding of $X$ in the product $\X\times_{\Spec(\O_E)}\X'$. Denote by 
   \[
   \X\xleftarrow{\ \pi } \Y\xrightarrow{\pi'}\X',
   \]
    the morphisms given by the two projections. For simplicity, we again set $\breve\Y=\Y\otimes_{\O_E}\O_{\breve E}$, $\breve \X=\X\otimes_{\O_E}\O_{\breve E}$, etc.  For $\bar y\in \Y(k)$, set $\bar x=\pi(\bar y)$, $\bar x'=\pi'(\bar y)$.
   
   \begin{prop}\label{MainProp}
 a) We have
 \[
\pi^* \calD_{\bar x} \simeq  \pi'^* \calD'_{\bar x'}
\]
 as Dieudonn\'e $(\Gg, \M)$-displays on the completion $\hat\O_{\breve\Y ,\bar y}$.
 
  b) The morphism $\pi$  induces an isomorphism
  \[
  \pi^*: \hat\O_{\breve\X ,\bar x}\xrightarrow{\sim} \hat\O_{\breve\Y ,\bar y}
   \] between the completions  of $\breve \Y$ and $\breve \X$,   at $\bar y$ and $\bar x$, respectively. Similarly, for $\pi'$.
   \end{prop}
   
\begin{proof}
Part (a) follows by the argument in the proof of Proposition \ref{depL}.  

Let us    show (b).
For simplicity, set $R=\hat \O_{\breve \X, \bar x}$, $R'=\hat\O_{\breve \X', \bar x'}$,
$R''=\hat\O_{\breve \Y, \bar y}$.
By the construction of $\Y$, we have a local homomorphism
\[
R\hat\otimes_{\O_{\breve E}}R'\to R''
\]
which is finite. Write $R_1$ for its image:
\[
R\hat\otimes_{\O_{\breve E}}R' \twoheadrightarrow R_1\hookrightarrow R''
\]
Applying $\iota$ and the functor ${\rm BT}$ to $\calD_{\bar x}$, $\calD'_{\bar x}$, gives $p$-divisible groups $\mathscr G$, $\mathscr G'$ over $R$, $R'$  respectively. By (a) we have
\begin{equation}\label{IsoPdiv}
  \pi^*\mathscr G\simeq \pi'^*\mathscr G'
 \end{equation}
 over $R''$. This isomorphism specializes to give $f_0: \bar x^*(\mathscr G)\xrightarrow{\sim}  \bar x'^*(\mathscr G')$, an isomorphism of $p$-divisible groups over the field $k$.
  
Let us write $T=\Spf(U)$  for the base change to $\O_{\breve E}$ of the universal deformation space of a $p$-divisible group $\mathscr G_0$ over $k$ which is isomorphic to the $p$-divisible groups $\bar x^*(\mathscr G)$ and $\bar x'^*(\mathscr G')$ above, and fix such isomorphisms. This allows us to view $f_0$ as an isomorphism $f_0: \mathscr G_0\xrightarrow{\sim} \mathscr G_0$.

Set $\Spf(R)=S$, $\Spf(R')=S'$, $\Spf(R'')=S''$, and   $Z=\Spf(R_1)$. By the locally universality condition on $\calD_{\bar x}$ and $\calD'_{\bar x}$, 
$S$ and $S'$ can both be identified with closed formal subschemes of $T$ given by ideals $I$ and $I'$ of $U$, respectively.  There is a closed formal subscheme $\Gamma$ of $T\hat\times_{\O_{\breve E}} T=\Spf(U\hat\otimes_{\O_{\breve E}}U)$ prorepresenting the subfunctor of pairs of deformations of $\mathscr G_0$ where $ f_0$ extends as an isomorphism. The subscheme $\Gamma$  is defined by the ideal generated by $(u\otimes 1-1\otimes f^*_0(u))$, $u\in U$, where $f^*_0: U\xrightarrow{\sim} U$ is the ``relabelling'' automorphism corresponding to $ f_0 $. By (\ref{IsoPdiv}), we have 
that $Z\subset T\hat\times_{\O_{\breve E}} T$ is contained (scheme theoretically) in the ``intersection''
 \[
 \Gamma\cap (S\hat\times_{\O_{\breve E}} S')=\Spf(U\hat\otimes_{\O_{\breve E}} U/((u\otimes 1-1\otimes f^*_0(u))_{u\in U}, I\otimes U, U\otimes I').
 \]
 The projection makes this isomorphic to $S\cap {f^*_0}^{-1}(S')$, the formal spectrum of $R/J$, where we set $J:=f^*_0(I')R$. From
 \[
 \Spf(R_1)=Z\subset \Gamma\cap (S\hat\times_{\O_{\breve E}} S')\simeq \Spf(R/J)
 \]
 we have $\dim(R_1)\leq \dim(R/J)$. Since $R_1$ is integral of dimension equal to that of $R''$
 and so of $R$, we have $\dim(R)=\dim(R_1)\leq \dim(R/J)\leq \dim(R)$. Since $R$ is an integral domain, this implies $J=(0)$ and that $R_1$, which is a quotient of $R/J$ of the same dimension, is also isomorphic to $R$. Since $R\simeq R_1\to R''$ is finite, and $R$, $R'$ and $R''$ are normal, the birational $R_1\to R''$ is an isomorphism; so is $R\to R''$ and, by symmetry, also $R'\to R''$.  
\end{proof}

\end{para}

 \subsection{Existence of associated systems} 
 
 Suppose that $\calX$, $\Lr$, and $\iota: \Gg\hookrightarrow \GL(\Lambda) $, are as in the beginning of Section \ref{s2}.

 \begin{thm}\label{Exists}
 Suppose that the \'etale local system $\Lr(\iota)$ is given by the Tate module of a $p$-divisible group $\sG$ over $\calX$.
 Then $\Lr$ is part of a unique, up to unique isomorphism, associated system $(\Lr, \{\alpha_{\bar x}, \calD_{\bar x}\}_{\bar x\in \calX(k)})$ for $\iota$.
  \end{thm}
 
 \begin{proof}
 The uniqueness part of the statement follows from Proposition \ref{depL}. Our task is to construct, for each $\bar x\in \calX(k)$, a $(\Gg, \calM)$-display $\calD_{\bar x}$ over $R=\hat R_{\bar x}$ that satisfies (A1) and (A2). Let 
 \[
 (M, M_1, F_1)
 \]
be the Dieudonn\'e display obtained by the evaluation $M=\DD(\sG)(\hat W(R))$ of the Dieudonn\'e crystal of $\sG$ over $R$. 
This gives a $(\GL_n, {\rm Gr}(d, \Lambda))$-display
$(\P_{\GL}, q_{\GL}, \Psi_{\GL})$. 
We want to upgrade this to a $(\Gg, \calM)$-display, the main difficulty being the construction of appropriate tensors $m_a\in M^\otimes$. The construction occupies several paragraphs:
 
 \begin{para}\label{332}

 We recall the notations and results of  \S\ref{appCompl}, \S\ref{appAinf}, for $R$.
In particular, we  fix an algebraic closure $\overline{ F(R)}$ of the fraction field $F(R)$, we denote by $\bar R$ the integral closure of $R$ in $\overline{ F(R)}$ and by $\ti R$ the union of all finite 
normal $R$-algebras $R'$ in $\overline{ F(R)}$ such that $R'[1/p]$ is \'etale over $R[1/p]$. Set $\bar R^\wedge$ and $\ti R^\wedge $ for their $p$-adic completions. For simplicity, we set
\[
S=\ti R^\wedge.
\]
Also, we set $\O$ for the $p$-adic completion of the integral closure $\bar\O_E$ of  $\O_E$ in the algebraic closure $\bar E$. 
 
 By \S \ref{appCompl}, $\bar R^\wedge$, $S=\ti R^\wedge$,  and $\O$, are integral perfectoid $\Z_p$-algebras in the sense of \cite[3.1]{BMS}, which are local Henselian and flat over $\Z_p$. 
 The Galois group 
 $\Gamma_R$ acts on $\ti R$, on $S$, and on $A_\inf(S)=W(S^\flat)$. 
 \end{para}

\begin{para}
Let $(M(\sG)=M(\mathscr G)(S), \phi_{M(\sG)})$ be the (finite free) Breuil-Kisin-Fargues   module over $A_{\inf}(S)$ attached to the base change  $\mathscr G_S$ of     $\mathscr G$. 

By \cite[Theorem 17.5.2]{Schber},
$M(\sG)(S)$ is the value of a functor which gives an equivalence between $p$-divisible groups over $S$ and finite projective BKF modules $(M, \phi_M)$ over $A_\inf(S)$ that satisfy
\[
M\subset \phi_{M}(M)\subset \frac{1}{\phi(\xi)}M.
\]
By loc. cit., the equivalence is functorial in $S$. Therefore,  $M(\sG)(S)$ supports an action of $\Gamma_R$ which 
commutes with $\phi_{M(\sG)(S)}$ and is semi-linear with respect to the action of $\Gamma_R$ on $A_\inf(S)$. By 
loc. cit., we have
\begin{equation}
T={\rm Hom}_S(\underline{\Q_p/\Z_p}, \mathscr G_S)\xrightarrow{\sim} {\rm Hom}_{A_{\inf}(S), \phi}(A_{\inf}(S), M(\sG))=M(\sG)^{\phi_{M(\sG)}=1}.
\end{equation}
This gives the comparison homomorphism
\[
c: T\otimes_{\Z_p} A_{\inf}(S)\xrightarrow{\ } M(\sG)(S)
\]
which is $\phi$ and Galois equivariant. 

Using the constructions in \cite[\S 17]{Schber} together with Lemma \ref{lemmainf0} and Proposition \ref{lemmainf1}, we see that $c$ is injective and gives
\[
T \otimes_{\Z_p}A_{\rm inf}(S)\subset M(\sG)(S)\subset   T\otimes_{\Z_p}\frac{1}{\mu}A_{\rm inf}(S).
\]
Therefore, we obtain a ``comparison'' isomorphism
\begin{equation}\label{tensorInf}
T^{\otimes }\otimes_{\Z_p}A_{\rm inf}(S)[1/\mu]\xrightarrow{\sim} M(\sG)(S)^{\otimes}[1/\mu].
\end{equation}
\end{para}

\begin{para}\label{536}
Let us set:
\[
M_{\rm inf}(S):=M(\sG)(S).
\]
Let 
\[
s_{a, \inf}\in M_\inf(S)^{\otimes}[1/\mu]
\]
 be the  tensors 
which correspond to  $s_{a}\in T^{\otimes}$ under  (\ref{tensorInf}). 
We have 
\[
\phi_{M_\inf}(s_{a,\inf})=s_{a,\inf}.
\]

 We can now construct a Breuil-Kisin-Fargues $\Gg$-module $(D_{\inf}(S), \phi_{D_{\inf}})$ over $A_\inf(S)$.
 
 \begin{prop}\label{p537}
 a) We have $s_{a, \inf}\in M_\inf(S)^\otimes$.
 
 b) By (a), we can consider 
 the $\Gg$-scheme
 \[
 D_{\inf}(S):=\underline{\rm Isom}_{(s_{a, \inf}), (s_a\otimes 1)}(M_\inf(S), \Lambda\otimes_{\Z_p}A_\inf(S)).
 \]
The scheme $D_{\inf}(S)$, with its natural $\Gg$-action, is  a $\Gg$-torsor over $A_\inf(S)$.
 
 c) There is a $\Gg$-equivariant isomorphism
 \[
\phi_{D_{\inf}}: (\phi^*D_{\inf}(S))[1/\phi(\xi)]\xrightarrow{\sim}D_{\inf}(S)[1/\phi(\xi)],
 \]
 where $\xi$ is any generator of the kernel of $\theta: A_\inf(S)\to S$.
 
 d) Suppose $\ti x: S\to \O$ extends  a point $\ti x: R\to \O_F$ which lifts $\bar x$.
 Then the base change of $(D_{\inf}(S), \phi_{D_{\inf}})$ by $\ti x: S\to \O$ is isomorphic to the BKF $\Gg$-module over $A_\inf(\O)$ which is attached to $\ti x^*\Lr$ by \ref{BKF}.
 \end{prop}

 \begin{proof}
 Consider  $ \ti x: S\to \O$ as in (d).  Let $(M(\sG)(\O), \phi_{M(\sG)(\O)})$ be the BKF module over $A_\inf(\O)$ attached to the 
$p$-divisible group $\ti x^*\mathscr G$ over $\O$.  
By functoriality under $S\to \O$ of the functor of \cite[Theorem 17.5.2]{Schber}, we have a canonical isomorphism
\[
M_\inf(S)\otimes_{A_\inf(S)}A_\inf(\O)\cong M_\inf(\O)=M(\sG)(\O)
\]
respecting the Frobenius structures. 

\quash{Note that for each such $\ti x$,
$(M(\sG)(\O), \phi_{M(\sG)(\O)})$ supports a BKF $\Gg$-module structure
which is attached to the a $\GG(\Z_p)$-valued crystalline representation
obtained by specializing ${\rm L}$ to $\tilde x: R\to \O_{\bar E}$.}

 \begin{lemma}\label{l538}
The pull-back
\[
\ti x^*(s_{a, \inf})\in M_\inf(S)^{\otimes}[1/\mu]\otimes_{A_\inf(S)}A_\inf(\O)=M_\inf(\O)^{\otimes}[1/\mu]
\]
lies in $M_\inf(\O)^{\otimes}$.
\end{lemma}

\begin{proof}  Recall that the tensors $s_{a, \inf}\in M_\inf(S)^{\otimes}[1/\mu]$ are defined using the comparison isomorphisms (\ref{tensorInf}). The statement  then follows from functoriality under $S\to \O$ using the fact that 
$(M(\sG)(\O), \phi_{M(\sG)(\O)})$ supports a $\Gg$-BKF module structure
(so, in particular, the corresponding \'etale tensors in $M_\inf(\O)^{\otimes}[1/\mu]$ extend over $A_\inf(\O)$,
 i.e. have no $\mu$-denominators.) 
\end{proof}

Now we can proceed with the proof of the proposition. Part (a) follows from the above Lemma and Lemma \ref{lemmainf1}.
By Lemma \ref{lemmainf0} and Proposition \ref{RaynaudGruson},  $D_{\inf}(S)$ is a $\Gg$-torsor, i.e. part (b) holds.
The identity  $\phi_{M_\inf}(s_{a,\inf})=s_{a,\inf}$ holds in $M_\inf(S)^\otimes [1/\mu]$ and so also in $M_\inf(S)^\otimes$ since
\[
M_\inf(S)\subset M_\inf(S)[1/\mu]
\] 
by Lemma \ref{lemmainf1}. 
Therefore, 
 \[
\phi_{D_{\inf}}: (\phi^*D_{\inf}(S))[1/\phi(\xi)]\xrightarrow{\sim}D_{\inf}(S)[1/\phi(\xi)].
 \]
 is $\Gg$-equivariant, which is (c). Finally, (d) follows from the above and functoriality under $S\to \O$.
 \end{proof}
 
 \begin{Remark} {\rm 
 a) Using these constructions and the comparison
 \[
 T\otimes_{\Z_p}A_\inf(S)[1/\mu]\simeq M(\sG)(S)[1/\mu],
 \]
 we can see that the BKF $\Gg$-module $(D_\inf(S), \phi_{D_\inf})$ only depends, up to isomorphism, on $\Lr$.
 Indeed, from \S\ref{BKF}, this statement is true when $S=\O$. In general, 
 the comparison isomorphism first implies that  the $\Gg$-torsor $D_\inf(S)[1/\mu]$ depends, up to isomorphism, only on $\Lr$. Then, by considering
 restriction along   $\ti x: S\to\O$ and using Lemma \ref{lemmainf1} (b), we see that we can determine $D_\inf(S)$ and $\phi_{D_{\inf}}$
 over $A_\inf(S)$.
 
b) As usual, we may think of $D_\inf(S)$ as an exact tensor functor
 \[
 {\rm Rep}_{\Z_p}(\Gg)\to {\rm Mod}^\phi_{/A_\inf(S)}.
 \]
 }
 \end{Remark}
 \end{para}

\begin{para}

 We can now complete the proof of Theorem \ref{Exists}.
 Recall that
 \[
 \theta_\infty: A_\inf(S)\to W(S)
 \]
  factors as a composition
 \[
  \theta_{\infty}: A_\inf(S)\to A_{\rm cris}(S)\to W(S).
  \]
By, \cite[Theorem 17.5.2]{Schber}, the  Frobenius module 
 \[
  M_\inf(S)\otimes_{A_{\inf}(S)}A_{\cris}(S) =M(\sG)(S)\otimes_{A_{\inf}(S)}A_{\cris}(S)
 \]
 describes the covariant Dieudonn\'e module of the base change $\mathscr G_S$ evaluated at the divided power thickening $A_{\rm cris}(S)\to S$. By  \cite[Theorem B]{Lau}, this evaluation of the Dieudonn\'e module is naturally isomorphic to $M\otimes_{\hat W(R) } W(S)$, with its Frobenius structure. Combining these now gives a natural isomorphism
\begin{equation}\label{combEq}
 M\otimes_{\hat W(R) } W(S)\simeq M_\inf(S)\otimes_{A_{\inf}(S)} W(S)
 \end{equation}
 which is compatible with Frobenius and the action of $\Gamma_R$. We obtain
 \[
 M^\otimes\otimes_{\hat W(R) } W(S)\simeq M_\inf(S)^\otimes\otimes_{A_{\inf}(S)} W(S).
 \]
 Since the tensors $s_{a, \inf}\otimes 1\in M_\inf(S)^\otimes\otimes_{A_{\inf}(S)} W(S)$ are $\Gamma_R$-invariant,
 we see that 
 \[
 s_{a, \inf}\otimes 1\in (M^\otimes\otimes_{\hat W(R) } W(S))^{\Gamma_R}=M^\otimes\otimes_{\hat W(R)} (W(S))^{\Gamma_R}.
 \]
 By Theorem \ref{FaltingsAlmost}, $(W(S))^{\Gamma_R}=W(R)$. Therefore, 
 \begin{equation}\label{combEqTen}
 s_{a, \inf }\otimes 1\in M^\otimes\otimes_{\hat W(R)} W(R).
 \end{equation}
 \end{para}

 In fact, we also have:
 
 \begin{prop}\label{Propinfres}
a) The tensors $m_a:=s_{a, \inf }\otimes 1$  lie in $M^\otimes$. 

b) The identity 
 \begin{equation}\label{identityxi}
 \ti x^*(m_a)=s_{a, \scriptscriptstyle{\rm D}, \ti x} 
 \end{equation}
holds in $ M^\otimes\otimes_{\hat W(R)}\hat W(\O_F)\cong M^\otimes_{\O_F}$.
\end{prop}
 
 In the above, $(M_{\O_F}, M_{1, \O_F}, F_{1, \O_F})$ is the Dieudonn\'e display over $\O_F$ associated
 by \S\ref{BKdisplay} to the Galois representation on $\Lambda$ given by $\ti x^*\Lr(\iota)$, and 
 $s_{a, \scriptscriptstyle{\rm D}, \ti x} $
are the corresponding tensors.
 
 \begin{proof} Using the compatibility of the construction with pull-back along points $\ti x: R\to \O_F$ we first see that the identity (\ref{identityxi}) holds in the tensor product $ M^\otimes\otimes_{\hat W(R)}  W(\O_F)$. However, the right hand side $s_{a, \scriptscriptstyle{\rm D}, \ti x}$ lies in the subset $M^\otimes\otimes_{\hat W(R)}\hat W(\O_F)$, and, hence, so is the left hand side $\ti x^*(m_a)$. Proposition \ref{corInter2} now  implies   (a), and (b) also follows.
\end{proof}
\smallskip

 We continue with the proof of Theorem \ref{Exists}. The tensors $m_a\in M^\otimes$ allow us to define 
 \[
 \calP:=\underline{\rm Isom}_{m_a, s_a\otimes 1}(M, \Lambda\otimes_{\Z_p}\hat W(R)).
 \] 
By the above, $\xi^*\calP$ is isomorphic to the $\Gg$-torsor given in \S\ref{BKdisplay}. By  Corollary \ref{CorRaynaudGruson} for $A=\hat W(R)$, $\calP$ is a $\Gg$-torsor over $\hat W(R)$. By definition, we have $\calP(\iota)=M$. 
It remains to construct $q$ and $\Psi$.

The filtration $I_RM\subset M_1\subset M$ gives a filtration of $M/I_RM$:
\[
(0)\subset {\rm Fil}^1:=M_1/I_RM\subset {\rm Fil}^0:=M/I_RM.
\]
This induces a filtration ${\rm Fil}^{\otimes, 0}$ of $ (M/I_RM)^\otimes$ and we have
\[
m_a\otimes 1\in {\rm Fil}^{\otimes, 0}\subset (M/I_RM)^\otimes[1/p]
\]
since this is true at all $F$-valued points.
Hence, $q: \P_{\GL}\otimes_{\hat W(R)}R\to {\rm Gr}(d, \Lambda)$ restricted to $\P\otimes_{\hat W(R)}R\subset \P_{\GL}\otimes_{\hat W(R)}R$ lands in $X_\mu(G)$ on the generic fiber. Since $\M$ is the Zariski closure of $X_\mu(G)$ in ${\rm Gr}(d, \Lambda)_{\O_E}$, we obtain
\[
q: \P\otimes_{\hat W(R)}R\to \M.
\]
Recall that we use $q$ to define the $\Gg$-torsor $\calQ$. From the construction, we have a $\Gg$-equivariant closed immersion $\calQ\subset \calQ_{\GL}$.
Finally, let us give $\Psi$: We consider $\Psi_{\GL}:\calQ_{\GL}\xrightarrow{\sim} \calP_{\GL}$. We will check that this restricts to $\Psi: \calQ\xrightarrow{\sim}\calP$: For this, is enough to show that the map $\Psi=\Psi_{\GL}(\iota): \tilde M_1\xrightarrow{\sim} M$ given by $\Psi_{\GL}$ preserves the tensors, i.e. 
\begin{equation}\label{phieq}
\Psi(\phi^*(m_a))=m_a.
\end{equation} 
This follows as in the proof of Proposition \ref{depL} by observing that the tensors are preserved after pulling back by all $\ti x: R\to\O_F$: 
Indeed, we have $\Psi_{\O_F}(\phi^*s_{a, \scriptscriptstyle{\rm D}, \ti x})=s_{a, \scriptscriptstyle{\rm D}, \ti x} $. Since, by (\ref{identityxi}) we also have $\ti x^*m_a=s_{a, \scriptscriptstyle{\rm D}, \ti x}$, we conclude that $\ti x^*\Psi=\Psi_{\O_F}$ maps $\phi^*(\ti x^*m_a)$ to $\ti x^*m_a$. The identity (\ref{phieq}) now follows.

The above define the $(\Gg, \M)$-display $\calD_{\bar x}=(\P, q, \Psi)=(\P_{\bar x}, q_{\bar x}, \Psi_{\bar x})$. By its construction, $ \calD_{\bar x}$ satisfies (A1) and (A2). This completes the proof of Theorem \ref{Exists}. 
 \end{proof}

 \begin{para}
    In fact, the proof of the Theorem \ref{Exists} also gives:
  
  \begin{prop}
 There is an isomorphism of $\Gg$-torsors
\begin{equation}\label{combEq2}
\calP_{\bar x}\otimes_{\hat W(R)}W(S)\simeq  D_{\inf}(S)\otimes_{A_\inf(S)}W(S),
\end{equation}
which is also compatible with the Frobenius structures.\hfill$\square$
  \end{prop}
  
   \end{para}

\smallskip
 
 \subsection{Independence} We show that the notion of ``associated" does not depend on the choice of the (strongly integral) local Hodge embedding $\iota$. More precisely:
 
 \begin{prop}\label{indassociated}
If $\Lr$ and $\calD_{\bar x}$ are associated for $\iota$, i.e. satisfy (A1) and (A2) for $\iota$, they also satisfy (A1) and (A2) for any other (strongly integral) local Hodge embedding $\iota'$.  \end{prop}
 \begin{proof}

 Assume that $\Lr$ and $\calD_{\bar x}$ are associated for $\iota$. By the uniqueness part of Theorem \ref{Exists}, we can assume that 
 $\calD_{\bar x}$ is obtained from $\Lr$ and $\iota$ by the construction in its proof. 
 We will use the notations of \S\ref{311}, \S\ref{332}:
 In particular, $R=\hat R_{\bar x}$ and $\sG=\sG(\bar x)$ is the $p$-divisible group over $R$ given by $\calD_{\bar x}(\iota)$.
 By \cite{LauGalois}, the Tate module $T:=T_p(\mathscr G)(\bar R)$ 
 can be identified with the kernel of 
 \[
 F_{1 }-1: \hat M_1 \to \hat M =M \otimes_{\hat W(R)}\hat W({\bar R}).
 \]
Here, we denote abusively by $\hat W(\bar R)$ the $p$-adic completion of 
 \[
 \varinjlim_{L/F(R)}\hat W(R_L)
 \]
 where $L$ runs over finite extensions of $F(R)$ in $\overline {F(R)}$, and $R_L$ is the normalization of $R$ in $L$. There is a natural surjective homomorphism $\hat W(\bar R)\to \bar R^\wedge$. 
 In the above, we set
 \[
 \hat M_1={\rm ker}(\hat M\to (M/M_1)\otimes_R \bar R^\wedge).
 \]
 The isomorphism
 \[
 {\rm per}: T\xrightarrow{\sim} \ker( F_{1 }-1: \hat M_1 \to \hat M )
 \]
 induces the comparison homomorphism
\[
T\otimes_{\Z_p}\hat W(\bar R)\to  M \otimes_{\hat W(R)}\hat W({\bar R}).
 \] 
 
 Let us  consider a second local Hodge embedding $\iota':\Gg\hookrightarrow \GL(\Lambda')$ which realizes $\Gg$ as the stabilizer of tensors $(s'_b)$. 
 
 As before, we have a $p$-divisible group $\mathscr G'=\mathscr G'(\bar x)$ over $R=\hat R_{\bar x}$ given by $\calD_{\bar x}(\iota')=(M', M'_1, F'_1)$. Denote its Tate module by  $T(\mathscr G')$. To show (A1) for $\iota'$, we have to give an isomorphism of $T(\mathscr G')$ with $T':=\Lr(\iota')=\Lambda'$. 
 
 Recall (\S\ref{536}), that we have a $\phi$-invariant isomorphism
\[
\Lambda \otimes_{\Z_p}A_{\rm inf}(S)[1/ \mu]    \simeq M_\inf(S)[1/ \mu]
\]
that sends the tensors $s_{a}\otimes 1$ to $s_{a, \inf}$. By a standard Tannakian argument
we see that this gives an isomorphism 
\begin{equation}\label{tannakaComp}
\Lambda'\otimes_{\Z_p}A_{\rm inf}(S)[1/\mu]\simeq M'_\inf(S)[1/\mu]
\end{equation}
where $M'_\inf(S)=D_{\inf}(S)(\iota')$ is obtained from the $\Gg$-torsor $D_{\inf}(S)$. 

 Since $\iota'$ is also a local Hodge embedding, we can see using
 \cite[Theorem 17.5.2]{Schber}, that $M'_\inf(S)$ is the BKF module $M(\mathscr H)$ of some $p$-divisible group $\mathscr H$ over $S$, so 
 \[
 M'_\inf(S)\simeq M(\mathscr H).
 \]
 
 \begin{lemma}\label{twoBTs}
For all $\ti x: S\to \O$ obtained from $\ti x: \ti R\to \bar\O_E$, we have:
 
 a)  $\ti x^*\mathscr H\simeq \ti x^* \sG'$,
 
 b) $T'\simeq T_p(\ti x^*\sG')$.
 \end{lemma}
 
 \begin{proof} Consider the Breuil-Kisin $\Gg$-module $(\calP_{\rm BK}, \phi_{\calP_{\rm BK}})$ attached to $\ti x^*\Lr$ by \S\ref{BKtorsor}. Then, $(\calP_{\rm BK}(\iota'), \phi_{\calP_{\rm BK}}(\iota'))$ gives a  ``classical'' Breuil-Kisin module which corresponds to a $p$-divisible group $\sG'_{\ti x}$ over $\O_F$. 
 The construction of the Breuil-Kisin $\Gg$-module implies that $T'\simeq \ti x^*\Lr(\iota')$ is identified with the Tate module $T_p(\sG'_{\ti x})$. By the compatibility (\S\ref{BKFcompatible}) of the constructions
 in \S\ref{BKF} and \S\ref{BKdisplay},  Proposition \ref{Propinfres}, and the fact \cite[Theorem 17.5.2]{Schber} that the  functor $M(-)$
 gives an equivalence of categories between $p$-divisible groups over $\O$ and (suitable) BKF modules over $A_\inf(\O)$, the base change of $\sG'_{\ti x}$ to $\O$ is isomorphic to both 
$ \ti x^*\mathscr H$ and $\ti x^* \sG'$. This gives (a). In fact, since the Tate module of $\sG'_{\ti x}$ is identified with $T'$, we then obtain (b).
 \end{proof}

 \smallskip
  
From (\ref{tannakaComp}), we obtain an injection
\[
c': \Lambda'=T'\hookrightarrow M'_\inf(S)[1/\mu].
\]
\begin{lemma}\label{545}
The map $c'$ gives an isomorphism
\[
   c': T'\xrightarrow{\ \sim   } M'_\inf(S)^{\phi_{M'}=1} ={M(\mathscr H)}^{\phi_{M(\mathscr H)}=1}
  \]
   which identifies $T'$ with the Tate module of $\mathscr H$.
\end{lemma}

\begin{proof}
For all $\ti x: S\to \O$ given by $\ti x: \ti R\to \bar\O_E$, we consider the composition
\[
T'\hookrightarrow  M'_\inf(S)[1/\mu] \to  M'_\inf(\O)[1/\mu]
\]
where the second map is given by pull-back along $A_\inf(S)\to A_\inf(\O)$ and functoriality.
From Lemma \ref{twoBTs} and its proof, we see that this composition is identified with the comparison isomorphism for $\ti x^*\mathscr H$ and so its image is contained in $M'_\inf(\O)$, in fact in $M'_\inf(\O)^{\phi_{M'^\vee}=1}$. It follows from Corollary \ref{corinf} (a) that the image of $c'$ is contained in 
$M'_\inf(S)$. 
Hence, we obtain:
\[
c':T'= \Lambda'\hookrightarrow M'_\inf(S)^{\phi_{M'}=1}.
\] 
Now, for all such $\ti x: S\to \O$, consider   
\[
T'\hookrightarrow M'_\inf(S)^{\phi_{M'}=1}\to M'_\inf(\O)^{\phi_{M'}=1}.
\]
As above, this composition is identified with the comparison map 
 for $\ti x^*\mathscr H$ and is therefore an isomorphism.  However,    
 \[
 (M'_\inf(S))^{\phi_{M'}=1}\simeq (M(\mathscr H))^{\phi_{M(\mathscr H)}=1}
 \]
 is the Tate module of $\mathscr H$, a finite free $\Z_p$-module  of rank equal to
  ${\rm rank}_{\Z_p}(T')$. Therefore 
  \[
   c': T'\xrightarrow{\ \sim   } M'_\inf(S)^{\phi_{M'}=1} ={M(\mathscr H)}^{\phi_{M(\mathscr H)}=1}
   \] 
   is an isomorphism and it identifies $T'$ with the Tate module of $\mathscr H$ as desired.
   \end{proof}
   \smallskip
   
 \quash{  The lemma and duality now gives that we also have
 \begin{equation}\label{twedgeeq}
  T'^\vee(1)\simeq {(M'_\inf(S)\{1\})}^{\phi_{M'\{1\}}=1}={M(\mathscr H^*)}^{\phi_{M(\mathscr H^*)}=1}
 \end{equation}
with both sides identifying with the Tate module of $\mathscr H^*$.
}

\begin{lemma}
The natural homomorphism
 \[
g: A_\inf(S)\xrightarrow{\theta_\infty} W(S)=W(\ti R^\wedge)\to W(\bar R^\wedge)
 \]
  factors through $ \hat W(\bar R)$ as a composition
 \[
 A_\inf(S)\hookrightarrow A_{\rm cris}(S)\to \hat W(\bar R)\to W(\bar R^\wedge).
  \]
\end{lemma}

\begin{proof} The diagram
\[
\begin{matrix}
 A_\inf(S) & \xrightarrow{\theta_\infty}& W(S) \\
\downarrow&&\downarrow \\
 A_\inf(\bar R^\wedge)&\xrightarrow{\theta_\infty} &W(\bar R^\wedge) 
 \end{matrix}
 \]
with vertical arrows given by $S\to \bar R^\wedge$, is commutative. Hence, the composition $g$ is  equal to
\[
A_\inf(S)\to A_\inf(\bar R^\wedge)\xrightarrow{\theta_\infty} W(\bar R^\wedge).
\]
We want to show $\theta_\infty: A_\inf(\bar R^\wedge)\xrightarrow{\ } W(\bar R^\wedge)$ factors through $\hat W(\bar R)$. 
We can argue as in the proof of \cite[Lemma 6.1]{LauGalois}: Note that, for each $n\geq 1$, 
all elements $a$ of the kernel of $\hat W(\bar R)/p^n\to \bar R^\wedge/p\bar R^\wedge$ satisfy $a^{p^n}=0$. Therefore, by the universal property of the Witt vectors (e.g. \cite[Lemma 1.4]{Lau}), this gives
\[
A_\inf(\bar R^\wedge)=W((\bar R^\wedge)^\flat )\to \hat W(\bar R)\to \bar R^\wedge/p\bar R^\wedge.
\]
This lifts  $\theta_{\bar R^\wedge} : A_\inf(\bar R^\wedge) \to \bar R^\wedge$. Now since $\hat W(\bar R)\to \bar R^\wedge$ is a divided power extension of $p$-adic rings, the map $\theta_{\bar R^\wedge}$ factors
\[
A_\inf(\bar R^\wedge)\to A_{\rm cris}(\bar R^\wedge)\to \hat W(\bar R)\to \bar R^\wedge
\]
 and using this we can conclude the proof.
  \end{proof}
  \smallskip
  
As in (\ref{combEq}), (\ref{combEq2}), we can use the above lemma to obtain an isomorphism
 \begin{equation}\label{combEq3}
 M'_\inf(S)\otimes_{A_{\inf}(S)}\hat W(\bar R)\simeq  M'\otimes_{\hat W(R)}\hat W(\bar R)
 \end{equation}
 which respects the Frobenius and Galois structures. This combined with the above  gives
 \[
  T'\simeq {M'_\inf(S)}^{\phi_{M'}=1} \xrightarrow{ \ } \hat M'^{F_1=1}_1\simeq T(\mathscr G')
 \]
 where both source and target are finite free $\Z_p$-modules of the same rank as that of $T'$.
 By pulling back via all  $\ti x: \bar R\to \bar \O_E$ and using Lemma \ref{twoBTs}, we see that this map is an isomorphism. Therefore, we have
 \[
 T'\simeq T(\mathscr G')
 \]
 which shows (A1). 
For (A2), it is enough to show that the tensors $m'_b$ in $\calD_{\bar x}(\iota')^\otimes$ 
restrict via $\ti x: R\to \O_F$ to the corresponding tensors $s'_{b, \scriptscriptstyle{\rm D}, \ti x}$ in $\calD_{\rho(x)}(\iota')^\otimes$. This now follows from the above 
construction  and (\ref{combEq3}). \end{proof}

\medskip

  \section{Canonical integral models}\label{sShimura}

 \subsection{}
 We now consider  Shimura varieties and their arithmetic models. Under certain assumptions, we give a definition of a ``canonical'' integral model.
 
 \begin{para}\label{31}
Let $\eG$ be a connected reductive group over $\mathbb Q$ and $X$ a conjugacy class of maps 
 of algebraic groups over $\RR$ 
$$ h: \mathbb S = \Res_{\C/\RR} \GG_m \rightarrow \eG_{\RR},$$ 
such that $(\eG,X)$ is a Shimura datum (\cite{DeligneCorvallis} \S 2.1.)

For any $\C$-algebra $R,$ we have $R\otimes_{\RR}\C = R \times c^*(R)$ where $c$ 
denotes complex conjugation, and we denote  
by $\mu=\mu_h$ the cocharacter given on $R$-points by 
$
 R^\times \rightarrow (R \times c^*(R))^{\times} = (R\otimes_{\RR}\C)^\times = \mathbb S(R) \overset h \rightarrow \eG_{\C}(R).
 $

Let $\AA_f$ denote the finite adeles over $\Q,$ and $\AA^p_f \subset \AA_f$ the subgroup 
of adeles with trivial component at $p.$ Let $\eK = \eK_p\eK^p \subset \eG(\AA_f)$ 
where $\eK_p \subset \eG(\Q_p),$ and $\eK^p \subset \eG(\AA^p_f)$ are compact open subgroups.   

If $\eK^p$ is sufficiently small then the Shimura variety
$$
\Sh_{\eK}(\eG,X)_{\C} = \eG(\Q) \backslash X \times \eG(\AA_f)/\eK 
$$
has a natural structure of an algebraic variety over $\mathbb C$.
This has a canonical model  $\Sh_{\eK}(\eG,X)$ over the reflex field; a number field ${\sf E} = {\sf E}(\eG,X)$ which is the minimal field 
of definition of the conjugacy class of $\mu_h.$ (See, for example, \cite{MilneClay}.)
We will always assume in the following that $\eK^p$ is sufficiently small; in particular, the quotient above exists as an algebraic variety. We will also assume that the center $Z(G)$ has the
same $\Q$-split rank  as   $\R$-split rank. (This condition is automatic for Shimura varieties of Hodge type.)

Now choose a place $v$ of $\eE$ over $p$, given by an embedding $\bar\Q\to \bar\Q_p$.
 We denote by $E={\sf E}_v/\Q_p$ the local reflex field and by $\{\mu\}$ the 
 $G(\bar\Q_p)$-conjugacy class -which is defined over $E$- of the minuscule cocharacter $\mu_h$. 
We denote by $\O_{\eE, (v)}$ the localization of the ring of integers $\O_{\eE}$ at $v$.  

Let $\Gg$ be a parahoric group scheme over $\Z_p$ with generic fiber 
$G_{\Q_p}$ and take $\eK_p=\Gg(\Z_p)\subset G(\Q_p)$.

We consider the system 
of covers $\Sh_{\eK'}(G, X)\to \Sh_{\eK}(G, X)$ where $\eK'=\eK'_p\eK^p\subset \eK=\eK_p\eK^p$, with $\eK'_p$ running over all compact open subgroups of $\eK_p=\calG(\Z_p)$. Using the condition on the center of $G$, we see that this gives a pro-\'etale $\Gg(\Z_p)$-cover ${\L}_\eK$ over $\Sh_{\eK}(G, X)$ (eg. see \cite[III]{MilneAA}, \cite[Thm 5.2.6]{MilneClay}).

\end{para}

\begin{para}\label{612}

Assume $p>2$, $(G, \{\mu\})$ is of local Hodge type and $\Mloc=\Mloc(\Gg, \{\mu\})$ is a local model as in \cite[Conjecture 21.4.1]{Schber}
(see the discussion in paragraph \ref{LMtheory}).  Assume also that the pair $(\Gg, \Mloc)$ is of \emph{strongly integral} local Hodge type.

 Suppose that for all sufficiently small $\eK^p$ we have  $\OEv$-models $\SS_{\eK}=\SS_{\eK_p\eK^p}$ (schemes of finite type, separated, and flat over $\OEv$) of the Shimura variety $\Sh_{\eK}(G, X)$ which are normal. We consider the conditions:
 
 \begin{itemize}
 
 \item[(1)] For $\eK'^p\subset \eK^p$, there are finite \'etale morphisms 
 \[
 \pi_{\eK'_p, \eK_p}: \SS_{\eK_p\eK'^p}\to \SS_{\eK_p\eK'^p}
 \]
  which extend the natural $\Sh_{\eK_p\eK'^p}(G, X)\to \Sh_{\eK_p\eK^p}(G, X)$.

 \item[(2)] The scheme $\SS_{\eK_p}=\varprojlim_{\eK^p}\SS_{\eK_p\eK^p}$ satisfies the ``extension property" for dvrs of mixed characteristic $(0, p)$:
 \[
 \SS_{\eK_p}(R[1/p])=\SS_{\eK_p}(R),
 \]
 for any such dvr $R$.
 
 \item[(3)] The $p$-adic formal schemes $\widehat \SS_{\eK}=\varprojlim_n \SS_\eK\otimes_{\O_{\eE, (v)}}\O_{\eE, (v)}/(p)^n$ support    locally universal {$(\Gg, \Mloc )$-displays} 
$
 \Dd_\eK
$
 which are associated with $\Lr_{\eK}$. We  ask that these are compatible for varying $\eK^p$, i.e. that there are  compatible isomorphisms
 \[
 \pi_{\eK'_p, \eK_p}^*\Dd_{\eK }\simeq \Dd_{\eK'}.
 \]
  \end{itemize}

 Instead of (3) we can also consider the condition:
 
 \begin{itemize}
 \item[(3*)]  The schemes $\SS_{\eK}$ support locally universal associated systems
 \[
 \widehat\Dd_\eK:=({\L}_\eK, \{\calD_{\bar x}\}_{\bar x\in \SS_\eK(k)}),
 \]
 where $\calD_{\bar x}$ are Dieudonn\'e $(\Gg, \Mloc)$-displays.  
  \end{itemize}

Note that (3) implies (3*);  this follows from the definitions.
 
 Theorem \ref{charThmA} below makes the following definition reasonable.

\begin{Definition}\label{canonical613}
A projective system of $\OEv$-models $\SS_{\eK} $ of the Shimura varieties $\Sh_{\eK}(G, X)$,  for $\eK=\eK_p\eK^p$ with $\eK_p$ fixed as above, is \emph{canonical}, if the models are  normal and satisfy the conditions (1), (2), (3) above.
\end{Definition}

 We conjecture that, under the hypotheses above, such canonical models always exist.
  In the next section, we show this for Shimura varieties of Hodge type  at tamely ramified primes.
 
 \begin{Remark}
 {\rm We could also consider a notion of a ``$\calM$-canonical model", where $(\Gg, \calM)$ is a more general integral local 
 Shimura pair, i.e. with $\calM$ not $\Mloc$. However, it is not clear if such added generality is very useful here.
  }
  \end{Remark}
   
  \begin{Remark}\label{LMD}
  {\rm Property (3) implies the existence of a \emph{local model diagram}:
  \begin{equation}\label{LMD}
\begin{gathered}
   \xymatrix{
	     &\widetilde{\widehat\SS_{\eK}}  \ar[dl]_-{\text{$\pi_\eK$}} \ar[dr]^-{\text{$q_\eK$}}\\
	   {\widehat\SS}_{\eK}   & & \ \Mloc \ ,
	}
\end{gathered}
\end{equation}
where $\pi_\eK$ is a $\Gg$-torsor and $q_\eK$ is $\Gg$-equivariant and smooth.

Indeed, let $\calD_\eK=(\calP_\eK, q_\eK, \Psi_\eK)$ be the ``universal" $(\Gg, \Mloc)$-display over $\widehat\SS_\eK$
as in (3). We set
\[
\widetilde{\widehat\SS_{\eK}}:=\calP_\eK\otimes_{W(\O_{\widehat\SS_{\eK}})}\O_{\widehat\SS_{\eK}}.
\]
This is a $\Gg$-torsor over $\widehat\SS_{\eK}$ and gives $\pi_\eK$. The morphism $q_\eK$ is obtained directly from the 
display datum. The smoothness of $q_\eK$ follows from the local universality condition by Proposition \ref{locally universalProp2}.

  }
  \end{Remark}

 \quash{Note that, in our definition of canonical, we first make a choice for the $\O_E$-scheme $\M$. However, it is reasonable 
 to ask, beforehand, for a canonical choice of $\M$ which depends only on $\Gg$ and $\{\mu\}$ and use this in the definition. This canonical choice should be the ``local model''. At tame primes of parahoric reduction the local models $\Mloc$ of \cite{PZ}, (slightly modified, as in \cite{HPR}), are uniquely determined by $\Gg$ and $\{\mu\}$
 \cite[Theorem 2.7, Remark 2.9]{HPR}; we use this choice in the next section.
 
 In the general case when $\Gg$ is parahoric, Scholze conjectures \cite[Conjecture 21.4.1]{Schber} that there should always be a canonical choice  
 $\M={\mathbb M}^{\rm loc}_{(\Gg, \mu)}$. He characterizes ${\mathbb M}^{\rm loc}_{(\Gg, \mu)}$ uniquely in terms of $(\Gg, \{\mu\})$ by giving
 its associated v-sheaf. We can then regard the integral models $\SS_{\eK}$ to be canonical only when they satisfy our conditions for this particular choice of $\M$. In any case, as it is shown in \cite{HPR}, $\Mloc={\mathbb M}^{\rm loc}_{(\Gg, \mu)}$ in all the cases we consider in \S\ref{sHodge}. In fact, assuming also the truth of the Conjecture in Remark \ref{conjEmb}, we can make sense of $(\Gg, \Mloc)$-displays 
 and give a notion of a canonical integral model more generally (i.e. without any tameness or local Hodge type hypothesis).
 }

   \begin{thm}\label{charThmA} Fix $\eK_p=\Gg(\Z_p)$ as above. Suppose that $\SS_\eK$, $\SS'_{\eK}$ are   $\OEv$-models of the Shimura variety ${\rm Sh}_{\eK}(G, X)$ for $\eK=\eK_p\eK^p$ that satisfy (1), (2) and (3*). Then there
 are isomorphisms $\SS_\eK\simeq \SS'_{\eK}$
giving the identity on the generic fibers and which are  compatible with the data in (1) and (3*). 
\end{thm}

Since condition (3) implies (3*), this immediately gives:
 
 \begin{cor}\label{charThm}
 Fix $\eK_p=\Gg(\Z_p)$ as above. Suppose that $\SS_\eK$, $\SS'_{\eK}$ are canonical $\OEv$-models of the Shimura variety ${\rm Sh}_{\eK}(G, X)$ for $\eK=\eK_p\eK^p$. Then there are isomorphisms $\SS_\eK\simeq \SS'_{\eK}$
giving the identity on the generic fibers and which are  compatible with the data in (1). \endproof
 \end{cor}

 \begin{proof} Let us denote by $\SS''_{\eK}$ the normalization of the Zariski closure of the diagonal embedding of ${\rm Sh}_{\eK}(G, X)$ in $\SS_{\eK}\times_{\OEv}\SS'_{\eK}$. This is  a third $\OEv$-model of the Shimura variety ${\rm Sh}_{\eK}(G, X)$ which is also normal. We can easily see that $\SS''_{\eK}$, for varying $\eK^p$, come equipped with data as in (1) and that (2) is satisfied. Denote by
 \[
 \pi_{\eK}: \SS''_{\eK}\to \SS_{\eK}, \quad \pi'_{\eK}: \SS''_{\eK}\to \SS'_{\eK}
 \]
 the morphisms induced by the projections. Both of these morphisms are the identity on the generic fiber and so they are birational. Using condition (3*), we see that by Proposition \ref{MainProp}, $\pi_{\eK}$ and $\pi'_{\eK}$ give isomorphisms between the completions of the strict Henselizations at geometric closed points of the special fiber. It follows that  the fibers of $\pi_{\eK}$ and of $\pi'_{\eK}$ over all such points are zero-dimensional. Hence, $\pi_{\eK}$ and  $\pi'_{\eK}$
 are quasi-finite. The desired result now quickly follow from this, Zariski's main theorem and the following:

 \begin{prop}\label{Proper}
 The morphisms $\pi_{\eK}$ and $\pi_{\eK}'$ are proper.
\end{prop}

\begin{proof} It is enough to prove that $\pi_{\eK}$ is proper, the properness of $\pi_{\eK}'$ then given by symmetry. We can also base change to the completion $\O_{\breve E}$ of the maximal unramified extension of $\OEv$; for simplicity we will omit this base change from the notation. We apply the Nagata compactification theorem (\cite[Thm. 4.1]{Conrad}) to $\pi_{\eK}$.
This provides a proper morphism $\bar\pi_{\eK}: \calT\to \SS_{\eK}$
and an open immersion $j: \SS''_{\eK}\hookrightarrow \calT$ with $\pi_{\eK}=\bar\pi_{\eK}\cdot j$. By replacing $\calT$ by the scheme theoretic closure of $j$, we can assume that $j(\SS''_{\eK})$ is dense in $\calT$. Since $\pi_{\eK}[1/p]$ is an isomorphism and hence proper, $j[1/p]$ is also proper. Hence, $j[1/p]$ is an isomorphism as a proper open immersion with dense image.  Since $\calT$ is the closure of its generic fiber by construction, it follows that $\calT$ is flat over $\O_E$ and that
the ``boundary", $\calT-j(\SS''_{\eK})$, if non-empty, is supported on the special fiber
of $\calT\to \Spec(\O_E)$.

If  $\calT-j(\SS''_{\eK})\neq \emptyset$, there is a $k$-valued point $\bar t$ of $\calT-j(\SS''_{\eK})$. By flatness,
$\bar t$ lifts to $\tilde t\in \calT(\O_F)$, for some finite extension $F/\breve E$. Set $t=x$ for the corresponding $F$-valued point of the Shimura variety $\calT[1/p]=\SS_{\eK}[1/p]=\SS''_{\eK}[1/p]$. This extends to $\tilde x:=\bar\pi_{\eK}(\tilde t) \in \SS_{\eK}(\O_F)$. Since $\O_F$ is strictly henselian, the point $\tilde x$ lifts to a point
\[
\ti z\in \SS_{\eK_p}(\O_F)=\varprojlim\nolimits_{\eK^p}\SS_{\eK_p\eK^p}(\O_F).
\]
 By the dvr extension property for $\SS''_{\eK_p}$, this also gives a point $\ti z''
 \in \SS''_{\eK_p}(\O_F)$. This maps to a point $\tilde x''\in \SS''_{\eK}(\O_F)$ which agrees
 with $x\in \SS_{\eK}(F)$ on the generic fiber. Since $\bar\pi_{\eK}: \calT\to \SS_{\eK}$ is separated, this implies that $\bar t$ lies on $j(\SS''_{\eK})$, which is a contradiction. We conclude that $j$ is an isomorphism and so $\pi_{\eK}$ is proper. 
  \end{proof}
\end{proof}

 \end{para}

\section{Shimura varieties at tame parahoric primes}\label{sHodge}

\subsection{}
We now concentrate our attention to Shimura varieties of Hodge type at tame primes where the level is parahoric (\cite{KP}).

\begin{para}
Fix a $\Q$-vector space $V$ with a perfect alternating pairing $\psi.$ 
For any $\Q$-algebra $R,$ we write $V_R = V\otimes_{\Q}R.$ 
Let $\GSp = \GSp(V,\psi)$ be the corresponding group of symplectic similitudes, and let $S^{\pm}$ be the Siegel 
double space, defined as the set of maps $h: \mathbb S \rightarrow \GSp_{\RR}$ such that 
\begin{enumerate}
\item The $\C^\times$-action on $V_{\RR}$ gives rise to a Hodge structure
\[
V_{\C} \iso V^{-1,0} \oplus V^{0,-1}
\]  of type $(-1,0), (0,-1)$. 
\item $(x,y) \mapsto \psi(x, h(i)y)$ is (positive or negative) definite on $V_{\RR}.$
\end{enumerate}
\smallskip
\end{para}

\begin{para}\label{assumptions}
 
 Let $(G, X)$ be a Shimura datum and $\eK=\eK_p\eK^p\subset G(\AA_f)$ with $\eK_p\subset G(\Q_p)$ and $\eK^p\subset G(\AA_f^p)$ as above, where $p$ is an odd prime. We assume:
 
 \begin{itemize}
 \item[(1)] $(G, X)$ is of Hodge type: There is a symplectic faithful representation 
 $\rho: G\hookrightarrow {\rm {GSp}}(V, \psi)$ inducing an embedding of Shimura data
 \[
 (G, X)\hookrightarrow ({\rm {GSp}}(V, \psi), S^{\pm}).
 \]
 
 \item[(2)] $G$ splits over a tamely ramified extension of $\Q_p$.
  
 \item[(3)] $\eK_p=\calG(\Z_p)$ is a parahoric stabilizer, so $\calG$ is the Bruhat-Tits stabilizer group scheme $\calG_x$ of a point $x$ in the extended Bruhat-Tits building of $G(\Q_p)$ and $\calG$ is connected, \emph{i.e.} we have
 $\calG=\calG_x=\calG_x^\circ$.
 
 \item[(4)]   $p\nmid \pi_1(G_{\rm der}(\bar\Q_p))$. 
 
 \end{itemize}

 We now fix a place $v$ of the reflex field ${\sf E}$ over $p$ and let $E={\sf E}_v/\Q_p$ and $\{\mu\}$ be as in \S\ref{sShimura} above.
  Associated with $(\Gg, \{\mu\})$ and $x$, we have the local model 
 \[
 \Mloc=\Mloc(\Gg, \{\mu\})=\Mloc_{x}(G, \{\mu\}).
 \]
  (Under the current assumptions, we can appeal to \cite{PZ} for the construction of $\Mloc$.
This satisfies Scholze's characterization by \cite[Theorem 2.15]{HPR}.)

 In  \cite[2.3.1, 2.3.15, 2.3.16]{KP}, it is shown that under the assumptions (1)-(4) above, there is a (possibly different) Hodge embedding 
 \[
 \iota: (G, X)\hookrightarrow ({\rm {GSp}}(V, \psi), S^\pm)
 \] and a $\Z_p$-lattice $\Lambda\subset V_{\Q_p}$ 
 such that $\Lambda\subset \Lambda^\vee$ and
 
 \begin{itemize}
 \item[a)]  There is a group scheme homomorphism which is a closed immersion
 \[
 \iota: \calG\hookrightarrow {\rm GL}(\Lambda ),
 \]
 such that $\iota(\Gg)$ contains the scalars $\Gm$,  and which extends 
\[
  G_{\Q_p}\hookrightarrow {\rm {GSp}}(V_{\Q_p}, \psi_{\Q_p})\subset {\rm {GL}}(V_{\Q_p}).
\] 

\item[b)]  There is a corresponding equivariant closed immersion
 \[
 \iota_*: \Mloc\hookrightarrow  {\rm {Gr}}(g, \Lambda)_{\O_E}.
 \] 

 \item[](So $\iota$ is a strongly integral local Hodge embedding for $(\Gg, \Mloc)$.)
  \end{itemize}
 
 Here, $\dim_{\Q_p}(V)=2g$ and ${\rm {Gr}}(g, \Lambda)$ is the Grassmannian over $\Z_p$.    
 \end{para}
 
 \begin{para}
  Let $V_{\Z_{(p)}} = \Lambda \cap V,$ and fix a $\Z$-lattice $V_{\Z} \subset V$ such that $V_{\Z}\otimes_{\Z}\Z_{(p)} = V_{\Z_{(p)}}$ 
and $V_{\Z} \subset V_{\Z}^\vee.$  
Consider the Zariski closure $G_{\Z_{(p)}}$ of $\eG$ in $\GL(V_{\Z_{(p)}})$; then $G_{\Z_{(p)}}\otimes_{\Z_{(p)}}\Z_p\cong \Gg$. Fix a finite set of  tensors $(s_{a}) \subset V_{\Z_{(p)}}^\otimes$ 
whose stabilizer is $G_{\Z_{(p)}}.$ Such  a set exists by \cite[Lemma 1.3.2]{KisinJAMS} and \cite{DeligneLetter}.

Set $\eK_p = \calG(\Z_p),$ and $\eK^\flat_p = \mathrm{GSp}(V_{\Q_p})\cap \GL(\Lambda).$ We set $\eK = \eK_p\eK^p$ and similarly for $\eK^\flat.$ 
By  \cite[Lemma 2.1.2]{KisinJAMS}, for any compact open subgroup $\eK^p \subset G(\AA^p_f)$ there exists 
$\eK^{\flat p} \subset \GSp(\AA^p_f)$ such that $\iota$ induces an embedding over $\sf E$
$$
\Sh_{\eK}(\eG,X) \hookrightarrow \Sh_{\eK^\flat}(\GSp(V, \psi), S^{\pm})\otimes_\Q{\sf E}.
$$

The choice of lattice $V_{\Z}$ gives an interpretation of the Shimura variety $\Sh_{\eK^\flat}(\GSp, S^{\pm})$ as a moduli scheme of 
polarized abelian varieties with $\eK^{\flat p}$-level structure, and hence an integral model $\calA_{\eK^\flat}=\SSh_{\eK^
\flat}(\GSp, S^{\pm})$ over $\Z_{(p)}$ 
 (see \cite{KisinJAMS}, \cite{KP}).  
 
We denote by $\SSh^-_{\eK}(\eG,X)$ the (reduced) closure of $\Sh_{\eK}(\eG,X)$ in the $\OEv$-scheme 
$\SSh_{\eK^\flat}(\GSp, S^{\pm})\otimes_{\Z_{(p)}}\OEv,$ and by $\SSh_{\eK}(\eG,X),$ the normalization of the closure $\SSh_{\eK}(\eG,X)^-.$
For simplicity, we set 
\[
\SSh_{\eK}:=\SSh_{\eK}(\eG,X)
\]
when there is no danger of confusion.  

\begin{thm}\label{exAssSystem}
 Assume that $p$ is odd and that the Shimura data $(G, X)$ and the level subgroup $\eK$ satisfy the assumptions (1)-(4) of \S \ref{assumptions}. Then, the $\OEv$-models $\SS_{\eK}(G, X)$ support locally universal associated systems 
 \[
 \Dd_\eK=({\L}_\eK, \{\calD_{\bar x}\}_{\bar x\in \SS_{\eK}(k)}),
 \]
  where $\calD_{\bar x}$ are Dieudonn\'e $(\Gg, \Mloc)$-displays.
\end{thm}

\begin{proof}  Recall the pro-\'etale $\Gg(\Z_p)$-cover ${\L}_\eK$ over $\Sh_{\eK}(G, X)$ given as in \S \ref{31} above.  Let $h: A \rightarrow  \SS_{\eK}$ denote the restriction of the universal abelian scheme via $
\SS_{\eK}\rightarrow \SSh_{\eK^\flat}(\GSp, S^{\pm}) 
$. Then the $\Z_p$-local system $\Lr_\eK(\iota)$ is isomorphic to the local system given by the
  Tate module of the $p$-divisible group $A[p^\infty]$ of the universal abelian scheme over $\SS_{\eK}$. The tensors
$s_a\in \Lambda^\otimes$ give corresponding global sections $s_{a, \et}$ of $\Lr_\eK(\iota)^\otimes$ over 
$\Sh_{\eK}(G, X)$.  Theorem \ref{Exists} implies that ${\L}_\eK$ extends to an associated system $({\L}_\eK, \{\calD_{\bar x}\}_{\bar x\in \SS_{\eK}(k)})$, where $\calD_{\bar x}$ are Dieudonn\'e $(\Gg, \Mloc)$-displays.
It remains to show:

\begin{prop}\label{locally universalProp}
For every $\bar x\in  \SS_{\eK}(k)$, the Dieudonn\'e $(\Gg, \Mloc)$-display $\calD_{\bar x}$ over 
$R=\hat\O_{\breve \SS_\eK, \bar x}$ is locally universal.
\end{prop}

\begin{proof} Set $\calD_{\bar x}=(\calP, q, \Psi)$.
Choose a section $s$ of $\calP$ over $\hat W(R)$ which is rigid in the first order at $\frakm_R$.
Then the corresponding section $\Spec(\hat W(R))\to \calP\subset \calP_{\GL}$ is rigid in the first order for 
the $\GL$-display $\calD_{\bar x}(\iota)=(\P_{\GL}, q_{\GL}, \Psi_{\GL})$ induced by $\iota$ and $\calD_{\bar x}$. 
We have a morphism 
\[
q\cdot (s\otimes 1): \Spec(R)\to \Mloc\subset {\rm Gr}(g,\Lambda)_{\O_{\breve E}}.
\]
We also have the morphism
\[
i: \Spec(R)=\Spec(\hat\O_{\breve \SS_\eK, \bar x})\to  \calA_{\eK^\flat}\otimes_{\Z_p}\O_{\breve E}.
\]
induced by the Hodge embedding.  By \cite[Prop. 4.2.2]{KP} and its proof, $i$ is a closed immersion. 
Since $\calD_{\bar x}=(\calP, q, \Psi)$ is associated with $\Lr_\eK$ (for $\iota$), the $p$-divisible
group that corresponds to $\calD_{\bar x}(\iota)$ is the $p$-divisible group obtained by pulling back the (versal)
$p$-divisible group of the universal abelian scheme via $i$. By Proposition \ref{Psiconstant} we obtain
that the morphism $\Spec(R)\to \Mloc\subset {\rm Gr}(g,\Lambda)_{\O_E}$ 
induces a surjection on cotangent spaces.
 It follows that $\Spec(R)\to \Mloc$ also induces
a surjection 
\[
\hat\O_{{\breve {\rm M}}^{\rm loc}, \bar y}\to R=\hat\O_{\breve\SS_\eK, \bar x}
\]
where $\bar y=(q\cdot (s\otimes 1))(\bar x)$.
 This surjection between
complete local normal rings of the same dimension has to be an isomorphism. This completes the proof.
\end{proof}
\end{proof}

\end{para}

 By combining Theorems \ref{exAssSystem} and \ref{charThm} we now obtain:

\begin{thm}\label{indThm}
  Assume that $p$ is odd and that the Shimura data $(G, X)$ and the level subgroup $\eK$ satisfy the assumptions (1)-(4) of \S \ref{assumptions}. Suppose $v$ is a place of $\eE$ over $p$. Then the $\O_{\eE, (v)}$-scheme  $\SSh_{\eK}(G, X)$ of \cite{KP} is independent of the choices of Hodge embedding  $\rho: (G, X)\hookrightarrow ({\rm {GSp}}(V, \psi), S^\pm)$, lattice $V_{\Z}\subset V $ and tensors $(s_a)$, used in its construction. \hfill $\square$
\end{thm}

\subsection{} Finally, we show:

\begin{thm}\label{thmLast}
Assume that $p$ is odd and that the Shimura data $(G, X)$ and the level subgroup $\eK$ satisfy the assumptions (1)-(4) of \S \ref{assumptions}. Then, the $\OEv$-models $\SS_{\eK}(G, X)$ of \cite{KP} are canonical, in the sense of Definition \ref{canonical613}.
\end{thm}

\begin{proof}  We  already know that $\SS_\eK=\SS_{\eK}(G, X)$ supports a locally universal associated system by Theorem \ref{exAssSystem}. We need to ``upgrade" this and show there is also an associated $(\Gg, \Mloc)$-display $\sD_\eK$ as in Definition \ref{defass2}. 
Write $\widehat\SS_\eK$ for the formal scheme obtained as the 
$p$-adic completion of $\SS_\eK$. The Dieudonn\'e crystal $\DD_\eK:=\DD(A[p^\infty])(W(\O_{\widehat\SS_\eK}))$ of the universal $p$-divisible group over $\SS_\eK$ gives a $\GL$-display over $\widehat\SS_\eK$. By work of Hamacher and Kim \cite[3.3]{HamaKim}, there are Frobenius invariant tensors $s_{a, \rm univ}\in \DD_\eK^\otimes$ which have the following property: For every $\bar x\in \SS_\eK(k)$,  the base change isomorphism
\begin{equation}\label{lastbc}
\DD_\eK\otimes_{W(\O_{\widehat\SS_\eK})}W(\hat R_{\bar x})\simeq M_{\bar x} \otimes_{\hat W(\hat R_{\bar x})}W(\hat R_{\bar x})
\end{equation}
maps $s_{a, \rm univ}$ to $m_a\otimes 1$. Here, we write $\calD_{\bar x}(\iota)=(M_{\bar x}, M_{1, \bar x }, F_{1, \bar x })$
and we recall that $m_a\in M^{\otimes}_{\bar x}$  are the tensors which are associated with $s_{a,\et}\in \Lr_{\eK}(\iota)^\otimes$ and are given by the $\Gg$-torsor $\calP_{\bar x}$ of the 
Dieudonn\'e $(\Gg, \Mloc)$-display $\calD_{\bar x}=(\calP_{\bar x}, q_{\bar x}, \Psi_{\bar x})$
(see the proof of Theorem \ref{Exists}).
We can now use this to give a $(\Gg, \Mloc)$-display $\sD_\eK=(\calP_\eK, q_\eK, \Psi_\eK)$ over $\widehat\SS_\eK$ as follows:
First set
\[
\calP_\eK=\underline{\rm Isom}_{(s_{a,\rm univ}), (s_a\otimes 1)}( \DD_\eK, \Lambda\otimes_{\Z_p}W(\O_{\widehat\SS_\eK})).
\]
Consider an open affine formal subscheme ${\rm Spf}(R)\subset \widehat\SS_\eK$. Then $R$ satisfies condition (N). 
Since (\ref{lastbc}) above respects the tensors, $\calP_\eK\otimes_{W(\O_{\widehat\SS_\eK})}W(\hat R_{\bar x})\simeq \calP_{\bar x}$, for all $\bar x\in \Spec(R/p)(k)$. Hence, for example by Corollary \ref{CorRaynaudGruson}, $\calP_\eK\otimes_{W(\O_{\widehat\SS_\eK})}W(R)$ is a $\Gg$-torsor over $\Spec(W(R))$. Therefore, $\calP_\eK$ is also a $\Gg$-torsor. It remains to give $q_\eK$ and $\Psi_\eK$. Recall that, under our assumptions, \cite[Theorem 4.2.7]{KP} gives a (``local model") diagram
\begin{equation*}
   \xymatrix{
	     &\widetilde{\SS_{\eK}}  \ar[dl]_-{\text{$\pi_\eK$}} \ar[dr]^-{\text{$q_\eK$}}\\
	   {\SS}_{\eK}   & & \ \Mloc \ ,
	}
\end{equation*}
in which the left arrow is a $\Gg$-torsor and the right arrow is smooth and $\Gg$-equivariant. The $\Gg$-torsor
$\widetilde\SS_\eK\to \SS_\eK$ is given as   
\[
\widetilde\SS_\eK:=\underline{\rm Isom}_{(s_{a,\rm DR}), (s_a)}( {\rm H}^1_{\rm DR}(A)^\vee, \Lambda\otimes_{\Z_p}\O_{ \SS_\eK}).
\]
Since by \cite[Cor. 3.3.4]{HamaKim} the comparison 
\[
\DD_\eK\otimes_{W(\O_{\widehat\SS_\eK})}\O_{\widehat\SS_\eK}\cong 
\rH^1_{\rm DR}(A)^\vee
\]
takes $s_{a, \rm univ}\otimes 1 $ to $s_{a, \rm DR}$, we have
\[
\calP_\eK\otimes_{W(\O_{\widehat\SS_\eK})}\O_{\widehat\SS_\eK}\cong \widetilde\SS_\eK\xrightarrow{ q_\eK} \Mloc
\]
which gives the desired $q_\eK$. Finally, we can give $\Psi_\eK$ using the Frobenius structure on $\DD_\eK$ following
the dictionary in \S \ref{compaZink2}. By similar arguments as above, this respects the tensors and so it gives an isomorphism
of $\Gg$-torsors. Then
$\sD_\eK$ gives the desired $(\Gg, \Mloc)$-display which satisfies the requirements of \S \ref{612}. The result follows.
\end{proof}

\begin{Remark} {\rm We expect that the above results (Theorems \ref{indThm}, \ref{thmLast}), can be extended so that they also apply to the integral models $\SS_{\eK}(G, X)$ constructed in \cite{KZhou}.
In the set-up of \cite{KZhou}, assumption (2) of \S\ref{assumptions} is weakened to allow for some wildly ramified groups with $G_{\rm ad}\otimes_\Q\Q_p\simeq \prod_{i=1}^m {\rm Res}_{F_i/\Q_p} (H_i)$, where each $H_i$ splits over a tamely ramified extension of $F_i$. 
}  
\end{Remark}

\end{document}